\documentclass[11pt,a4paper,reqno]{amsart}
\usepackage{mathtools}
\usepackage{amsfonts}
\usepackage{amsmath,amssymb,amsthm,amsxtra}
\usepackage{booktabs}
\usepackage{array}
\usepackage{multirow}
\usepackage{mathrsfs}
\usepackage{float}
\usepackage[colorlinks, linkcolor=blue!72,anchorcolor=orange,
    citecolor=red,urlcolor=Emerald, bookmarksopen,bookmarksdepth=2]{hyperref}
\usepackage[usenames,dvipsnames]{xcolor}
\usepackage{enumitem}
\usepackage{geometry,array}
\usepackage{graphicx}
\usepackage{subfigure}

\usepackage{bookmark}
\usepackage{tabularx} 
\usepackage{tikz}
\usepackage{tikz-cd}
\usepackage{subcaption}
\usepackage{url}

\usetikzlibrary{matrix,positioning,decorations.markings,arrows,decorations.pathmorphing, 
    backgrounds,fit,positioning,shapes.symbols,chains,shadings,fadings,calc}
\tikzset{->-/.style={decoration={  markings,  mark=at position #1 with
    {\arrow{>}}},postaction={decorate}}}
\tikzset{-<-/.style={decoration={  markings,  mark=at position #1 with
    {\arrow{<}}},postaction={decorate}}}
\usepackage{cleveref}
\usepackage{verbatim}
\usepackage[all]{xy}

\usepackage{pdflscape}
\usepackage{booktabs}

\usepackage{aliascnt}

\theoremstyle{plain}
\newtheorem{theorem}{Theorem}[section]

\newaliascnt{lemma}{theorem}
\newtheorem{lemma}[lemma]{Lemma}
\aliascntresetthe{lemma}

\newaliascnt{corollary}{theorem}
\newtheorem{corollary}[corollary]{Corollary}
\aliascntresetthe{corollary}

\newaliascnt{proposition}{theorem}
\newtheorem{proposition}[proposition]{Proposition}
\aliascntresetthe{proposition}

\newaliascnt{Fact}{theorem}
\newtheorem{Fact}[Fact]{Fact}
\aliascntresetthe{Fact}

\theoremstyle{definition}

\newaliascnt{definition}{theorem}
\newtheorem{definition}[definition]{Definition}
\aliascntresetthe{definition}

\newaliascnt{construction}{theorem}
\newtheorem{construction}[construction]{Construction}
\aliascntresetthe{construction}

\newaliascnt{example}{theorem}

\aliascntresetthe{example}

\newaliascnt{remark}{theorem}
\newtheorem{remark}[remark]{Remark}
\aliascntresetthe{remark}

\newaliascnt{notations}{theorem}

\aliascntresetthe{notations}

\newaliascnt{convention}{theorem}
\newtheorem{convention}[convention]{Convention}
\aliascntresetthe{convention}

\newaliascnt{assumption}{theorem}

\aliascntresetthe{assumption}

\numberwithin{equation}{section}

\crefname{theorem}{theorem}{theorems}
\Crefname{theorem}{Theorem}{Theorems}

\crefname{lemma}{lemma}{lemmas}
\Crefname{lemma}{Lemma}{Lemmas}

\crefname{corollary}{corollary}{corollaries}
\Crefname{corollary}{Corollary}{Corollaries}

\crefname{proposition}{proposition}{propositions}
\Crefname{proposition}{Proposition}{Propositions}

\crefname{definition}{definition}{definitions}
\Crefname{definition}{Definition}{Definitions}

\crefname{construction}{construction}{constructions}
\Crefname{construction}{Construction}{Constructions}

\crefname{example}{example}{examples}
\Crefname{example}{Example}{Examples}

\crefname{remark}{remark}{remarks}
\Crefname{remark}{Remark}{Remarks}

\crefname{notations}{notations}{notations}
\Crefname{notations}{Notations}{Notations}

\crefname{convention}{convention}{conventions}
\Crefname{convention}{Convention}{Conventions}

\crefname{assumption}{assumption}{assumptions}
\Crefname{assumption}{Assumption}{Assumptions}

\numberwithin{equation}{section}

\newcommand\Old{\bgroup\markoverwith{\textcolor{red}{\rule[0.5ex]{2pt}{0.4pt}}}\ULon}

\def\<{\langle}
\def\>{\rangle}
\def\shk{F\x}
\def\NN{\mathbb{N}}
\def\ZZ{\mathbb{Z}}
\def\QQ{\mathbb{Q}}
\def\CC{\mathbb{C}}
\def\XX{\mathbb{X}} 
\newcommand{\cC}{\mathcal{C}}
\newcommand{\cD}{\mathcal{D}}
\newcommand{\QQi}{\mathbb{Q}_\infty}
\newcommand{\vc}{\vec{c}}
\newcommand{\vn}{\vec{w}}  
\newcommand{\vx}{\vec{x}}
\newcommand{\lcm}{\operatorname{lcm}}
\newcommand{\rk}{\operatorname{rk}}   

\newcommand{\genus}[1]{g_{#1}}

\renewcommand{\mod}{\mathsf{mod}\hspace{.01in}}
\newcommand{\Int}{\operatorname{Int}\nolimits}
\newcommand{\Hom}{\operatorname{Hom}\nolimits}

\newcommand{\Ext}{\operatorname{Ext}\nolimits}

\newcommand{\id}{\operatorname{id}\nolimits}

\newcommand{\Ind}{\operatorname{Ind}\nolimits}
\newcommand{\D}{\operatorname{\mathcal{D}}}
 \newcommand{\Pic}{\operatorname{Pic}\nolimits}
\def\surf{\mathbf{S}} 
\def\surfi{\surf^\circ}

\newcommand{\oInt}{\overrightarrow{\operatorname{Int}}}

\newcommand{\gind}{\operatorname{ind}\nolimits}
\def\numbers{\begin{enumerate}[label=\arabic*{$^\circ$}.]}
\def\ends{\end{enumerate}}

\def\wt{\omega}

\newcommand{\bv}{\mathbf{v}}
\newcommand{\bh}{\mathbf{h}}
\def\nn{node{$\bullet$}}

\def\grad{\eta}

\def\gmsx{\gms\x}         

\def\x{_\vot}

\newcommand{\MCG}{\operatorname{MCG}}

\def\coh{\operatorname{coh}}

\def\Dcoh{\D^b(\coh \XX)}

\def\Aut{\operatorname{Aut}}
\def\Ind{\operatorname{Ind}}

\def\Hom{\operatorname{Hom}}

\def\Ext{\operatorname{Ext}}

\def\Br{\operatorname{Br}}
\newcommand{\tubR}{\mathcal{S}}
\newcommand{\tubL}{\mathcal{L}}

\def\PSL{\operatorname{PSL}}
\def\Dfang{\mathrm{D}^2_\vot}
\newcommand{\typ}[1]{\mathrm{t}(#1)}

\newcommand{\cO}{\mathcal{O}}

\newcommand{\gms}{\surf^\grad}
\def\surf{\mathbf{S}}                       
\def\M{\mathbf{M}} 
\def\A{\mathbf{A}} 
\def\coh{\operatorname{coh}}

\def\CPone{\mathbb{CP}^1_{\wt}}
\newcommand\DC[1]{\D^b(\coh(#1))}
\def\CCP{\mathcal{C}(\CPone)}

\def\TA{\mathbf{T}\mathbf{A}}
\def\wA{\widetilde{\A}}
\def\wX{\widetilde{X}}

\renewcommand\P{\mathbf{P}} 

\newcommand{\bfH}{\mathbb{H}}
\newcommand{\bfh}{\mathbf{h}}
\renewcommand{\nn}{node{$\bullet$}}            %

\usepackage{marvosym}
\usepackage{rotating}
\usepackage{marvosym}
\def\vot{\text{\Biohazard}}
\def\sx{\text{\Cancer}}
 
\begin{document}
\title [Top models: the weighted
projective lines  of type $(2,2,2,2)$]
 {Topological models for   categories  associated with the weighted
projective lines  of type $(2,2,2,2)$}

\author{Jianmin Chen}
\address{Cj: School of Mathematical Sciences, Xiamen University, 361005 Xiamen,  China.}
\email{chenjianmin@xmu.edu.cn}


\author{Li Fan}
\address{Fl: Morningside Center of Mathematics, Chinese Academy of Sciences, 100190 Beijing, China.}
\email{fan-l17@tsinghua.org.cn}

 \author{Yu Qiu}
\address{Qy:
	Yau Mathematical Sciences Center and Department of Mathematical Sciences,
	Tsinghua University,
    100084 Beijing,
    China.
    \&
    Beijing Institute of Mathematical Sciences and Applications, Yanqi Lake, Beijing, China.}
\email{yu.qiu@bath.edu}

\author{Yiting Zheng}
\address{Zy: School of Mathematical Sciences, Xiamen University, 361005 Xiamen, China.}
\email{ytzhengxmu@163.com}

\dedicatory{}

\begin{abstract}
We construct a geometric model for the derived category of the weighted projective line $\mathbb{CP}^1_{\omega}$ of type $(2,2,2,2)$ using a graded sphere with four binaries. 
More precisely, we give a bijection between the set of indecomposable rigid objects and the set of certain arcs, such that the dimensions of  $\operatorname{Hom}$-spaces are computed by oriented intersection numbers.
As applications, we provide a geometric realization of the indecomposable rigid objects in the cluster category $\mathcal{C}(\mathbb{CP}^1_{\omega})$ via tagged arcs on the four-punctured sphere, with $\operatorname{Hom}$-space dimensions   given by tagged intersection numbers. 
When the weighted points of $\mathbb{CP}^1_{\omega}$ are $(0,1,\infty,\frac{1}{2})$,  both   correspondences are compatible with the natural actions of the automorphism groups and the corresponding mapping class groups.
\end{abstract}

\maketitle
\tableofcontents
\addtocontents{toc}{\protect\setcounter{tocdepth}{1}}

\setlength\parindent{0pt}
\setlength{\parskip}{5pt}

\section{Introduction}
\subsection{Motivation}
  Weighted projective lines and their categories of coherent sheaves were introduced by Geigle and Lenzing~\cite{GL87}, providing a geometric framework for the representation theory of canonical algebras in the sense of Ringel~\cite{Ringel1984}. 
The study of weighted projective lines is closely  connected with several branches of mathematics, including Lie theory~\cite{Crawley-Boevey2010,DengRuanXiao2020,DouJiangXiao2012,Schiffmann2004}, representation theory~\cite{M2004,CLLR2021,FG383,HR1999}, and singularity theory~\cite{EbelingPloog2010,Hubner1989,Hubner1996,Lenzing1994,Lenzing1998,Lenzing2011}. 
In particular, it plays an important role  in the understanding of Arnold's strange duality~\cite{EbelingPloog2010,EbelingTakahashi2013} and homological mirror symmetry~\cite{Ebeling2003,Scerbak1978}.

The category of coherent sheaves on a weighted projective line is derived equivalent to the module category of a canonical algebra of the same type. 
Moreover, as shown in~\cite{GL87,Lenzing2004,CCZ}, the category of coherent sheaves on a weighted projective line of tubular type is equivalent to the category of equivariant coherent sheaves on an elliptic curve with respect to a suitable finite cyclic group action; see also~\cite{P2006}.

In this paper, we focus on the weighted projective line
\[
\CPone \coloneqq \bigl(\mathbb{P}^{1},\,(0,1,\infty,\lambda),\,(2,2,2,2)\bigr)
\]
of tubular type $(2,2,2,2)$ over the field $\CC$, with $\lambda\in \CC\setminus\{0,1\}$.
Our aim is to construct a geometric model, via a marked surface, for the bounded derived category $\DC{\CPone}$ of the coherent sheaf category on $\CPone$.

Surface models for categories have been extensively studied in recent years. Fomin--Shapiro--Thurston (FST)~\cite{FST} introduced the notion of a marked surface $\surf$ with punctures.
Given any tagged triangulation of $\surf$, they associate  a quiver to it, while tagged arcs correspond to cluster variables. Moreover, flips of tagged triangulations correspond to mutations of quivers. Such a triangulation is obtained by gluing puzzle pieces of types I--III, except for the type-IV piece, which is a triangulation of the sphere $S_{0,4}$ with four punctures; see \Cref{fig:fst-puzzle-pieces}. Qiu--Zhou~\cite{QZ1}
studied the skewed-gentle algebras associated with an admissible triangulation of a punctured marked surface with nonempty boundary, and established a correspondence between tagged curves and string objects. Their construction therefore does not cover $S_{0,4}$. In this paper, we study this exceptional case by realizing the indecomposable rigid objects in the cluster category associated with $\CPone$ as tagged arcs on $S_{0,4}$.

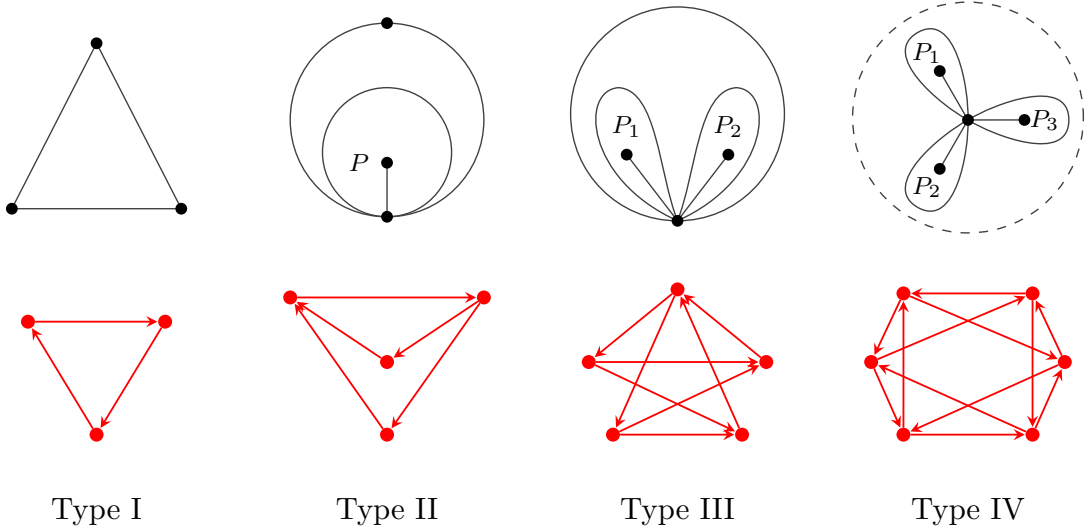
\begin{figure}[htbp]
\centering
\resizebox{.98\textwidth}{!}{%
\begin{tikzpicture}[
  x=.92cm,y=.92cm,
  pcurve/.style={draw=black!75,line width=.45pt},
  ppoint/.style={circle,fill=black,inner sep=1.35pt},
  qedge/.style={->,>=stealth,draw=red,line width=.6pt,
    line cap=round,line join=round,
    shorten <=2.5pt,shorten >=2.5pt},
  qpoint/.style={circle,fill=red,inner sep=1.6pt},
  qlabel/.style={text=red,font=\scriptsize,inner sep=.5pt},
  every node/.style={font=\scriptsize}
]
\begin{scope}[shift={(0,0)}]
\node[font=\small] at (1.30,-3.40) {Type I};
\coordinate (Ibl) at (.25,.35);
\coordinate (Ibr) at (2.35,.35);
\coordinate (It) at (1.30,2.40);
\draw[pcurve] (Ibl)--(It)--(Ibr)--cycle;
\foreach \p in {Ibl,Ibr,It}\node[ppoint] at (\p) {};
\coordinate (IqB) at (.45,-1.05);
\coordinate (IqC) at (2.15,-1.05);
\coordinate (IqD) at (1.30,-2.45);
\draw[qedge,line width=.6pt,shorten <=2.5pt,shorten >=2.5pt] (IqB)--(IqC);
\draw[qedge,line width=.6pt,shorten <=2.5pt,shorten >=2.5pt] (IqC)--(IqD);
\draw[qedge,line width=.6pt,shorten <=2.5pt,shorten >=2.5pt] (IqD)--(IqB);
\foreach \p in {IqB,IqC,IqD}
  \node[qpoint,inner sep=1.6pt] at (\p) {};
\end{scope}
\begin{scope}[shift={(3.45,0)}]
\node[font=\small] at (1.45,-3.40) {Type II};
\coordinate (IIc) at (1.45,.25);
\coordinate (IIt) at (1.45,2.65);
\draw[pcurve] (1.45,1.45) circle[radius=1.20];
\draw[pcurve] (1.45,1.05) circle[radius=.80];
\coordinate (IIP) at (1.45,.92);
\draw[pcurve] (IIP)--(IIc);
\foreach \p in {IIc,IIt,IIP}\node[ppoint] at (\p) {};
\node[left=2pt of IIP] {$P$};
\coordinate (IIqZ) at (.25,-.75);
\coordinate (IIqW) at (2.65,-.75);
\coordinate (IIqAplus) at (1.45,-1.55);
\coordinate (IIqAminus) at (1.45,-2.45);
\draw[qedge,line width=.6pt,shorten <=2.5pt,shorten >=2.5pt]
  (IIqZ)--(IIqW);
\draw[qedge,line width=.6pt,shorten <=2.5pt,shorten >=2.5pt]
  (IIqW)--(IIqAplus);
\draw[qedge,line width=.6pt,shorten <=2.5pt,shorten >=2.5pt]
  (IIqAplus)--(IIqZ);
\draw[qedge,line width=.6pt,shorten <=2.5pt,shorten >=2.5pt]
  (IIqW)--(IIqAminus);
\draw[qedge,line width=.6pt,shorten <=2.5pt,shorten >=2.5pt]
  (IIqAminus)--(IIqZ);
\foreach \p in {IIqZ,IIqW,IIqAplus,IIqAminus}
  \node[qpoint,inner sep=1.6pt] at (\p) {};
\end{scope}
\begin{scope}[shift={(7.05,0)}]
\node[font=\small] at (1.45,-3.40) {Type III};
\coordinate (IIIo) at (1.45,.20);
\draw[pcurve] (1.45,1.525) circle[radius=1.325];
\draw[pcurve] (IIIo)
  .. controls (1.22,.78) and (1.20,1.82) .. (.72,1.85)
  .. controls (.24,1.82) and (.30,.72) .. (IIIo);
\draw[pcurve] (IIIo)
  .. controls (1.68,.78) and (1.70,1.82) .. (2.18,1.85)
  .. controls (2.66,1.82) and (2.60,.72) .. (IIIo);
\coordinate (IIIPone) at (.82,1.02);
\coordinate (IIIPtwo) at (2.08,1.02);
\draw[pcurve] (IIIPone)--(IIIo)--(IIIPtwo);
\foreach \p in {IIIo,IIIPone,IIIPtwo}\node[ppoint] at (\p) {};
\node[above=2pt of IIIPone] {$P_1$};
\node[above=2pt of IIIPtwo] {$P_2$};
\coordinate (IIIqC) at (1.45,-.65);
\coordinate (IIIqAonePlus) at (.35,-1.55);
\coordinate (IIIqAtwoPlus) at (2.55,-1.55);
\coordinate (IIIqAoneMinus) at (.65,-2.45);
\coordinate (IIIqAtwoMinus) at (2.25,-2.45);
\draw[qedge] (IIIqC)--(IIIqAonePlus);
\draw[qedge] (IIIqAonePlus)--(IIIqAtwoPlus);
\draw[qedge] (IIIqAtwoPlus)--(IIIqC);
\draw[qedge] (IIIqC)--(IIIqAoneMinus);
\draw[qedge] (IIIqAoneMinus)--(IIIqAtwoMinus);
\draw[qedge] (IIIqAtwoMinus)--(IIIqC);
\draw[qedge] (IIIqAonePlus)--(IIIqAtwoMinus);
\draw[qedge] (IIIqAoneMinus)--(IIIqAtwoPlus);
\foreach \p in {IIIqC,IIIqAonePlus,IIIqAtwoPlus,
  IIIqAoneMinus,IIIqAtwoMinus}\node[qpoint] at (\p) {};
\end{scope}
\begin{scope}[shift={(10.65,0)}]
\node[font=\small] at (1.45,-3.40) {Type IV};
\coordinate (IVo) at (1.45,1.45);
\draw[pcurve,dashed] (1.45,1.48) circle[radius=1.43];
\begin{scope}[shift={(1.45,1.45)},rotate=120]
\draw[pcurve] (0,0)
  .. controls (.32,.20) and (1.25,.55) .. (1.25,0)
  .. controls (1.25,-.55) and (.32,-.20) .. (0,0);
\coordinate (IVPone) at (.70,0);
\draw[pcurve] (0,0)--(IVPone);
\end{scope}
\begin{scope}[shift={(1.45,1.45)},rotate=240]
\draw[pcurve] (0,0)
  .. controls (.32,.20) and (1.25,.55) .. (1.25,0)
  .. controls (1.25,-.55) and (.32,-.20) .. (0,0);
\coordinate (IVPtwo) at (.70,0);
\draw[pcurve] (0,0)--(IVPtwo);
\end{scope}
\begin{scope}[shift={(1.45,1.45)}]
\draw[pcurve] (0,0)
  .. controls (.32,.20) and (1.25,.55) .. (1.25,0)
  .. controls (1.25,-.55) and (.32,-.20) .. (0,0);
\coordinate (IVPthree) at (.70,0);
\draw[pcurve] (0,0)--(IVPthree);
\end{scope}
\foreach \p in {IVo,IVPone,IVPtwo,IVPthree}\node[ppoint] at (\p) {};
\node at (.94,2.28) {$P_1$};
\node at (.94,.62) {$P_2$};
\node at (2.40,1.45) {$P_3$};
\coordinate (IVqAthreeMinus) at (.65,-.70);
\coordinate (IVqAtwoMinus) at (2.25,-.70);
\coordinate (IVqAonePlus) at (.25,-1.55);
\coordinate (IVqAoneMinus) at (2.65,-1.55);
\coordinate (IVqAtwoPlus) at (.65,-2.45);
\coordinate (IVqAthreePlus) at (2.25,-2.45);
\draw[qedge] (IVqAtwoMinus)--(IVqAthreeMinus);
\draw[qedge] (IVqAthreeMinus)--(IVqAonePlus);
\draw[qedge] (IVqAonePlus)--(IVqAtwoPlus);
\draw[qedge] (IVqAtwoPlus)--(IVqAthreePlus);
\draw[qedge] (IVqAthreePlus)--(IVqAoneMinus);
\draw[qedge] (IVqAoneMinus)--(IVqAtwoMinus);
\draw[qedge] (IVqAonePlus)--(IVqAtwoMinus);
\draw[qedge] (IVqAtwoMinus)--(IVqAthreePlus);
\draw[qedge] (IVqAthreePlus)--(IVqAonePlus);
\draw[qedge] (IVqAthreeMinus)--(IVqAoneMinus);
\draw[qedge] (IVqAoneMinus)--(IVqAtwoPlus);
\draw[qedge] (IVqAtwoPlus)--(IVqAthreeMinus);
\foreach \p in {IVqAthreeMinus,IVqAtwoMinus,IVqAonePlus,
  IVqAoneMinus,IVqAtwoPlus,IVqAthreePlus}\node[qpoint] at (\p) {};
\end{scope}
\end{tikzpicture}
}
\caption{The four puzzle pieces and associated quiver in \cite{FST} 
}
\label{fig:fst-puzzle-pieces}
\end{figure}  
Haiden--Katzarkov--Kontsevich (HKK)~\cite{HKK} later realized the derived categories of graded gentle algebras as topological Fukaya categories of graded  marked surfaces, in which  the surface grading encodes the shift functor.
For graded skew-gentle algebras, Qiu--Zhang--Zhou~\cite{QZZ} constructed a geometric model for perfect derived categories using graded marked surfaces with binaries.
 At each binary, the relation $\mathrm{D}^2=\mathrm{id}$ gives a topological realization of the $\mathbb{Z}_2$-symmetry, where $\mathrm{D}$ is the Dehn twist along the binary. By adding one decoration in each triangle of a triangulation, Qiu~\cite{QQ,QiuB} constructed a topological model for the $3$-Calabi--Yau category associated with the corresponding quiver with potential. This model is compatible with the cluster category model of Qiu--Zhou~\cite{QZ1}, in the sense that forgetting the decorations corresponds to passing to the cluster category~\cite{QiuB}. A new topological realization of the $\mathbb{Z}_2$-symmetry at punctures was introduced by Qiu--Zhou~\cite{QZ3} via decorated marked surfaces with vortices. Thus tagging, binaries, and vortices provide three realizations of the same local $\mathbb{Z}_2$-symmetry. Geometrically, the vortex model is obtained from the binary model by moving the closed marked point into the interior as a decoration and regarding the open marked point as a puncture, see the introduction of \cite{QZ3} for further discussion.

\subsection{Main results}

Motivated by the   work  of  Barot--Geiss  \cite{BG}, which   established a bijection between indecomposable rigid objects in the category $\coh(\CPone)$ and tagged arcs on $S_{0,4}$, we aim to construct a geometric model for the derived category $\DC{\CPone}$ in terms of a  graded sphere with four binaries.
Let    $\widetilde{S\x^2}$ be  the graded  sphere with four binaries. Here a binary is    a boundary component with one marked point and satisfies the condition that the square of the Dehn twist along it is the identity. We give a topological model for the indecomposable rigid objects in $\DC{\CPone}$ as follows.
\begin{theorem}[{\Cref{thm:X}, \Cref{autd=mag}}]
  There is a bijection $\wX$ between graded simple arcs on $\widetilde{S\x^2}$ and indecomposable rigid objects in $\DC{\CPone}$. In particular, if $\CPone=(\mathbb{P}^1,(0,1,\infty,\frac{1}{2}),(2,2,2,2))$, then $\Phi: \MCG^\sx(\widetilde{S\x^2})\to \Aut\DC{\CPone}$ is a group isomorphism,
   \begin{equation}
\begin{tikzpicture}[xscale=.6,yscale=.6]
\draw(180:3)node(o){$\wA^\circ(\widetilde{S\x^2} )$}(-3,2.2)node(b){\small{$\MCG^\sx(\widetilde{S\x^2})$}}
(-.8,.5)node{$\wX$} (-.8,-.3)node{$\sim$};
\draw(0:3)node(a){$\Ind^\circ\DC{\CPone}$}(3,2.2)node(s){\small{$\Aut\DC{\CPone}$}};
\draw[->,>=stealth](o)to(a);\draw[->,>=stealth](b)to(s);
\draw[->,>=stealth](-3.2,.6).. controls +(135:2) and +(45:2) ..(-3+.2,.6);
\draw[->,>=stealth](3-.2,.6).. controls +(135:2) and +(45:2) ..(3+.2,.6);
\draw (-.5,2.6)node{$\Phi$} (-.5,1.9)node{$\sim$};
\end{tikzpicture}
\end{equation}
and $\wX$ is compatible with $\Phi$, where $\MCG^\sx(\widetilde{S\x^2})$ is the binary mapping class group of $\widetilde{S\x^2}$.
\end{theorem} 

If we forget the grading  on $\widetilde{S\x^2}$,  the underlying surface is obtained from    $S_{0,4}$ via
Construction~\ref{cons:binary}.  
However, unlike the construction in  \cite{BG}, the
  image of an indecomposable rigid object in $\DC{\CPone}$ under $\widetilde{X}$ is determined by
 the universal covering map $\tilde{\pi} \colon\mathbb{R}^2 \to S_{0,4} $ described in Construction~\ref{arcs pair}.
 Using this covering description, we can compute the number of interior intersections between graded simple arcs on $\widetilde{S\x^2}$. 
Consequently, we have the following result.
\begin{theorem}[{\Cref{thm:X2}}]
  Under the bijection $\wX$, the dimensions of the Hom-space between two rigid objects in $\DC{\CPone}$ are given by oriented intersection numbers:
\begin{equation}
\oInt^d(\widetilde{\sigma},\widetilde{\gamma})
=\dim\Hom_{\cD}(\wX(\widetilde{\sigma}),\wX(\widetilde{\gamma})[d])
\end{equation}
for graded simple arcs $\widetilde{\sigma},\widetilde{\gamma}$ on $\widetilde{S\x^2}$ and $d\in\ZZ$.
\end{theorem}
 
As an application, we provide the geometric realization of the indecomposable rigid objects in the cluster category $\CCP$ via tagged arcs on the sphere with four punctures $S_{0,4}$, although this realization can also be obtained directly, without passing through the derived category.
\begin{theorem}[{\Cref{thm:C}, \Cref{thm:C2}, \Cref{autc=mac}}]
There is a bijection $X_c$ between the set of tagged arcs on $S_{0,4}$ and the set of indecomposable rigid objects in $\CCP$.  The dimensions of $\Hom$-spaces between any two such objects equal the tagged intersection numbers of the corresponding tagged arcs. 

In particular, if $\CPone=(\mathbb{P}^1,(0,1,\infty,\frac{1}{2}),(2,2,2,2))$, then $\Phi': \MCG^\times  (S_{0,4})\to \Aut\CCP$ is a group isomorphism, 
\begin{equation}
\begin{tikzpicture}[xscale=.6,yscale=.6]
\draw(180:3)node(o){$\TA^\times(S_{0,4})$}(-3,2.2)node(b){\small{$ \MCG^\times  (S_{0,4})$}}
(-.2,.45)node{$X_c$} (-.2,-.3)node{$\sim$};
\draw(0:3)node(a){$\Ind^\circ\CCP$}(3,2.2)node(s){\small{$\Aut\CCP$}};
\draw[->,>=stealth](o)to(a);\draw[->,>=stealth](b)to(s);
\draw[->,>=stealth](-3.2,.6).. controls +(135:2) and +(45:2) ..(-3+.2,.6);
\draw[->,>=stealth](3-.2,.6).. controls +(135:2) and +(45:2) ..(3+.2,.6);
\draw (0,2.6)node{$\Phi'$} (0,1.9)node{$\sim$};
\end{tikzpicture}
\end{equation}
and $X_c$ is compatible with $\Phi'$, where $\MCG^\times  (S_{0,4})$ is the tagged mapping class group of $S_{0,4}$.
\end{theorem}
Although  the bijection $X_c$ was  already established by  Barot and Geiss in  \cite{BG}, its geometric interpretation of the $\Hom$-space  dimensions and automorphism group of $\CCP$ remained implicit.
The present work  fills the gap. In addition, we show  that every loop in    the exchange graph for cluster-tilting objects in $\CCP$ decomposes into squares and pentagons, see \Cref{cor:5.6}.

\subsection{Contents}
The paper is organized as follows. In \Cref{sec:2}, we review background material on coherent sheaves on weighted projective lines and graded marked surfaces with binaries. In \Cref{sec:3}, we provide some properties of rigid sheaves in $\coh(\CPone)$. In \Cref{sec:4}, we give a surface model for indecomposable rigid objects in  the derived category $\DC{\CPone}$ via a graded sphere with four binaries. Moreover, we provide geometric interpretations for the dimensions of $\Hom$-spaces by oriented intersection
numbers, and for the automorphism group of   $\DC{\CPone}$ through the binary mapping class group. In \Cref{sec:5}, we show that the sphere with four punctures gives a geometric model of 
 the cluster category  $\CCP$, where  the dimensions of  $\Hom$-spaces  are expressed by tagged intersection numbers, and the automorphism group is described by the tagged mapping class group.
 
In \Cref{tab_1}, we summarize the notations used in this paper.
\begin{center}
\begin{table}[htbp]
\caption{List of notations}\label{tab_1}
\centering
\renewcommand{\arraystretch}{1.15}
\begin{tabular}{c@{\qquad}p{0.8\textwidth}}
\hline
$\CPone$ & the weighted projective line of type $(2,2,2,2)$\\
$\coh(\CPone)$ & the category of coherent sheaves on $\CPone$\\
$\DC{\CPone}$ & the bounded derived category of $\coh(\CPone)$\\
$\CCP$ & the cluster category associated with $\CPone$\\
$E_p^x$ & the indecomposable rigid sheaf indexed by $(p,x)$ in $\coh(\CPone)$\\
$\bv_p^x$ & the positive real Schur root associated with $E_p^x$\\
$S_{0,4}$ & the sphere with four punctures\\
$\widetilde{S}_{0,4}$ & the graded sphere with four punctures\\
$\widetilde{S\x^2}$ & the graded sphere with four binaries\\
$\alpha_p^\epsilon$ & the simple arc of slope $p$ on $S_{0,4}$, where $\epsilon\in\{+,-\}$\\
$\widetilde{\alpha}_p^\epsilon$ & the graded simple arc corresponding to $\alpha_p^\epsilon$\\
$\alpha_{p,x}$ & the tagged arc associated with $p\in\QQi$ and $x\in\bfH$\\
$\widetilde{\alpha}_{p,x}$ & the graded tagged arc associated with $\alpha_{p,x}$\\
$\oInt^d(\widetilde{\sigma},\widetilde{\gamma})$ & the oriented intersection number  of index $d$ from $\widetilde{\sigma}$ to $\widetilde{\gamma}$\\
$\Int^d(\alpha_{p,x},\alpha_{q,y})$ & the tagged intersection number of index $d$ between $\alpha_{p,x}$ and $\alpha_{q,y}$\\
$\MCG^\sx(\widetilde{S\x^2})$ & the binary mapping class group of $\widetilde{S\x^2}$\\
$\MCG^\times(S_{0,4})$ & the tagged mapping class group of $S_{0,4}$\\
\hline
\end{tabular}
\end{table}
\end{center}
\subsection*{Acknowledgments} J. Chen and Y. Zheng were partially  supported by   the  National Natural Science Foundation of China (Grants Nos. 12371040 and 12131018). Y. Qiu was supported by the National Natural Science Foundation of China (Grant No. 12425104). 

\section{Preliminaries}\label{sec:2}
  
\subsection{Coherent sheaves on weighted projective lines} 
Let $\mathbf{k}$ be an algebraically closed field and denote by $\mathbb{P}^1$ the projective  
line over $\mathbf{k}$.  
A \emph{weighted projective line} $\XX=(\mathbb{P}^1,\boldsymbol{\lambda},\mathbf{p})$ over $\mathbf{k}$ is specified by giving a  \emph{weight sequence} $\mathbf{p}=(p_1,\ldots,p_t)$ of integers $p_i\in\NN_{\geq 2}$ and a \emph{parameter sequence }
$\boldsymbol{\lambda}=(\lambda_1,\ldots,\lambda_t)$ of pairwise distinct points in $\mathbb{P}^1$, which can be normalized as  
$\lambda_1=0$, $\lambda_2=\infty$ and $\lambda_3=1$ for $t\geq 3$.
Each such  weighted projective line is associated  with a  \emph{weight function}  defined by
\[
p_\lambda:\mathbb{P}^1\rightarrow\NN,\mu\mapsto\begin{cases}
p_i, &\text{if } \mu=\lambda_i \text{ for some } i,\\
1,   &\text{else.}\end{cases} 
\]

  The genus  of   $\XX$, denoted by $\genus{\XX}$,   is given by
$$\genus{\XX} =1+\frac{1}{2}\left((t-2)\lcm(p_{1}, \cdots, p_{t})-\sum_{i=1}^t \frac{\lcm(p_{1}, \cdots, p_{t})}{p_i}\right).$$ 
This genus gives the following classification: a weighted projective line of genus $\genus\XX< 1$($\genus\XX= 1$, resp. $\genus\XX> 1$) is called of \emph{domestic (tubular, resp. wild)} type. 
Up  to permutation of weights,  it is elementary to verify that $\XX$ is 
 domestic   if and only if
  \[
\mathbf{p}\in\{( ), (p_1), (p_1 , p_2) ,(2, 2, p_3) ,(2, 3, 3), (2, 3, 4),  (2, 3, 5)  \},
 \] where $()$ corresponds to the case   $t=0$, and  $\XX$ is tubular if and only if 
 \[
\mathbf{p}\in\{(6,3,2),\ (4,4,2),\ (3,3,3),\ (2,2,2,2)\}.
 \]   

We now revisit  the definition of the category $\coh\XX$ of coherent sheaves on $\mathbb{X}$,  as   introduced in  \cite[Section 1]{GL87}. 
Let $\mathbb{L}(\mathbf{p})$ be the rank 1 additive group
\[
\mathbb{L}\coloneqq\mathbb{L}(\mathbf{p})= \<\vx_1,\ldots,\vx_t\mid p_1\vx_1=\cdots=p_t\vx_t=\vc\>,
\] where $\vc$ is called the \emph{canonical element} of $\mathbb{L}$,
and $S(\mathbf{p},\boldsymbol{\lambda})$ be the $\mathbb{L}$-graded commutative algebra
\[
{\rm S}\coloneqq S(\mathbf{p},\boldsymbol{\lambda}) = \mathbf{k}[x_1,\ldots,x_t]/
\<x_i^{p_i}-x_2^{p_2}-\lambda_i x_1^{p_1}\mid i=1,\ldots,t\>, 
\]
where $\deg x_i=\vx_i$ and $\lambda_i \in\mathbb{P}^1$.
   Let ${\rm mod}^{\mathbb{L}}\ {\rm S}$ be the abelian category of finitely generated $\mathbb{L}$-graded ${\rm S}$-modules, and 
${\rm mod}_0^{\mathbb{L}}\ {\rm S}$ be its Serre subcategory formed by finite dimensional modules. Denote by $${\rm qmod}^{\mathbb{L}}\ {\rm S}:={\rm mod}^{\mathbb{L}}\ {\rm S}/{\rm mod}_0^{\mathbb{L}}\ {\rm S}$$ the quotient abelian category. By \cite[Theorem 1.8]{GL87}, the sheafification functor yields an equivalence
\[
{\rm qmod}^{\mathbb{L}}\ {\rm S}\xrightarrow{\sim} \coh\XX.
\] The free module ${\rm S}$ gives the structure
sheaf $\cO$, and shifting the grading gives twists $E(\vx)$ for any sheaf
$E\in\coh\XX$ and $\vx\in \mathbb{L}$. 
From now on, we will identify these two categories.  

It is well known that $\coh\XX$ is a hereditary  $\Hom$-finite abelian category  satisfying
Serre duality
\[
D\Ext^1_\XX(E,F)\cong\Hom_\XX(F,E(\vn)),
\]
 where $\vn = \sum_{i=1}^t(\vc-\vx_i)-2\vc$ is called the \emph{dualizing element} of $\mathbb{L}$. This  implies the existence of almost split sequences for the category $\coh\XX$  with the Auslander-Reiten translation $\tau:\coh\XX\to\coh\XX$ given by the grading shift with $\vn$. 
 Moreover, 
$\coh\XX$ admits a splitting torsion pair $({\rm vect}\mbox{-}\mathbb{X}, {\rm coh}_{0}\mbox{-}\mathbb{X})$, 
where ${\rm vect}\mbox{-}\mathbb{X}$ (resp.  ${\rm coh}_{0}\mbox{-}\mathbb{X}$) denotes the full subcategory of $\coh\XX$ consisting of coherent sheaves without any simple subobjects (resp. coherent sheaves of finite length).
The objects in ${\rm vect}\mbox{-}\mathbb{X}$ are called \emph{vector bundles}.  
There is a specific type of vector bundles  called \emph{line bundles}. Up to isomorphism, each line bundle has the form $\cO(\vec{x})$ for a uniquely determined $\vec{x}\in\mathbb{L}$.    

By \cite[Proposition 1.1]{LR}, the torsion subcategory ${\rm coh}_0\mbox{-}\mathbb{X}$ of $\coh\XX$ decomposes into a coproduct
$\coprod_{\mu\in \mathbb{P}^1}\mathcal{T}_{\mu}$, where $\mathcal{T}_{\mu}$ is a connected uniserial length category, whose associated Auslander-Reiten quiver is a stable tube $\mathbb{ZA}_{\infty}/(\tau^r)$ for some $r\in \mathbb{Z}_{\geq 1}$ (c.f.\cite{ASS}). Here, the integral $r$ is called the rank of the stable tube $\mathbb{ZA}_{\infty}/(\tau^r)$ and depends on $\mu$. 
Precisely, 
\[
r=\begin{cases}
p_i, &\text{if }   \mu=\lambda_i,\\
1,   &\text{if $\mu\in   \mathbf{k}\setminus \{ \lambda_1,\ldots,\lambda_t\}$.}\end{cases} 
\] 
Objects that lie at the bottom of the stable tubes are all simple objects of $\coh\XX$. Each $\mu\in   \mathbf{k}\setminus \{ \lambda_1,\ldots,\lambda_t\}$ is associated with a unique simple sheaf $S_{\mu}$, called \emph{ordinary simple}; while $\mu=\lambda_i$   is associated with $p_i$  simple objects $ S_{\lambda_i,j} \, (j\in \mathbb{Z}/p_i\mathbb{Z}) $ 
called \emph{exceptional simples}. 
 For each $\vec{x}= \sum_{i=1}^t l_i\vx_i\in \mathbb{L}$ , the twists act on the simple sheaves by  
\begin{equation}\label{shift}
 S_{\lambda_i,j} (\vec{x})=S_{\lambda_i, j+l_{i}},\;\;   S_{\mu}(\vec{x})=S_{\mu}\;\;{\rm for}\;\, \mu\in  \mathbf{k}\setminus \{ \lambda_1,\ldots,\lambda_t\}.    
\end{equation}
Besides, for each ordinary simple sheaf $S_{\mu}$, there is an
 exact sequence
$$0\longrightarrow \mathcal{O}\stackrel{x_2^{p_2}-\mu x_1^{p_1}}{\longrightarrow}\mathcal{O}(\vec{c})\longrightarrow S_{\mu}\longrightarrow 0.$$
For exceptional simple sheaf  $ S_{\lambda_i,j} $, there is an exact sequence 
\begin{equation}\label{es}
    0\longrightarrow \mathcal{O}((j-1)\vec{x}_i)\stackrel{x_i}{\longrightarrow} \mathcal{O}(j\vec{x}_i)\longrightarrow  S_{\lambda_i,j} \longrightarrow 0,\;\;j\in\mathbb{Z}/p_i\mathbb{Z}.
\end{equation}

The Grothendieck group $K_{0}(\XX)$ of $\coh\XX$  was computed by Geigle and Lenzing \cite{GL87}. In this paper, we   write  $[E] \in  K_{0}(\XX)$
 for the class of an object $E \in \coh\XX$. The Euler form on $K_{0}(\XX)$ is defined as follows on classes of objects $E,F \in \coh\XX$
\[
\<[E],[F]\>= \dim\Hom_\XX(E,F)-\dim\Ext^1_\XX(E,F).
\] 

There are several important $\mathbb{Z}$-linear maps on $K_0(\XX)$, including the \emph{determinant} $\det$, the \emph{rank} $\rk$, and the \emph{degree} $\deg$.
The determinant map is the group homomorphism
\[
\det \colon K_0(\XX) \to \mathbb{L}
\]
given by $\det[\mathcal{O}](\vec{x}) = \vec{x}$.
The degree function
is determined by
\[ \deg [S]=\frac{\lcm(p_{1}, \cdots, p_{t})}{p_\lambda(\mu)},\quad 
\deg[\mathcal{O}(\vec{x})] = \delta(\vec{x}),
\]
where $S$ is a simple sheaf concentrated at $\mu\in\mathbb{P}^1$  and $\delta: \mathbb{L} \to \ZZ$ is  the group homomorphism defined on generators by $\delta(\vx_{i}) = \lcm(p_{1}, \cdots, p_{t})/p_{i}$. 
 The rank function $\rk \colon K_0(\XX) \to \mathbb{Z}$ is characterized by
\[
\rk[\mathcal{O}(\vec{x})] = 1.
\] 
For each non-zero object $E \in \coh\XX$,  define the \emph{slope} of   $E$  as
\[
\mu E \coloneqq \frac{\deg [E]}{\rk [E]} \in \QQi\coloneqq \mathbb{Q} \cup \{\infty\}.
\]  
An indecomposable object $E \in \coh\XX$ is called \emph{semistable (resp. stable)} if for each non-trivial subbundle $E'$ of $E$, we have $\mu E' \leq \mu E$ (resp. $\mu E' <\mu E$). 
For each $q \in \QQi$, we denote by $\mathcal{C}_q$ the full subcategory of $\coh\XX$ consisting of all semistable coherent sheaves of slope $q$.

\begin{proposition} \label{coh:prp-stab} Let $\XX$ be a  weighted projective line.   Then the following hold:
\begin{itemize}
\item[(a)]
For each $q\in\QQ$, $\cC_q$ is an extension closed exact abelian finite length subcategory of
$\coh\XX$ with the simple objects being precisely the stable vector bundles.
\item[(b)] For $q, q' \in\QQi$ such that $q > q'$,
$\Hom_\XX(\cC_q,\cC_{q'})=0$.
 \end{itemize}
 In particular, if  $\XX$ is tubular, then the following additional properties hold:
\begin{itemize}
\item[(c)] 
For any $q\in\QQi$, the subcategory $\cC_q\subset\coh\XX$ 
 is closed under the formation of Auslander-Reiten sequences and it is equivalent to $\cC_\infty$.
\item[(d)]
For any indecomposable  sheaf $E\in\coh\XX$, $E\in\cC_q$ for some $q\in\QQi$.
\end{itemize}
\end{proposition}

\subsection{Derived category of $\coh\XX$} 
Denote by $\Dcoh$ the   bounded derived category of $\coh\XX$,  and by $[1]$ the   shift   functor of $\Dcoh$.  
Since $\coh\XX$ is hereditary,  
 the following proposition is well known.

\begin{proposition}\label{prop:db}
    The category $\Dcoh$ is naturally equivalent to the repetitive category  $\lor_{n\in\ZZ}\coh\XX[n]$,
where each  $\coh\XX[n]$ is a copy of $\coh\XX$, with objects written $E[n]$ for $E$ in  $\coh\XX$, and
morphisms given by
\[
\Hom_{\cD}(E[n], F[m])=\Ext_{\XX}^{m-n}(E,F).
\]  
 Here, the expression   $\lor_{n\in\ZZ}\coh\XX[n]$ has two meanings: first, it denotes the additive closure   ${\rm add}(\cup_{n\in\ZZ}\coh\XX[n])$  of the union of all  $\coh\XX[n]$; second, it indicates that there are no nonzero morphisms  from  $\coh\XX[n]$  to $\coh\XX[m]$ 
whenever $n > m$.
\end{proposition}

  Recall that   a sheaf $E\in\coh\XX$ is   \emph{rigid} if $\Ext^1_\XX(E,E)=0$, and 
 an object $E$ in $\Dcoh$ is  \emph{rigid} if $\Hom_{\cD}(E, E[i])= 0$ for all $i\ne0$.  Proposition~\ref{prop:db} implies that every indecomposable rigid object of  $\Dcoh$  has the form $E[n]$, where $E $ is an indecomposable rigid sheaf in    $\coh\XX$ and $n$ is an integer.

\subsection{Cluster category of   $\coh\XX$} 
 Let $\tau:\Dcoh\to \Dcoh$ be the Auslander-Reiten translation functor.
The {\it cluster category of   $\coh\XX$} is defined as the orbit category \[\cC_\XX\coloneqq\Dcoh/\tau^{-1}[1].\] 

 Its objects  
are the same as those of $\Dcoh$, 
and morphism spaces are given by \[\Hom_\cC( {X}, {Y})\coloneqq
 \bigoplus_{i\in\ZZ}\Hom_{\cD}(X,(\tau^{-1}[1])^iY)
 \]
 with obvious composition. 
 The cluster category of $\coh\XX$ is a 
triangulated 2-Calabi-Yau category \cite{K}
admitting a cluster structure in the sense of \cite{BIRS}. 
 As explained in \cite{BKL},
  the composition of the canonical functors
  \[
  \coh\XX \hookrightarrow    \Dcoh \twoheadrightarrow\cC_\XX
  \]
   allows us to think of $\coh\XX$ as a non-full subcategory of  $ \cC_\XX$  that shares the same isomorphism classes of indecomposable objects.

An object $E$ in $\cC_\XX$ is called \emph{rigid} if $\Hom_\cC(E,E[1])=0$. 
 A straightforward calculation shows that 
the rigid objects in $\cC_\XX$ coincide with those in  $\coh\XX$.
Recall from \cite{BMRRT} that a \emph{cluster-tilting} object $T$ in $\cC_\XX$  is an object satisfying $\Hom_\cC(T,X[1])=0$ if and only if $X\in {\rm add }\ T$.
Let $T=\overline{T}\oplus X$ and $T^{\prime}=\overline{T}\oplus X^{\prime}$ be two cluster-tilting objects in $\cC_\XX$, where $X$ and $X^{\prime}$ are two non-isomorphic indecomposable objects in $\cC_\XX$. Then $T^{\prime}$ is called the \emph{mutation} of $T$ at $X$.
 The \emph{exchange
graph for cluster tilting objects} in $\cC_\XX$ has as vertices the isomorphism classes of  cluster-tilting objects in $\cC_\XX$, while two vertices are connected by an edge if and only if the associated cluster-tilting objects differ by precisely one indecomposable direct summand.


\subsection{Marked surfaces and mapping class groups}
In this subsection, we review some background material on marked surfaces and their mapping class groups. For further details, we refer the reader to \cite{FST,DL,FM}.
\begin{definition} 
    A \emph{marked surface with punctures} $\surf= (\surf,\M,\P) $ is   a compact connected oriented  surface $ \surf $ with 
    \begin{itemize}
        \item   a finite set $ \M \subset \partial \surf$ of   \emph{marked points} such that  each boundary component contains at least one marked point;
        \item   a finite set $ \P  $ of  \emph{punctures} in  $\surf \backslash \partial \surf$.
    \end{itemize}
\end{definition}

\begin{definition}  A \emph{curve} on $\surf$ is an immersion $\gamma :[0,1]\to \surf$.  A \emph{closed curve} is one with  $\gamma(0)=\gamma(1)$, and a  \emph{simple curve} is one where $\gamma$  is injective, except possibly at endpoints.  
\end{definition}

\begin{definition}
   A \emph{simple  arc} $ \gamma $ in $ \surf $ is a simple curve   such that  
	\begin{itemize}
	     \item the endpoints  of $\gamma  $ lie in $ \M\cup \P $;
        \item  $ \gamma $ is disjoint from $\M$ and from the boundary of $\surf$, except for the  endpoints;
        \item  $ \gamma $ does not cut out an unpunctured monogon or an unpunctured digon. 
	\end{itemize}
\end{definition}
In this paper, we consider curves and  simple arcs up to homotopy.
 Denote by  $\A(\surf) $ 
     the subset consisting of simple arcs with distinct endpoints in $\surf $.

\begin{definition} \label{tarc}
A \emph{tagged arc} on $\surf$ is a pair $(\gamma,\kappa)$ of an  arc $\gamma\in\A(\surf)$ and a \emph{tagged function} 
$$\kappa:\{\gamma(t)\in\P\mid  t\in \{0,1\}\}\to\{-1, 1\}.$$
Denote by $\TA^\times(\surf)$ the set of  tagged arcs on $\surf$.
\end{definition}

 \begin{definition}[{\cite[Definition~7.4]{FST}}]
    For two tagged arcs $(\gamma,\kappa)$  and $(\beta,\kappa')$ in  $\TA^\times(\surf)$, they are \emph{compatible} if one of the following holds:
\begin{enumerate}
    \item $\beta \neq \gamma$, and   $\kappa(p) = \kappa'(p)$ for every puncture $p \in \{\gamma(0), \gamma(1)\} \cap \{\beta(0), \beta(1)\}$.
    \item $\beta = \gamma$, and there exists an endpoint $p \in \{\gamma(0), \gamma(1)\}$ such that $\kappa(p) = \kappa'(p)$.
\end{enumerate} 
A \emph{tagged triangulation} is   a maximal collection of pairwise compatible tagged arcs. 
 \end{definition}

\begin{definition}[{\cite[Section~7]{FST}}]
 For a tagged triangulation $\mathcal{T}$ of $\surf$ and each tagged arc $(\gamma,\kappa)$ of $\mathcal{T}$, 
   a \emph{flip of $\mathcal{T}$    along $(\gamma,\kappa)$}  is a transformation of  $\mathcal{T}$  that  replaces  $(\gamma,\kappa)$ with  a (unique) different tagged arc $(\beta,\kappa')$ such that $\mu_{(\gamma,\kappa)}(\mathcal{T})=\mathcal{T}\setminus\{(\gamma,\kappa)\}\cup\{(\beta,\kappa')\} $ is again a tagged triangulation. The \emph{exchange graph} of $\surf$ has as vertices the triangulations of $\surf$, with an edge between two vertices $\mathcal{T}$ and $\mathcal{T}^{\prime}$ whenever $\mathcal{T}^{\prime}$ is obtained from $\mathcal{T}$ by the flip of a tagged arc.  
\end{definition}

\begin{definition}\label{defmcg}
	The \emph{mapping class group} for a  marked surface  $ \surf $ is defined by
	\[ \MCG(\surf) \coloneqq \text{\rm Homeo}^{+}(\surf)/\text{\rm Homeo}^{+}_{0}(\surf) ,\]
	where $ \text{\rm Homeo}^{+}(\surf) $ is the group of orientation-preserving homeomorphisms such that $ g(\M\cup \P)=\M\cup \P $ for any  $ g \in  \text{\rm Homeo}^{+}(\surf) $, and $ \text{\rm Homeo}^{+}_{0}(\surf) $ is its subgroup of homeomorphisms homotopic to the identity.
\end{definition}

\begin{definition}[{\cite[Section~3.1]{BQ}}]\label{tagmcg}
The \emph{tagged mapping class group} of $\surf$ is defined to be
\[
\MCG^\times(\surf)=\MCG(\surf)\ltimes_1 \{\pm1\}^{\P}
\]
such that
\[
(g_1,\delta_1)*(g_2,\delta_2)
\coloneqq(g_1g_2,(g_2\cdot\delta_1)\delta_2),\quad
(g_1,\delta_1),(g_2,\delta_2)\in\MCG^\times(\surf).
\]
Here $\{\pm1\}^{\P}$ is the group of maps $\delta:\P\to\{\pm1\}$ with pointwise multiplication, and $(g\cdot\delta)(p)=\delta(g^{-1}(p))$.
\end{definition}

\subsection{Graded marked surfaces with binaries}\label{sec:gms}
 In this subsection, we recall 
 some concepts of graded marked surfaces with binaries
 from \cite{QZZ,HKK,LP,IQZ}.
\begin{definition}
   A \emph{grading} $\grad$ on $\surf$ is a homotopy class of sections
\[
\grad:\surfi\to \mathbb{P}T(\surfi),
\]
where $\surfi=\surf\setminus\partial\surf$ and $\mathbb{P}T(\surfi)$ is the projectivized tangent bundle. 
\end{definition}
Choose a representative of this homotopy class, still denoted by $\grad$. Let
\[
\operatorname{cov}\colon \mathbb{R}T(\surfi) \longrightarrow \mathbb{P}T(\surfi)
\]
be the associated $\ZZ$-cover: for $\ell\in\mathbb{P}T_a(\surfi)$, the fiber
$\operatorname{cov}^{-1}(\ell)$ consists of homotopy classes of paths in
$\mathbb{P}T_a(\surfi)$ from $\grad(a)$ to $\ell$.

\begin{definition}
    A \emph{graded marked surface $\gms$ with punctures}  is a pair of a marked surface $\surf$ with punctures and a grading $\grad$ on it.
\end{definition}
We now recall how the grading of $\gms$ induces gradings on curves.
 
\begin{definition}\label{def:gcurve} 
A \emph{grading} $\widetilde{\gamma}$ on a curve $\gamma$ is given by a homotopy class of paths in $\mathbb{P}T_{\gamma(t)}(\surfi)$ from $\grad(\gamma(t))$ to $\dot{\gamma}(t)$, varying continuously with $0\leq t\leq 1$, where $\dot{\gamma}(t)$ is the tangent of $ {\gamma}(t)$. The pair $(\gamma,\widetilde{\gamma})$   is called a {\it graded curve},  also denoted simply by $\widetilde{\gamma}$.  
When the underlying curve $\gamma$  is a simple arc, we call  
 $\widetilde{\gamma}$ 
 a \emph{graded simple arc}.

For any graded curve $\widetilde{\gamma}$ and any $\rho\in\ZZ$, denote by $\widetilde{\gamma}[\rho]$ the graded curve whose underlying curve is the same as $\widetilde{\gamma}$ and whose grading is the composition of $\widetilde{\gamma}(t):\grad(\gamma(t))\to\dot{\gamma}(t)$ and the path from $\dot{\gamma}(t)$ to itself given by clockwise rotation by $\rho\pi$.

Additionally, we  define $\widetilde{\gamma}^0$  as the graded curve sharing the same underlying curve as $\widetilde{\gamma}$, obtained by taking the unique 0-lifts  of the tangents $\dot{\gamma}(t),0\leq t\leq 1$  under the covering
$\operatorname{cov}\colon \mathbb{R}T(\surfi) \longrightarrow \mathbb{P}T(\surfi)$.

\end{definition}

\begin{definition}[{\cite[Section~2.1]{HKK}}]\label{def:gind}
Let $\widetilde{\sigma}$ and $\widetilde{\tau}$ be graded simple arcs on $\gms$, with underlying arcs $\sigma$ and $\tau$. Let $a\in\surf$ be a transverse intersection point, say $\sigma(s)=\tau(t)=a$, and $\dot{\sigma}(s)\ne\dot{\tau}(t)$ in $\mathbb{P}T_a(\surfi)$. Let $\lambda_a$ be the path in $\mathbb{P}T_a(\surfi)$ from $\dot{\sigma}(s)$ to $\dot{\tau}(t)$ given by counterclockwise rotation through an angle strictly less than $\pi$. The \emph{intersection index} of $\widetilde{\sigma}$ and $\widetilde{\tau}$ at $a$ is defined by
\begin{equation}\label{eq:gind}
\gind_a(\widetilde{\sigma},\widetilde{\tau})=\widetilde{\sigma}(s)\cdot\lambda_a\cdot\widetilde{\tau}(t)^{-1}\in\pi_1(\mathbb{P}T_a(\surfi),\grad(a))\cong\ZZ.
\end{equation}
\end{definition}

\begin{remark}\label{rem:gind}
For a transverse intersection point $a$, the index satisfies
\begin{equation}\label{eq:gind-sym}
\gind_a(\widetilde{\sigma},\widetilde{\tau})+\gind_a(\widetilde{\tau},\widetilde{\sigma})=1,
\end{equation}
and, for any $m,n\in\ZZ$,
\begin{equation}\label{eq:gind-shift}
\gind_a(\widetilde{\sigma}[m],\widetilde{\tau}[n])=\gind_a(\widetilde{\sigma},\widetilde{\tau})+m-n.
\end{equation}
\end{remark}

\begin{definition}\label{def:mcgofgms}
The \emph{graded mapping class group} $\MCG(\gms)$ of a graded marked surface $\gms$ is the group of isotopy classes of pairs $(f,\widetilde f)$, where $f\colon \surf \to \surf$ is an orientation-preserving homeomorphism preserving the marked points and punctures setwise such that $f$ is \emph{compatible with the grading}, namely $f^*\eta$ is homotopic to $\eta$ as sections of $\mathbb{P}T(\surfi)$, and $\widetilde f$ is a lift of $df$ such that the following diagram commutes:
\begin{equation}\label{eq:cov}
\begin{tikzcd}
\mathbb{R}T(\surfi) \arrow[r, "\widetilde f"] \arrow[d, "\operatorname{cov}"']
& \mathbb{R}T(\surfi) \arrow[d, "\operatorname{cov}"] \\
\mathbb{P}T(\surfi) \arrow[r, "df"]
& \mathbb{P}T(\surfi) .
\end{tikzcd}
\end{equation}
\end{definition}
By definition, for any graded arc $\widetilde{\gamma}$ in $\gms$, we have that
\[(f,\widetilde f)(\widetilde{\gamma})[1]=(f,\widetilde f)(\widetilde{\gamma}[1]).\]
For any graded simple arc $\widetilde{\sigma}$ in $\gms$, the associated \emph{braid twist} $B_{\widetilde{\sigma}}\in\MCG(\gms)$ is defined in \Cref{fig:bt}. We have the formula
\begin{equation}\label{eq:bt}
B_{\Psi(\widetilde{\sigma})}=\Psi\circ B_{\widetilde{\sigma}}\circ\Psi^{-1}
\end{equation}
for any $\Psi\in\MCG(\gms)$. 
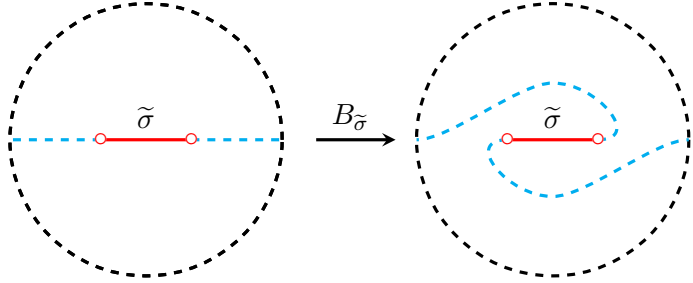
\begin{figure}[ht]\centering
	\begin{tikzpicture}[scale=.3]
    \draw[very thick, dashed](0,0)circle(6)node[above,black]{$\widetilde{\sigma}$};
	\draw[very thick,dashed](0,0)circle(6);
	\draw(-2,0)edge[red, very thick](2,0)  edge[cyan,very thick, dashed](-6,0);
	\draw(2,0)edge[cyan,very thick,dashed](6,0);
	\draw(-2,0)node[white] {$\bullet$} node[red] {$\circ$};
	\draw(2,0)node[white] {$\bullet$} node[red] {$\circ$};
	\draw(0:7.5)edge[very thick,-stealth](0:11);\draw(0:9)node[above]{$B_{\widetilde{\sigma}}$};
	\end{tikzpicture}\;
	\begin{tikzpicture}[scale=.3,yscale=-1]
	\draw[very thick, dashed](0,0)circle(6)node[above,black]{$\widetilde{\sigma}$};
	\draw[red, very thick](-2,0)to(2,0);
	\draw[cyan,very thick, dashed](2,0).. controls +(0:2) and +(0:2) ..(0,-2.5)
	.. controls +(180:1.5) and +(0:1.5) ..(-6,0);
	\draw[cyan,very thick,dashed](-2,0).. controls +(180:2) and +(180:2) ..(0,2.5)
	.. controls +(0:1.5) and +(180:1.5) ..(6,0);
	\draw(-2,0)node[white] {$\bullet$} node[red] {$\circ$};
	\draw(2,0)node[white] {$\bullet$} node[red] {$\circ$};
	\end{tikzpicture}
	\caption{The braid twist}
	\label{fig:bt}
\end{figure}
The action of a braid twist on graded arcs can be described by smoothing. Let $\widetilde{\sigma}$ and $\widetilde{\tau}$ be graded simple arcs in $\gms$ with $\widetilde{\sigma}(0)=\widetilde{\tau}(0)=a$. The arc $B_{\widetilde{\sigma}}(\widetilde{\tau})$ is the extension $\widetilde{\tau}\wedge\widetilde{\sigma}$ of $\widetilde{\tau}$ by $\widetilde{\sigma}$ with respect to the common starting point. This is the smoothing shown in \Cref{fig:ext.}. More precisely, if $\gind_a(\widetilde{\sigma},\widetilde{\tau})=n$, then $\widetilde{\tau}\wedge\widetilde{\sigma}$ is obtained by smoothing $\widetilde{\tau}\cup\widetilde{\sigma}[1-n]$ at $a$, with the grading inherited from $\widetilde{\tau}$. Thus $\gind_b(\widetilde{\tau},\widetilde{\tau}\wedge\widetilde{\sigma})=0$ with $\widetilde{\sigma}(0)=(\widetilde{\tau}\wedge\widetilde{\sigma})(0)=b$, cf. \Cref{fig:ext.}. Equivalently,
$$B_{\widetilde{\tau}}(\widetilde{\tau}\wedge\widetilde{\sigma})=\widetilde{\sigma}[1-n],\qquad B_{\widetilde{\tau}}^{-1}(\widetilde{\sigma}[1-n])=\widetilde{\tau}\wedge\widetilde{\sigma}.$$
Note that shifting the grading does not change the associated braid twist:
$B_{\widetilde{\sigma}[m]}=B_{\widetilde{\sigma}}$ for all $m\in\ZZ$.
	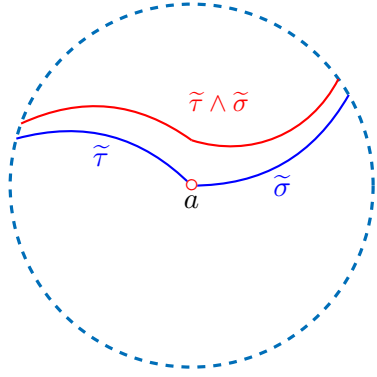
\begin{figure}[ht]\centering
		\begin{tikzpicture}[scale=1.2]
		\draw[NavyBlue,dashed,very thick](0,0)circle(2);
		\draw[thick,blue](-195:2)edge[bend left,>=stealth](0,0)
		(0,0)edge[bend right,>=stealth](30:2)
		(1,0)node{$\widetilde{\sigma}$}(-1,.1)node[above]{$\widetilde{\tau}$};
		\draw[thick,red](160:2)to[bend left](0,.5)(0.3,.7)
		node[above]{\small{$\widetilde{\tau}\wedge\widetilde{\sigma}$}};
		\draw[thick,red](0,0.5)edge[bend right=40,>=stealth](36:2);
		\draw(0,0)node[white] {$\bullet$} node[red](a){$\circ$} (0,0)node[below]{$a$};
		\end{tikzpicture}
		\caption{The extension as smoothing out}
		\label{fig:ext.}
	\end{figure}
\begin{lemma}\label{lem:1.13}
Let $\surf $ be the marked surface underlying $\gms$. The forgetful map $F:\gms\to \surf $ induces a group homomorphism
\[
F_*:\MCG(\gms)\longrightarrow \MCG(\surf),\qquad [(f,\widetilde f)]\longmapsto [f],
\]
and
\[
\ker F_*=\langle [1]\rangle\cong\ZZ.
\]
Consequently, there is a short exact sequence
\begin{equation}\label{eq:ses}
1\longrightarrow \ZZ \xrightarrow{\;\iota\;} \MCG(\gms)
\xrightarrow{\;F_*\;} \operatorname{Im}(F_*) \longrightarrow 1,
\end{equation}
where $\iota(n)=[n]$.
\end{lemma}

\begin{proof}
The map $F_*$ is obtained by forgetting the lift $\widetilde f$. If $[(f,\widetilde f)]\in\ker F_*$, then after changing the representative we may assume $f=\id_{\surf}$. Hence $\widetilde f$ is a deck transformation of the covering $\operatorname{cov}\colon\mathbb{R}T(\surfi)\to\mathbb{P}T(\surfi)$, so the class of $(f,\widetilde f)$ is $[n]$ for some $n\in\ZZ$. Conversely, every shift $[n]$ acts trivially on the underlying marked surface. Thus $\ker F_*=\langle[1]\rangle\cong\ZZ$, and the exact sequence follows.
\end{proof}

\begin{construction}\label{cons:binary} 
We construct a \emph{graded marked surface with binaries}  $\gmsx$ from $\gms$ by replacing each puncture $P\in\P$ by a boundary component $\sx_P$,
called a \emph{binary}, with one  marked point $m_P$  on it.  Denote by $\wA (\gmsx)$ the set of all  graded simple arcs on   $\gmsx$ with distinct endpoints,  and 
 $\vot$ 
 the set of all binaries $\sx_P$, $P\in\P$.
\end{construction}

\begin{remark}
By the construction, $\gmsx$ can be regarded as a graded marked surface without punctures $(\surf\x,\M\x,\grad)$, with $\M\x=\M\cup\M_\P$, where $\M_\P=\{m_P\mid P\in\P \}$. 
\end{remark}

Let $\mathrm{D}_{\vot}^{k}$ be the subgroup of the  usual mapping class group  of $\gmsx$  generated by the  $k$-th powers of Dehn twists $\mathrm{D}_{\sx}$ along all binaries $\sx\in\vot$, where $k\in\ZZ_{\geq0}$.   The binary mapping class group $\MCG^\sx(\gmsx)$ of  $\gmsx$ is
the quotient of the usual one by  $\mathrm{D}_{\vot}^{2}$.
\begin{definition}\label{defn:orbit}  
    The \emph{$\mathrm{D}_{\vot}^{k}$-orbit} $\mathrm{D}_{\vot}^{k}\cdot\widetilde{\sigma}$ of a graded simple arc $\widetilde{\sigma}\in\wA (\gmsx)$ consists of the graded simple arcs   obtained from $\widetilde{\sigma}$ by actions of  $\mathrm{D}_{\vot}^{k}$ on the ends (which are in binaries) separately.
\end{definition} 
\begin{definition}\label{wfunction}
Let $\widetilde{\sigma}$ be a
  graded simple arc   on $\gmsx$. The \emph {winding function} of  $\widetilde{\sigma}$   is the map
\[
f_{\widetilde{\sigma}} \colon \{\sigma(0), \sigma(1)\} \to \{0,1\}
\]  such that $f_{\widetilde{\sigma}}(\sigma(t))$ is equal to $w\ \mod 2$, where $w$ is the  number of clockwise winds that  $ {\sigma}$ makes around the binary $\sx_{\sigma(t)}$.
\end{definition}

 
\section{Coherent sheaves on weighted projective line of type (2,2,2,2)}\label{sec:3}

 In the remainder of this paper,  we focus on  the tubular weighted projective line   
 \[\CPone\coloneqq(\mathbb{P}^1,(0,1,\infty,\lambda),(2,2,2,2))\]
  over the   field  $\CC$ with $\lambda\in \CC \setminus\{0,1\}$.
A connection between  rigid sheaves in  $\coh(\CPone)$ and  the positive Schur roots associated to  $\CPone$
 was established in \cite{BG}. 
Let   $\Ind^\circ \coh(\CPone)$  denote the set of isomorphism classes of indecomposable rigid sheaves in $\coh(\CPone)$.
\begin{proposition}[{\cite[Lemma~4.2 and  Proposition~4.3]{BG}}]\label{prop:rigid-to-vectors} 
The map $E\mapsto  [E]$ induces a bijection from  
$\Ind^\circ \coh(\CPone)$ to the set of positive real Schur roots  $\{\bv_p^x \mid p\in\QQi, x \in\bfH\}$ in the Grothendieck group   of $\coh(\CPone)$, where $\bfH$ is  the quaternion group
$$
 \bfH=\{\pm 1,\pm i,\pm j,\pm k\}
$$
with  multiplication rules  $ij=k=-ji$, $jk=i=-kj$ and $ki=j=-ik$.
Moreover, the map $(p,x)\mapsto E_{p}^{x} $  defines a bijection from $\QQi\times \bfH$ to   $\Ind^\circ \coh(\CPone)$, where $E_{p}^{x}$ is the unique (up to isomorphism) rigid sheaf satisfying $  [E^x_p] =\bv_p^x$ and  $  [\tau E^x_p] =\bv_p^{-x}$.
\end{proposition}

\begin{convention}\label{convention:roots}
We assign a positive real Schur root to  rigid object $E$ whose slope satisfies $\mu E\in\{0,1,\infty\}$, according to the following correspondence:
\begin{table}[htbp]
  \centering
  \caption{Correspondence between rigid objects and positive   Schur roots}
  \label{tab:schur-roots}
   \renewcommand{\arraystretch}{1.2}
\setlength{\tabcolsep}{12pt}
  \begin{tabular}{cl} 
    \toprule
    $E$ & $[E]$ \\
    \midrule
    $\mathcal{O}$ & $\mathbf{v}_0^1$ \\
    $\mathcal{O}(\vec{x}_2 - \vec{x}_4)$ & $\mathbf{v}_0^i$ \\
    $\mathcal{O}(\vec{x}_2 - \vec{x}_3)$ & $\mathbf{v}_0^j$ \\
    $\mathcal{O}(\vec{x}_3 - \vec{x}_4)$ & $\mathbf{v}_0^k$ \\
    \bottomrule
  \end{tabular}
  \quad 
  \begin{tabular}{cl}
    \toprule
    $E$ & $[E]$ \\
    \midrule
    $\mathcal{O}(\vec{x}_1)$ & $\mathbf{v}_1^1$ \\
    $\mathcal{O}(\vec{x}_2)$ & $\mathbf{v}_1^j$ \\
    $\mathcal{O}(\vec{x}_3)$ & $\mathbf{v}_1^k$ \\
    $\mathcal{O}(\vec{x}_4)$ & $\mathbf{v}_1^i$ \\
    \bottomrule
  \end{tabular}
  \quad
  \begin{tabular}{cl}
    \toprule
    $E$ & $[E]$ \\
    \midrule
    $S_{0,0}$ & $\mathbf{v}_\infty^{-1}$ \\
    $S_{1,0}$ & $\mathbf{v}_\infty^i$ \\
    $S_{\infty,0}$ & $\mathbf{v}_\infty^j$ \\
    $S_{\lambda,0}$ & $\mathbf{v}_\infty^k$ \\
    \bottomrule
  \end{tabular}
  
\end{table}

\end{convention}

To compute the Euler form between rigid sheaves in $\coh(\CPone)$, we first recall some definitions from \cite{BG}.

\begin{definition}
     For  any $p \in \QQi$,   fix a coprime  integer pair $(a(p), b(p)) $ 
such that $p =\frac{a(p)}{b(p)}$ and $b(p)\geq 0$. 
Define
\[
\bfh_{p}=b(p)\bfh_0+a(p)\bfh_{\infty},
\]
     where $\mathbf{h}_0=[\mathcal{O}(\vx)]+[\mathcal{O}(\vx+\vec{\omega})]$  and $\mathbf{h}_{\infty}=[S_{\mu}]$  with $\vx\in \mathbb{L}$ and $\mu\in\mathbb{P}^1\setminus\{0,1,\infty,\lambda\}$.
     
The   type of $p$ is  defined by
$$\typ{p}= 
\begin{cases}
0, & \text{if } a(p) \equiv 0,\ b(p) \equiv 1  \quad \mod 2, \\
1, & \text{if } a(p) \equiv 1,\ b(p) \equiv 1 \quad  \mod 2, \\
\infty, & \text{if } a(p) \equiv 1,\ b(p) \equiv 0 \quad \mod 2.
\end{cases}
$$  
\end{definition} 

For any $p\in\QQi$ and $x\in\bfH$,  we have the decomposition
\[
\bv_{p}^{x}=\bv_{\typ{p}}^{x} +\lfloor\frac{b(p)}{2}\rfloor \mathbf{h}_0+\lfloor\frac{a(p)}{2}\rfloor \mathbf{h}_{\infty} \text{ and }\bfh_p=\bv_{p}^{x}+\bv_{p}^{-x},
\]
where $\lfloor r \rfloor$ denotes the integer part of $r\in \QQi $.
For any $ p, q \in \QQi$, let
$$\Delta(p,q)
=a(q)b(p)-a(p)b(q).$$ 
Then the Euler form can be explicitly computed as follows.

 \begin{proposition}\label{hom dim} The following identities  hold for all 
  $p, q\in\QQi$ and $x,y,h\in\bfH$
 \begin{equation}
    \label{eq:formula1}
     \<\bv_0^x,\bv_1^{hx}\>=
   \<\bv_1^x,\bv_\infty^{hx}\>=
    \<\bv_0^{-hx}, \bv_\infty^{x}\>=
      \begin{cases}
        1,&\text{if $h\in \{1,i,j,k\}$},\\
        0,&\text{otherwise};
      \end{cases}
  \end{equation}   
   \begin{equation}
    \label{eq:formula2}
     \<\bv_1^{hx},\bv_0^x\>=
   \<\bv_\infty^{hx},\bv_1^x\>=
    \<\bv_\infty^{x},\bv_0^{-hx}\>=
      \begin{cases}
        0,&\text{if $h\in \{1,i,j,k\}$},\\
        -1,&\text{otherwise};
      \end{cases}
  \end{equation}
and 
   \begin{equation}
    \label{eq:formula3}
    \langle \bv_{p}^x, \bv_{q}^y\rangle
=\langle \bv_{\typ{p}}^x, \bv_{\typ{q}}^y\rangle+\frac{1}{2}\Delta(p,q)-\frac{1}{2}\Delta(\typ{p},\typ{q}).
\end{equation}
\end{proposition}
\begin{proof} 
Equation~\eqref{eq:formula1}   is proved in \cite[Lemma 4.1]{BG}.
By Serre duality  and \cite[Lemma~2.8 and Proposition~2.9]{BG}, we have
 \[
  \<\bv_p^{-x},\bv_q^{-y}\>=\<\bv_p^x,\bv_q^y\>=-\<\bv_q^y,\bv_p^{-x}  \>.
 \]
   Combining this with the identities 
$$\<\bv_0^x,\bh_\infty\>=\<\bv_1^x,\bh_\infty\>=
\< \bh_0,\bv_1^x\>=\<\bh_0,\bv_\infty^x\>=1,$$  which also follow from \cite[Lemma 4.1]{BG}, we obtain that for  $a,b\in\{0,1,\infty\}$ and $a<b$,
\[
\<\bv_b^y,\bv_a^x\>=\<\bv_a^x,\bv_b^y\>-1.
\]
  This implies that Equation~\eqref{eq:formula2} holds.
The proof of Equation~\eqref{eq:formula3}  is considered case by case according to  the types $\typ{p},\typ{q}\in\{0,1,\infty\}$ respectively.

First suppose $\typ{p}=\typ{q}=0$. In this case we have $\Delta(\typ{p},\typ{q})=0$. Moreover, by  the proof of  \cite[Proposition 4.7]{BG},
$$\langle \bv_{p}^x, \bv_{q}^y\rangle=\langle \bv_{0}^x, \bv_{0}^y\rangle  +\frac{1}{2}\Delta(p,q),
$$
which agrees with the formula.
If $\typ{p}=1$ and $\typ{q}=0$, then  $\Delta(\typ{p},\typ{q})=-1$. A similar computation shows
$$\langle \bv_{p}^x, \bv_{q}^y\rangle=\langle \bv_{1}^x, \bv_{0}^y\rangle   +\frac{1}{2}\Delta(p,q)+\frac{1}{2}.
$$   The remaining cases are verified analogously.
\end{proof}

\begin{corollary}
For any $p, q\in\QQi$ and $x,y\in\bfH$,
$$\langle \bv_{p}^x, \bv_{q}^y\rangle
+\langle \bv_{p}^x, \bv_{q}^{-y}\rangle
=\Delta(p,q)\quad \text{and} \quad
|\langle \bv_{p}^x, \bv_{q}^y\rangle
-\langle \bv_{p}^x, \bv_{q}^{-y}\rangle|
\leq 2,$$
where $| n|$ denotes the absolute  value of $n$.
\end{corollary}
\begin{proof} According to Proposition \ref{hom dim}, we have
\begin{align*}&\langle \bv_{p}^x, \bv_{q}^y\rangle
+\langle \bv_{p}^x, \bv_{q}^{-y}\rangle
-\Delta(p,q)\\
=&(\langle \bv_{p}^x, \bv_{q}^y\rangle-\frac{1}{2}\Delta(p,q))
+(\langle \bv_{p}^x, \bv_{q}^{-y}\rangle
-\frac{1}{2}\Delta(p,q))\\
=&\langle (\bv_{\typ{p}}^x, \bv_{\typ{q}}^y\rangle-\frac{1}{2}\Delta(\typ{p},\typ{q}))+(\langle \bv_{\typ{p}}^x, \bv_{\typ{q}}^{-y}\rangle-\frac{1}{2}\Delta(\typ{p},\typ{q}))\\
=&\langle \bv_{\typ{p}}^x, \bv_{\typ{q}}^y+\bv_{\typ{q}}^{-y}\rangle-\Delta(\typ{p},\typ{q})\\
=&\langle \bv_{\typ{p}}^x, \bfh_{\typ{q}}\rangle-\Delta(\typ{p},\typ{q})\\
=&\langle \bv_{\typ{p}}^x, b(\typ{q})\bfh_{0}+a(\typ{q})\bfh_{\infty}\rangle-\Delta(\typ{p},\typ{q}). 
\end{align*} 
Note that $\langle \bfh_{0},  \bv_{p}^x\rangle=a(p)$ and $\langle \bv_{p}^x, \bfh_{\infty}\rangle=b(p)$ for any $p\in\QQi$ and $x\in\bfH$. Then
\begin{align*}&\langle \bv_{\typ{p}}^x, b(\typ{q})\bfh_{0}+a(\typ{q})\bfh_{\infty}\rangle\\
 =&b(\typ{q})\langle \bv_{\typ{p}}^x, \bfh_{0}\rangle
+a(\typ{q})\langle \bv_{\typ{p}}^x, \bfh_{\infty}\rangle\\
 =&-b(\typ{q})a(\typ{p})+a(\typ{q})b(\typ{p})\\
 =&\Delta(\typ{p},\typ{q}).
\end{align*}

It follows that $\langle \bv_{p}^x, \bv_{q}^y\rangle
+\langle \bv_{p}^x, \bv_{q}^{-y}\rangle
=\Delta(p,q)$.
Besides, also by Proposition~\ref{hom dim},
\[
|\langle \bv_{p}^x, \bv_{q}^y\rangle
-\langle \bv_{p}^x, \bv_{q}^{-y}\rangle|
=|\langle \bv_{\typ{p}}^x, \bv_{\typ{q}}^y\rangle
-\langle \bv_{\typ{p}}^x, \bv_{\typ{q}}^{-y}\rangle|
\leq 2. 
\qedhere
\]
\end{proof}

\section{A surface model  for the rigid  objects in derived category $\DC{\CPone}$}\label{sec:4}

 In this section, we  characterize indecomposable  rigid  objects in $\DC{\CPone}$ via graded simple arcs on a sphere with four binaries, and then
  the dimensions of Hom-spaces between two
 rigid  objects are identified with 
 the oriented intersection numbers between the corresponding  graded simple arcs.

\subsection{Simple arcs in the sphere with four punctures} 

Let $S_{0,4}$ denote the sphere with four punctures.
There is a  close geometric relation between $S_{0,4}$ and the torus  $T^2$.  Indeed, we realize $ T^2$  as the unit square $[0,1]^2$ with opposite sides identified. The hyperelliptic involution $\iota$  is the $ \pi$-rotation about the square’s center. This map has four fixed points:
\[
\{( \frac{1}{2}, \frac{1}{2} ), ( \frac{1}{2}, 0 ) \sim ( \frac{1}{2}, 1 ),( 0, \frac{1}{2}) \sim ( 1, \frac{1}{2} ),(0,0) \sim (0,1) \sim (1,0) \sim (1,1)\},
\]
  where the notation $\sim$ indicates identifications under the boundary gluings of the square.  
Thus the quotient by $\iota$  is topologically a sphere with four distinguished points, which we identify with  $S_{0,4}$, see \Cref{fig:sphere}. 
 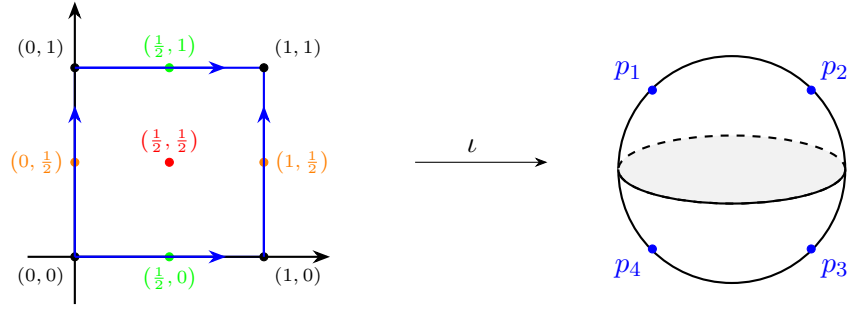
\begin{figure}[h]
     \centering
      \begin{tikzpicture}[scale=2.5, >=Stealth]
    \draw[thick,blue] (0,0) rectangle (1,1);
    \node[below left, font=\tiny] at (0,0) {$(0,0)$};
    \node[below right, font=\tiny] at (1,0) {$(1,0)$};
    \node[above left, font=\tiny] at (0,1) {$(0,1)$};
    \node[above right, font=\tiny] at (1,1) {$(1,1)$};
    \filldraw[red] (0.5,0.5) circle (0.6pt) node[above  , font=\tiny] {$\left(\frac{1}{2},\frac{1}{2}\right)$};
    \filldraw[green] (0.5,1) circle (0.6pt) node[above, font=\tiny] {$\left(\frac{1}{2},1\right)$};
    \filldraw[orange] (1,0.5) circle (0.6pt) node[right, font=\tiny] {$\left(1,\frac{1}{2}\right)$};
    \filldraw[green, font=\tiny] (0.5,0) circle (0.6pt) node[below] {$\left(\frac{1}{2},0\right)$};
    \filldraw[orange] (0,0.5) circle (0.6pt) node[left, font=\tiny] {$\left(0,\frac{1}{2}\right)$};
    \filldraw[black] (0,0) circle (0.6pt);
    \filldraw[black] (1,0) circle (0.6pt);
    \filldraw[black] (0,1) circle (0.6pt);
    \filldraw[black] (1,1) circle (0.6pt);
\draw[->, thick] (0,-.25) --  (0,1.35);
 \draw[->, thick] ( -.25,0) --  (1.35,0);
 \draw[-> ,  thick, blue ] (0,0) --  (0,0.8);
 \draw[-> ,  thick,blue ] (1,0) --  (1,0.8);
  \draw[-> ,  thick, blue ] (0,1) --  ( 0.8,1);
 \draw[-> ,  thick,blue ] (0,0) --  ( 0.8,0);
  \draw[-> ] (1.8,.5) --  ( 2.5,.5);
      \node[above] at (2.1,0.5) {$\iota$};
\end{tikzpicture}  
\hspace{0.5cm}
\begin{tikzpicture}[scale=1.5]

\draw[thick] (0,0) circle (1cm);

\draw[thick, dashed, fill=gray!10] (0,0) ellipse (1cm and 0.3cm);

\draw[thick] (-1,0) arc (180:360:1cm and 0.3cm);

\filldraw[blue] (0.7,0.7) circle (1pt) node[above right] {$p_2$};
\filldraw[blue] (-0.7,0.7) circle (1pt) node[above left] {$p_1$};

\filldraw[blue] (0.7,-0.7) circle (1pt) node[below right] {$p_3$};
\filldraw[blue] (-0.7,-0.7) circle (1pt) node[below left] {$p_4$};
 \node[above] at (0,-1.2) {}; 
\end{tikzpicture}
  \caption{The sphere with four punctures $p_1$, $p_2$, $p_3$ and  $p_4$}
     \label{fig:sphere}
 \end{figure}

 To better describe the homotopy classes of simple arcs on  $S_{0,4}$, we first recall the classification of simple closed curves on  $T^2$.

 \begin{Fact}
The covering $\tilde{p} \colon\mathbb{R}^2\to T^2$ identifies $\pi_1(T^2)$ with $\mathbb{Z}^2$. Thus each nontrivial simple closed curve on $T^2$ is represented by a primitive vector $(b,a)\in\mathbb{Z}^2$, and after forgetting the orientation it is determined by the slope $a/b\in\QQi$.
\end{Fact}  

\begin{construction}  \label{arcs pair}
Let $\tilde{\pi} \colon\mathbb{R}^2 \to S_{0,4} $  be the composition of the universal covering map  $\tilde{p} \colon\mathbb{R}^2 \to T^2$  and the quotient map   $T^2 \to T^2 / \iota \cong S_{0,4}$ induced by the hyperelliptic involution $\iota$.
 For each $p \in \QQi$, 
 we define a pair of simple arcs  $\tilde{\pi}\circ \overline{\alpha}_p^-$ and $\tilde{\pi}\circ \overline{\alpha}_p^+$, where  
\[
\overline{\alpha}_p^- : [0,1] \to \mathbb{R}^2, \quad t \mapsto
\begin{cases}
\left( \dfrac{b(p)}{2}t, \dfrac{a(p)}{2}t \right), & \text{if } b(p) \text{ is odd}, \\
\left( \dfrac{b(p)}{2}t+\dfrac{1}{2}, \dfrac{a(p)}{2}t \right), & \text{if } b(p) \text{ is even},
\end{cases}
\]
\[
\overline{\alpha}_p^+ : [0,1] \to \mathbb{R}^2, \quad t \mapsto
\begin{cases}
\left( \dfrac{b(p)}{2}t, \dfrac{a(p)}{2}t+\dfrac{1}{2} \right), & \text{if } b(p) \text{ is odd}, \\
\left( \dfrac{b(p)}{2}t, \dfrac{a(p)}{2}t \right), & \text{if } b(p) \text{ is even}.
\end{cases}
\] 
We denote by $\alpha_p^\epsilon$ the simple arc $\tilde{\pi}\circ\overline{\alpha}_p^\epsilon$ on $S_{0,4}$, for $\epsilon\in\{+,-\}$.
\end{construction}

\begin{proposition}\label{inner int}
Let $p, q\in\QQi$  and $\epsilon,\varphi\in\{+,-\}$ such that $\alpha_p^\epsilon$ and $\alpha_q^\varphi$ are distinct. The  intersection number  $i(\alpha_{p}^{\epsilon},\alpha_{q}^{\varphi})$  of the   simple arcs $\alpha_{p}^{\epsilon}$ and $\alpha_{q}^{\varphi}$ in the interior of $S_{0,4}$ is determined by the following formula:
\begin{itemize}
\item If  $\varphi=\epsilon$, then
$$i(\alpha_{p}^{\varphi}, \alpha_{q}^{\epsilon})= \lfloor \frac{|\Delta(p,q)|-1}{2} \rfloor;$$
\item If  $\varphi=-\epsilon$, then
$$i(\alpha_{p}^{\varphi}, \alpha_{q}^{\epsilon})=\begin{cases}
   \frac{|\Delta(p,q)| }{2}  & \text{if $\Delta(p,q)$ is even,} \\
  \frac{|\Delta(p,q)|-1 }{2} &\text{if $\Delta(p,q)$ is odd,}
\end{cases} $$
\end{itemize}
where $\lfloor u \rfloor$ denotes the integer part of $u$ and $|u|$ denotes the absolute  value of $u$.
\end{proposition}
\begin{proof}
Let
$$\P=\{p_1=\tilde{\pi}(0,0),p_2=\tilde{\pi}(\frac{1}{2},0),p_3=\tilde{\pi}(\frac{1}{2},\frac{1}{2}),p_4=\tilde{\pi}(0,\frac{1}{2})\}$$
  denote the set of punctures.
From Construction~\ref{arcs pair}, the endpoints of the arcs are
$$\begin{array}{c|cc}
\typ{p} & \alpha_p^- & \alpha_p^+\\ \hline
0 & \{p_1,p_2\} & \{p_4,p_3\}\\
1 & \{p_1,p_3\} & \{p_4,p_2\}\\
\infty & \{p_2,p_3\} & \{p_1,p_4\}.
\end{array}$$
Let $r$ be the number of common endpoints of $\alpha_p^\epsilon$ and $\alpha_q^\varphi$. The inverse image in $T^2$ of $\alpha_p^\epsilon$ is one half of a closed curve of primitive class $(b(p),a(p))$, and similarly for $q$. By the usual determinant formula on $T^2$, these two primitive closed curves can be chosen in minimal position and meet in
$$\left|\det\begin{pmatrix} b(p)&b(q)\\ a(p)&a(q)\end{pmatrix}\right|=|\Delta(p,q)|$$
points on $T^2$. Among them, exactly $r$ are fixed points of the involution, namely the common endpoints above. The other $|\Delta(p,q)|-r$ points occur in pairs under the involution and give the interior intersections on $S_{0,4}$. Hence
$$i(\alpha_p^\epsilon,\alpha_q^\varphi)=\frac{|\Delta(p,q)|-r}{2}.$$
If $\typ{p}=\typ{q}$, then $\Delta(p,q)$ is even. Here $\typ{p}$ means the type of $p$. The table gives $r=2$ when $\varphi=\epsilon$, and $r=0$ when $\varphi=-\epsilon$. Therefore
$$i(\alpha_p^\epsilon,\alpha_q^\epsilon)=\frac{|\Delta(p,q)|-2}{2}= \lfloor\frac{|\Delta(p,q)|-1}{2}\rfloor,$$
and
$$i(\alpha_p^\epsilon,\alpha_q^{-\epsilon})=\frac{|\Delta(p,q)|}{2}.$$
If $\typ{p}\ne\typ{q}$, then checking the three pairs of types $(0,1)$, $(0,\infty)$ and $(1,\infty)$ shows that $\Delta(p,q)$ is odd. The table gives $r=1$ for all choices of signs. Thus
$$i(\alpha_p^\epsilon,\alpha_q^\varphi)=\frac{|\Delta(p,q)|-1}{2}$$
in both the same-sign and opposite-sign cases. Combining the two cases gives the stated formulas.
\end{proof}

We now consider the mapping class group of $S_{0,4}$. We first recall the group
\begin{equation*}
\PSL(2,\ZZ)=\left\{\begin{pmatrix}a&b\\c&d\end{pmatrix}\bigg| a,b,c,d\in\ZZ,ad-bc=1\right\},
\end{equation*}
which is generated by
\[
t_1=\begin{pmatrix}
1 & 1 \\
0 & 1
\end{pmatrix}\quad\text{and}\quad
t_2=\begin{pmatrix}
1 & 0 \\
-1 & 1
\end{pmatrix}.
\]
The group $\PSL(2,\ZZ)$ acts on $\QQi$ by linear fractional transformations:
\begin{equation}\label{eq:fl}
\begin{pmatrix}a&b\\c&d\end{pmatrix}\cdot\bigg(\cfrac{r}{s}\bigg)=\frac{ar+bs}{cr+ds},
\end{equation}
where $p=\frac{r}{s}\in\QQi$. In particular,
\[
t_1(p)=p+1,\qquad t_2(p)=\frac{p}{1-p}.
\]
By \Cref{lem:even-cf}, every $p$ has an even continued fraction expansion $p=[a_1,\ldots,a_{2m}]$, equivalently
\begin{equation}\label{eq:expre}
p=t_1^{a_1}t_2^{-a_2}t_1^{a_3}t_2^{-a_4}\cdots t_1^{a_{2m-1}}t_2^{-a_{2m}}\left(\frac{1}{0}\right).
\end{equation}

\begin{proposition}[{\cite[Proposition~2.7]{FM}}]\label{prop:MCGofS42}
The mapping class group  $\MCG(S_{0,4})$ of the marked surface $S_{0,4}$ is isomorphic to the semidirect product 
\begin{equation}\label{mcg of S_2^4}
\PSL(2,\ZZ)\ltimes_{\rm conj} \<\iota_1,\iota_2\>,  
\end{equation}
where the group $ \<\iota_1,\iota_2\>$ is generated by the  hyperelliptic involutions $\iota_1$ and $\iota_2$  on
$S_{0,4}$, 
with lifts $(a,b)\mapsto (a,b+\frac{1}{2})$ and $(a,b)\mapsto (a+\frac{1}{2},b)$ under $\tilde{\pi}\colon\mathbb{R}^2\to S_{0,4}$, respectively.
\end{proposition}
 The semidirect product  $\ltimes_{\rm conj}$ in \eqref{mcg of S_2^4} indicates that   $\PSL(2,\ZZ)$  acts on the normal subgroup $\<\iota_1,\iota_2\>$ by conjugation within the
larger group $\MCG(S_{0,4})$.
Consequently, every element of $\MCG(S_{0,4})$   can be written as
$h \jmath$, where  $h \in \PSL(2,\ZZ)$  and  $ \jmath \in \<\iota_1,\iota_2\>$.
This  means that  the action of $h\jmath$
   on   $S_{0,4}$ is obtained by first applying the   homeomorphism $h $  and then applying the hyperelliptic involution $\jmath$.  
\subsection{Graded simple arcs and braid twists}
We first fix the grading and the initial graded arcs used in this subsection. On $S_{0,4}$ we use the grading locally represented by the horizontal line field. We denote the resulting graded surface by $\widetilde{S}_{0,4}$. For $\epsilon\in\{+,-\}$, let $\widetilde{\alpha}_\infty^\epsilon$ be the zero lift of $\alpha_\infty^\epsilon$ with respect to this grading. We choose $\widetilde{\alpha}_0^\epsilon$ to be the unique shift of the zero lift of $\alpha_0^\epsilon$ such that
\[
\gind_{p_2}(\widetilde{\alpha}_0^{-},\widetilde{\alpha}_\infty^{-})=
\gind_{p_4}(\widetilde{\alpha}_0^{+},\widetilde{\alpha}_\infty^{+})=0.
\]
We write
\[
B_\infty \coloneqq B_{\widetilde{\alpha}_{\infty}^{-}}
\quad\text{and}\quad
B_0 \coloneqq B_{\widetilde{\alpha}_{0}^{-}}
\]
for the braid twists associated with the graded arcs of slopes $\infty$ and $0$, respectively. Under $F_*$ in \eqref{eq:ses}, one has $F_*(B_\infty)=t_1$ and $F_*(B_0)=t_2$.

We call the following lifts the braid-normalized lifts.
For $p=[a_1,\ldots,a_{2m}]\in\QQ$, set
\[
W_p\coloneqq B_\infty^{a_1}B_0^{-a_2}\cdots B_\infty^{a_{2m-1}}B_0^{-a_{2m}},
\]
and set $W_\infty=1$. With the conventions $0=[-1,1]$ and $\infty=[]$, the graded arcs
\begin{equation}\label{eq:braid-normalized-lifts}
\widetilde{\alpha}_p^\epsilon[n]\coloneqq
\begin{cases}
W_p(\widetilde{\alpha}_\infty^\epsilon[n]),& p\ge 0\text{ or }p=\infty,\\[4pt]
W_p(\widetilde{\alpha}_\infty^\epsilon[n-1]),& p<0
\end{cases}
\end{equation}
are the lifts of $\alpha_p^\epsilon$ in \Cref{arcs pair}.

\begin{proposition}\label{prop:braid-twist-slope-action}
For every $p\in\QQi$, $\epsilon\in\{+,-\}$ and $n\in\ZZ$, the braid twists satisfy
\[
 B_\infty(\widetilde{\alpha}_p^\epsilon[n])
 =\widetilde{\alpha}_{p+1}^\epsilon[n],
\]
and
\[
 B_0(\widetilde{\alpha}_p^\epsilon[n])
 =
\begin{cases}
\widetilde{\alpha}_{\frac{p}{1-p}}^\epsilon[n],& p\le 1,\\[4pt]
\widetilde{\alpha}_{\frac{p}{1-p}}^\epsilon[n+1],& 1<p\le\infty.
\end{cases}
\]
\end{proposition}
\begin{proof}
  We use the graded braid-word identities in \Cref{lem:bt-normal-forms}, which keep track of the shifts in the grading cover. In particular, the extra shift in the second formula comes from the normal form of $B_0W_p$.
With the notation $\nu(p,n)$ used in \Cref{lem:bt-normal-forms}, the definition reads
\[
\widetilde{\alpha}_p^\epsilon[n]=W_p(\widetilde{\alpha}_\infty^\epsilon[\nu(p,n)]).
\]
The identities in \Cref{lem:bt-normal-forms} hold for both choices of
$\epsilon$. Hence the first one gives
\[
B_\infty(\widetilde{\alpha}_p^\epsilon[n])
=B_\infty W_p(\widetilde{\alpha}_\infty^\epsilon[\nu(p,n)])
=W_{p+1}(\widetilde{\alpha}_\infty^\epsilon[\nu(p+1,n)])
=\widetilde{\alpha}_{p+1}^\epsilon[n].
\]
For the second formula, put $q=\frac{p}{1-p}$, with $q=-1$ when $p=\infty$.
The second identity in \Cref{lem:bt-normal-forms} gives
\[
B_0(\widetilde{\alpha}_p^\epsilon[n])
=
\begin{cases}
W_q(\widetilde{\alpha}_\infty^\epsilon[\nu(q,n)]),& p\le 1,\\[4pt]
W_q(\widetilde{\alpha}_\infty^\epsilon[\nu(q,n+1)]),& 1<p\le\infty.
\end{cases}
\]
By the definition of $\widetilde{\alpha}_q^\epsilon[\cdot]$, these two terms are
$\widetilde{\alpha}_q^\epsilon[n]$ and $\widetilde{\alpha}_q^\epsilon[n+1]$,
respectively. Replacing $q$ by $\frac{p}{1-p}$ gives the required formula.
\end{proof}

\begin{proposition}\label{prop:mcgS42}
Assume that $\iota_1$ and $\iota_2$ admit lifts
$\tilde\iota_1,\tilde\iota_2\in \MCG(\widetilde{S}_{0,4})$ preserving the zero section of the grading cover. Then there is a split short exact sequence
\begin{equation}\label{eq:ses-mcgS42}
1\longrightarrow\langle\tilde\iota_1,\tilde\iota_2\rangle
\longrightarrow \MCG(\widetilde{S}_{0,4})
\xrightarrow{\overline{F_*}} \Br_3\longrightarrow 1,
\end{equation}
where the group  $\langle\tilde\iota_1,\tilde\iota_2\rangle$ is generated by the hyperelliptic involutions $\tilde\iota_1$  and $\tilde\iota_2$ on $\widetilde{S}_{0,4}$.
More precisely, the splitting is given by the subgroup
$\langle B_0,B_\infty\rangle\cong\Br_3$. Hence,
\[
\MCG(\widetilde{S}_{0,4})\cong \Br_3\ltimes_{\operatorname{conj}}(\ZZ_2\times \ZZ_2).
\]
\end{proposition}

\begin{proof}
By \Cref{eq:braid-relation-B0-Binfty,eq:central-shift}, we have
\[
\langle B_0,B_\infty\rangle\cong \Br_3.
\]
Under this identification, the projection $\pi:\Br_3\to\PSL(2,\ZZ)$ sends
$B_0$ to $t_2$ and $B_\infty$ to $t_1$.

 We now construct $\overline{F_*}$ in \eqref{eq:ses-mcgS42} using this copy of $\Br_3$.   First, the subgroup
$\langle\tilde\iota_1,\tilde\iota_2\rangle$ is normal in
$\MCG(\widetilde{S}_{0,4})$. Indeed, for
$g\in\MCG(\widetilde{S}_{0,4})$ and
$\widetilde\jmath\in\langle\tilde\iota_1,\tilde\iota_2\rangle$, the
image $F_*(g\widetilde\jmath g^{-1})$ lies in
$\<\iota_1,\iota_2\>$. Choose
$\widetilde\jmath'\in\langle\tilde\iota_1,\tilde\iota_2\rangle$ with
$F_*(\widetilde\jmath')=F_*(g\widetilde\jmath g^{-1})$. Then
$g\widetilde\jmath g^{-1}(\widetilde\jmath')^{-1}\in\langle[1]\rangle$. Since
this element is a central grading shift and both
$g\widetilde\jmath g^{-1}$ and $\widetilde\jmath'$ have finite order, the shift
must be trivial. Hence
$g\widetilde\jmath g^{-1}\in\langle\tilde\iota_1,\tilde\iota_2\rangle$.

Next, every $g\in\MCG(\widetilde{S}_{0,4})$ can be written uniquely as
$g=h\widetilde\jmath$, where
$h\in\langle B_0,B_\infty\rangle$ and
$\widetilde\jmath\in\langle\tilde\iota_1,\tilde\iota_2\rangle$. For
existence, write $F_*(g)=u\jmath$ with $u\in\PSL(2,\ZZ)$ and
$\jmath\in\<\iota_1,\iota_2\>$. Lift $u$ to some
$h\in\langle B_0,B_\infty\rangle$ and lift $\jmath$ to
$\widetilde\jmath\in\langle\tilde\iota_1,\tilde\iota_2\rangle$. Then
$(h\widetilde\jmath)^{-1}g\in\ker F_*=\langle[1]\rangle$, and this shift lies
in $\langle B_0,B_\infty\rangle$ because $[1]=(B_0B_\infty)^3$. Absorbing it
into $h$ gives the required expression. For uniqueness, if
$h\widetilde\jmath=h'\widetilde\jmath'$, then
${h'}^{-1}h\in\langle B_0,B_\infty\rangle\cap
\langle\tilde\iota_1,\tilde\iota_2\rangle$. Its image under $F_*$ lies
in both the $\PSL(2,\ZZ)$ factor and $\<\iota_1,\iota_2\>$, hence is
trivial. Since
$F_*:\langle\tilde\iota_1,\tilde\iota_2\rangle\to
\<\iota_1,\iota_2\>$ is an isomorphism, we get $h=h'$ and
$\widetilde\jmath=\widetilde\jmath'$.

Define $\overline{F_*}(g)$ to be the braid corresponding to the first factor
$h$ in this decomposition. The normality proved above shows that
$\overline{F_*}$ is a homomorphism. We therefore have the commutative diagram
\[
\begin{tikzcd}
            &                                & 1 \arrow[d]                                         & 1 \arrow[d]                               &   \\
            & 1 \arrow[r] \arrow[d]          & {\langle\tilde\iota_1,\tilde\iota_2\rangle} \arrow[r,"\sim"] \arrow[d]           & {\<\iota_1,\iota_2\>} \arrow[r] \arrow[d] & 1 \\
1 \arrow[r] & \mathbb{Z} \arrow[r] \arrow[d,equal] & \MCG(\widetilde{S}_{0,4}) \arrow[r,"F_*"] \arrow[d,"\overline{F_*}", dashed] & \MCG(S_{0,4}) \arrow[r] \arrow[d,"p"]           & 1 \\
1 \arrow[r] & \mathbb{Z} \arrow[r] \arrow[d] & \Br_3 \arrow[r,"\pi"] \arrow[d]                            & {\PSL(2,\ZZ)} \arrow[r] \arrow[d]  & 1 \\
            & 1                              & 1                                                   & 1                                         &  
\end{tikzcd}
\]
where the map $\mathbb Z\to\Br_3$ sends $1$ to
$(B_0B_\infty)^3$. Applying the Snake Lemma to the two middle rows,
we obtain that $F_*:\ker(\overline{F_*})\to\ker(p)$ is an isomorphism and that
$\overline{F_*}$ is surjective. Since
$\ker(p)=\<\iota_1,\iota_2\>$, it follows that
$\ker(\overline{F_*})=\langle\tilde\iota_1,\tilde\iota_2\rangle$:
indeed, $\overline{F_*}$ records the braid factor in the decomposition
$g=h\widetilde\jmath$. Hence
\[
1\to \langle\tilde\iota_1,\tilde\iota_2\rangle
\to \MCG(\widetilde{S}_{0,4})\xrightarrow{\overline{F_*}}\Br_3\to1
\]
is exact. Finally, the inclusion
$\langle B_0,B_\infty\rangle\cong\Br_3\hookrightarrow\MCG(\widetilde{S}_{0,4})$
is a section, so the sequence splits.
\end{proof}

   Note that for all  $p\in\QQi$ and  $\epsilon \in \{+, -\}$, the simple arc $\alpha_p^\epsilon$   connects the same pair of punctures as  $\alpha_{\typ{p}}^\epsilon$.
We now define a (graded) tagged arc for each $p\in\QQi$ and $x \in\bfH$.

\begin{definition} \label{def:tarc}
For $p\in\QQi$ and $x\in\bfH$, we first define a tagged arc $\alpha_{p,x}=(\alpha_{p,x}^\circ,\kappa)$, whose underlying arc is
\begin{equation*}
\alpha_{p,x}^\circ=
\begin{cases}
\alpha_p^{+},&\text{if }(\typ{p},x)\in\{(0,\pm i),(0,\pm j),(1,\pm 1), (1,\pm j),(\infty,\pm 1),(\infty,\pm i)\},\\
\alpha_p^{-},&\text{otherwise.} 
\end{cases}
\end{equation*}
The tagging function $\kappa$ is determined by Table~\ref{kappa} for $p\in\{0,1,\infty\}$ and $x\in\{1,i,j,k\}$; for other values of $p$, we use the same tagging as for $\alpha_{\typ{p},x}$. For negative $x$, it is fixed by the condition $\alpha_{p,-x}=(\alpha_{p,x}^\circ,-\kappa)$.

\begin{table}[htbp]
\centering
\caption{The value of $\kappa$}
\label{kappa}
\renewcommand{\arraystretch}{1.3}
\setlength{\tabcolsep}{4.5pt}

\begin{tabular}{l *{12}{c}}
\toprule
& $\alpha_{0,1}$ & $\alpha_{0,i}$ & $\alpha_{0,j}$ & $\alpha_{0,k}$
& $\alpha_{1,1}$ & $\alpha_{1,i}$ & $\alpha_{1,j}$ & $\alpha_{1,k}$
& $\alpha_{\infty,1}$ & $\alpha_{\infty,i}$ & $\alpha_{\infty,j}$ & $\alpha_{\infty,k}$ \\
\midrule
$p_1$ & $1$ &   &   & $1$ &   & $1$ &   & $1$ & $-1$ & $1$ &   &   \\
$p_2$ & $-1$ &   &   & $1$ & $-1$ &   & $-1$ &   &   &   & $-1$ & $-1$ \\
$p_3$ &   & $-1$ & $1$ &   &   & $-1$ &   & $1$ &   &   & $-1$ & $1$ \\
$p_4$ &   & $-1$ & $-1$ &   & $1$ &   & $-1$ &   & $1$ & $1$ &   &   \\
\bottomrule
\end{tabular}
\end{table}
Adding the tagging function $\kappa$ to the corresponding graded arc
$\widetilde{\alpha}_p^\pm$, we obtain the graded tagged arc
$\widetilde{\alpha}_{p,x}$.
 \end{definition}
\begin{remark}
By the preceding constructions, we have
   $$\A(S_{0,4})=\{\alpha_p^\epsilon| \epsilon \in \{+, -\}, p \in \QQi \} \text{ and }\TA^\times(S_{0,4})=\{\alpha_{p,x}|p\in\QQi,x \in\bfH\}.$$
 Moreover, two tagged arcs $\alpha_{p,x}$ and $\alpha_{q,y}$ are
isotopic if and only if $(p,x)=(q,y)$. In fact, this definition of  $\alpha_{p,x}$
      is essentially compatible with the construction given by Barot and Geiss in  \cite[Section~5]{BG}. 
 \end{remark}
\begin{corollary}\label{cor:graded-tag-action}
For $p\in\QQi$, $x\in\bfH$ and $n\in\ZZ$, write
\[
B_\infty(\widetilde{\alpha}_{p,x}[n])
=\widetilde{\alpha}_{p+1,y}[n],
\qquad
B_0(\widetilde{\alpha}_{p,x}[n])
=\widetilde{\alpha}_{\frac{p}{1-p},z}
\bigl[n+\mathbf{1}_{\{1<p\le\infty\}}\bigr],
\]
where the labels $y$ and $z$ are given in Table~\ref{tab:bt-tag-action}; for $-x$ they are replaced by $-y$ and $-z$, respectively.
\begin{table}[htbp]
\centering
\caption{The values of $y$ and $z$ in \Cref{cor:graded-tag-action}}
\label{tab:bt-tag-action}
\renewcommand{\arraystretch}{1.2}
\setlength{\tabcolsep}{6pt}

\begin{tabular}{ccccc ccccc ccccc}
\toprule
& \multicolumn{4}{c}{$\typ{p}=0$}
& \multicolumn{4}{c}{$\typ{p}=1$}
& \multicolumn{4}{c}{$\typ{p}=\infty$} \\
\cmidrule(lr){2-5} \cmidrule(lr){6-9} \cmidrule(lr){10-13}
$x$ & $1$ & $i$ & $j$ & $k$
    & $1$ & $i$ & $j$ & $k$
    & $1$ & $i$ & $j$ & $k$ \\
\midrule
$y$ & $i$ & $j$ & $-1$ & $k$
    & $-j$ & $1$ & $i$ & $k$
    & $1$ & $i$ & $j$ & $-k$ \\
$z$ & $-1$ & $i$ & $j$ & $k$
    & $1$ & $-k$ & $-i$ & $-j$
    & $1$ & $-j$ & $-k$ & $-i$ \\
\bottomrule
\end{tabular}
\end{table}
\end{corollary}
\begin{proof}
The slope and the grading shift are given by Proposition~\ref{prop:braid-twist-slope-action}. It remains only to read the new quaternion label. Applying $B_\infty$ or $B_0$ rotates the endpoints of the corresponding graded tagged arc, and the resulting tags are then read from Table~\ref{kappa}. This gives the entries for $y$ and $z$ in Table~\ref{tab:bt-tag-action}. The entries for negative labels follow from the convention $\alpha_{p,-x}=(\alpha_{p,x}^\circ,-\kappa)$.
\end{proof}
 

\subsection{Rigid indecomposable objects via graded simple arcs}\label{sec:ind=ca}
 From now on, we work with  the marked surface with binaries obtained from  $S_{0,4}$ via Construction~\ref{cons:binary},  which we denote   by $S\x^2$,     and  
$\widetilde{S\x^2} $ its graded version.  We will establish a correspondence between indecomposable  rigid  objects in $\DC{\CPone}$ and graded simple arcs on $\widetilde{S\x^2} $. Moreover, we   give the notion of oriented intersection number of graded simple arcs on $\widetilde{S\x^2} $, and show it is compatible with the dimensions of $\Hom$-spaces.
 
Define the quotient set
 \[
\wA^\circ (\widetilde{S\x^2} )\coloneqq\{\widetilde{\sigma}\in\wA (\widetilde{S\x^2} )\}/\< \Dfang\cdot\widetilde{\sigma}\>.
 \]  From now on, a graded simple arc $\widetilde{\sigma}$ on $\widetilde{S\x^2} $ means  an element in $\wA^\circ (\widetilde{S\x^2}).$
The following proposition then holds.

\begin{proposition}\label{bi:ba-ta} Let  $ \shk:S\x^2\to S_{0,4}$ be the  collapse map that
  shrinks each binary $\sx_{p_i}$ back to the corresponding puncture $p_i$ for $i = 1, 2, 3, 4$.
 There is a bijection
 \[\varkappa\colon
 \wA^\circ(\widetilde{S\x^2} )\to \TA^\times(S_{0,4})\times\ZZ, \quad \widetilde{\sigma} \longmapsto(\shk({\sigma}),\kappa,n),
 \]
 where  $n\in\ZZ$ is the unique integer such that $\widetilde{\sigma}= \widetilde{\sigma}^0[n]$ and $\kappa$ is the tagging function  
 \begin{equation}\label{winding and tag}
\kappa \colon \{\shk({\sigma})(t)|t=0,1 \} \to \{-1,1\}, \quad \shk( {\sigma})(t)\longmapsto  -(-1)^{f_{\widetilde{\sigma}}(\sigma(t))}   , 
\end{equation}
with $f_{\widetilde{\sigma}}$ the winding function of the graded simple arc $\widetilde{\sigma}$.
\end{proposition}
\begin{proof}
According to Definition~\ref{def:gcurve}, for  any graded simple arc $\widetilde{\sigma}$ on $S\x^2$, there exists a unique integer $n$ such that $\widetilde{\sigma}= \widetilde{\sigma}^0[n]$. The map defined in \eqref{winding and tag} is a  tagged function for  simple arc $\shk( {\sigma} )$    according to  Definition~\ref{tarc}. Besides, \Cref{wfunction} implies that the winding function  $f_{\widetilde{\sigma}}$ 
does not depend on    the   $\Dfang$-orbit representative.
 Therefore,  $\varkappa$ is
    well-defined. 
 
   To show that $\varkappa$ is bijective, we construct its inverse. 
Define
 \[\varsigma\colon
 \TA^\times(S_{0,4})\times\ZZ \to  \wA^\circ(\widetilde{S\x^2} ), \quad 
 (\gamma,\kappa,n) \longmapsto 
 \widetilde{\sigma}  ,
 \]
 by the requirements
\[
\shk(\sigma)=\gamma,\qquad
\widetilde{\sigma}=\widetilde{\sigma}^0[n],
\]
and by prescribing the winding function of $\widetilde{\sigma}$ as
\[
f_{\widetilde{\sigma}}(\sigma(t))\equiv   \frac{1+\kappa(\shk({\sigma})(t))}{2}\pmod{2} ,\qquad t=0,1.
\]
This construction yields a unique   graded simple arc $\widetilde{\sigma}$, and one readily verifies that $\varkappa(\widetilde{\sigma})=(\gamma,\kappa,n) $. Hence, $\varsigma$ is indeed the inverse of  $\varkappa$
\end{proof}
By \Cref{prop:farey-braid-recursion} and \Cref{rem:farey-smoothing-index}, the braid-normalized lift $\widetilde{\alpha}_{p,x}[n]$ agrees with the standard lift $\varsigma(\alpha_{p,x},n)$.
\begin{theorem}\label{thm:X} There is a bijection
$$\wX\colon \wA^\circ(\widetilde{S\x^2} ) \to \Ind^\circ\DC{\CPone}$$
satisfying
 \[
\widetilde{X}\bigl(\varsigma(\alpha_{p,x},n)\bigr)=
\begin{cases}
E_p^x[n],   & n\ \text{even},\\[4pt]
E_p^{-x}[n], & n\ \text{odd},
\end{cases}
\qquad n\in\mathbb{Z},
\]
where $\alpha_{p,x} \in  \TA^\times(S_{0,4})$ is a tagged arc as defined in Definition~\ref{def:tarc}. 
\end{theorem}
 \begin{proof}
Every indecomposable rigid object of $\DC{\CPone}$ can be written as $E[n]$ with $E\in\Ind^\circ\coh(\CPone)$ and $n\in\mathbb{Z}$. 
The assertion now follows by combining the correspondence between rigid sheaves and tagged arcs given in Proposition~\ref{prop:rigid-to-vectors} with the bijection $\varkappa$ of Proposition~\ref{bi:ba-ta}.
 \end{proof}

\begin{remark} 
Let $\widetilde{\sigma}$ be a graded simple arc on $\widetilde{S\x^2} $ with 
endpoints  $m_{p_i},m_{p_j}\in \M_\P.$
Since $\tau {E}_p^x={E}_p^{-x}$ holds for any   $p\in\QQi$ and $x\in\bfH$, we have $$\wX(\mathrm{D}_{ {p_i}}\mathrm{D}_{ {p_j}}(\widetilde{\sigma}))=\tau\wX(\widetilde{\sigma}),$$
where $\mathrm{D}_{p_i}$ is  the Dehn twist along binary    $\sx_{p_i}$, $p_i\in\P$.
\end{remark}

Before defining oriented intersections, we record the local index of the graded arcs in \eqref{eq:braid-normalized-lifts}. Let $a$ be an interior intersection of $\widetilde{\alpha}_p^\epsilon$ and $\widetilde{\alpha}_q^\varphi$. By \Cref{prop:farey-braid-recursion} and \Cref{rem:farey-smoothing-index}, the Farey recursion preserves the grading fixed by the initial lifts. Hence, the defining counterclockwise rotation has index $0$ when $p<q$; the case $p>q$ follows from \eqref{eq:gind-sym}. Thus
\[
\gind_a(\widetilde{\alpha}_p^\epsilon,\widetilde{\alpha}_q^\varphi)=
\begin{cases}
0,&p<q,\\
1,&p>q.
\end{cases}
\]

\begin{definition} \label{def:boint}
Let $\widetilde{\sigma},\widetilde{\gamma}$ be two graded simple arcs on  $\widetilde{S\x^2} $ in   minimal position  with $\shk(\sigma)=\alpha_{p}^{\epsilon}$ and  $\shk(\gamma)=\alpha_{q}^{\varphi}$.
The \emph{binary oriented intersection} $\overrightarrow{\cap}^0(\widetilde{\sigma},\widetilde{\gamma})$   of index $0$  from $\widetilde{\sigma} $ to $\widetilde{\gamma}$  consists of the following three types of  intersections $a$ between $ \widetilde{\sigma}$ and $ \widetilde{\gamma}$: 
\begin{enumerate}
	\item[(I)] $a\notin \M_\P$,   as shown  in \Cref{fig:oriented int}(I) and $p<q$.
	\item[(II)] $a\in \M_\P$, $ f_{\widetilde{\sigma}}(a)=f_{\widetilde{\gamma} }(a)$,  
    and $p<q$.
	\item[(III)]  $a\in \M_\P$, $\widetilde{\gamma}=\widetilde{\sigma}$, and  there is a  counterclockwise angle at $a$   from $\widetilde{\sigma}$ to $\widetilde{\gamma}$,  as shown in   \Cref{fig:oriented int}(II).
	\end{enumerate}

    Similarly,  the \emph{binary oriented intersection} $\overrightarrow{\cap}^1(\widetilde{\sigma},\widetilde{\gamma})$   of index $1$  from $\widetilde{\sigma}$ to $\widetilde{\gamma}$  consists of the following three types of   intersections $a $: 
\begin{enumerate}
	\item[(I)] $a\notin \M_\P$,  as shown in  \Cref{fig:oriented int}(I), and $p>q$.
	\item[(II)] $a\in \M_\P$, $ f_{\widetilde{\sigma}}(a)\ne f_{\widetilde{\gamma} }(a)$, 
    and $p>q$.
	\item[(III)]  $a\in \M_\P$, $\widetilde{\gamma}\in \mathrm{D}_{\vot}\cdot \widetilde{\sigma}\setminus\Dfang \cdot\widetilde{\sigma}$, and  there is a  counterclockwise angle at $a$   from $\widetilde{\sigma}$ to $\widetilde{\gamma}$, as shown in \Cref{fig:oriented int}(II).
	\end{enumerate} 
\end{definition}
\begin{figure}[h]
\begin{tikzpicture}[scale=.5]
	    \begin{scope}[shift={(0,0)}]
		\draw[blue,very thick](-2,-2)to(2,2)node[above]{$\widetilde{\gamma }$}(2,-2)to(-2,2)node[above]{$\widetilde{\sigma}$};
		\draw (0,0.1)[above]node{$a$};
		\draw (0,-3)node{(I)};
	   \end{scope}

    \begin{scope}[shift={(7,-0.5)}, scale=.5, rotate=270]
   			\draw[ultra thick,fill=gray!10] (90:0.5) ellipse (1.5) node {$\sx_a$};
   			\draw(-5.7,3)node[thick, Green]{$\widetilde{\sigma}$}(125:3.5)node[thick, blue]{$\widetilde{\gamma} $};		
   			\draw[blue, ultra thick] (90:7)
   			.. controls +(150:5) and  +(-150:5) .. (-90:1);
            	\draw[Green, ultra thick] (90:7)
   			.. controls +(150:7.5) and  +(-150:7.5) .. (-90:1); 
            \draw[thick, red,->-=.99,>=stealth]  (-2.6,0.3)to[bend right=15] (-4.1,-0.4)  ;
            \draw[thick, red,->-=.99,>=stealth]    (-4.1,6 ) to[bend right=15](-3,5.3); 
            
             
   			\draw[very thick,blue] (-90:1)\nn ;
   			\begin{scope}[shift={(0,6)}, rotate=180]
   				\draw[ultra thick,fill=gray!10] (90:0.5) ellipse (1.5) node {$\sx_{a'}$};
   				\draw[very thick,blue] (-90:1)\nn  ; 
   			\end{scope}
            \draw (5,3)node{(III)};
   		\end{scope}

\begin{scope}[shift={(15,0.5)}, scale=.5, rotate=270]
   			\draw[ultra thick,fill=gray!10] (90:0.5) ellipse (1.5) node {$\sx_a$};
   			\draw(40:5.5)node[thick, Green]{$\widetilde{\sigma}$}(142:5.6)node[thick, blue]{$\widetilde{\gamma}$};		
   			\draw[blue, ultra thick] (90:7)
   			.. controls +(150:5) and  +(-150:5) .. (-90:1);
 \draw[thick, red,->-=.99,>=stealth]  (-1,-1.5)to[bend right=15](1,-1.5);

   \draw[thick,red,->-=.99,>=stealth]  (1, 7.5)to[bend right=15](-1,7.5);
\draw[Green, ultra thick] (90:7)
   			.. controls +(30:5) and  +(-30:5) .. (-90:1);
   			\draw[very thick,blue] (-90:1)\nn ;
   			\begin{scope}[shift={(0,6)}, rotate=180]
   				\draw[ultra thick,fill=gray!10] (90:0.5) ellipse (1.5) node {$\sx_{a'}$};
   				\draw[very thick,blue] (-90:1)\nn  ; 
   			\end{scope}
            \draw  (6.8,3.2)node{(II)};
   		\end{scope}

\end{tikzpicture}
    \caption{Binary oriented intersections from $\widetilde{\sigma}$ to $\widetilde{\gamma}$}
	\label{fig:oriented int}
\end{figure}
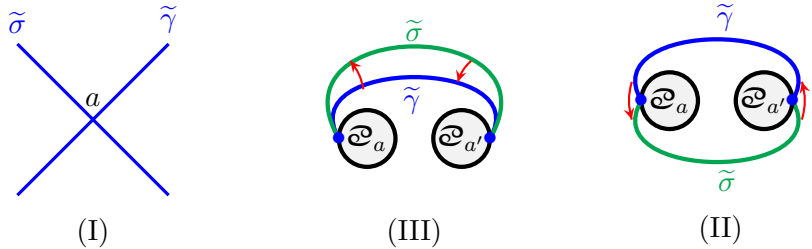

Note that for $d=0,1$, \begin{equation*} \overrightarrow{\cap}^d(\widetilde{\sigma},\widetilde{\gamma})=\overrightarrow{\cap}_{\surfi\x}^d(\widetilde{\sigma},\widetilde{\gamma})\cup \overrightarrow{\cap}_{\partial\surf\x}^d(\widetilde{\sigma},\widetilde{\gamma}), 
\end{equation*}
where  the first term on the right-hand side consists of interior intersections (i.e. type (I)) and the second  term consists of intersections at the endpoints  (all other types).

\begin{definition}\label{def:Int}
    For  any  graded simple  arcs 
$\widetilde{\sigma}=\widetilde{\sigma}^0[n],\widetilde{\gamma}=\widetilde{\gamma}^0[m]$, and   any
 $d\in\mathbb{Z}$, the \emph{oriented intersection number} of index $d$ from $\widetilde{\sigma}$ to $\widetilde{\gamma}$  is   defined by
    \[
\oInt^d(\widetilde{\sigma},\widetilde{\gamma})=\begin{cases}
         |\overrightarrow{\cap}^0(\widetilde{\sigma},\widetilde{\gamma})|, & \text{if  $d=n-m$},\\
         |\overrightarrow{\cap}^1(\widetilde{\sigma},\widetilde{\gamma})|, &\text{if $d=n-m+1$},\\
         0, &\text{ otherwise.}
     \end{cases}
    \]  
Moreover, the following relation holds for any $s,t\in \ZZ$:
  $$\oInt^{d+t-s}(\widetilde{\sigma},\widetilde{\gamma})=
  \oInt^{d}(\widetilde{\sigma}[s],\widetilde{\gamma}[t]).$$
\end{definition}
We now present the main result of this subsection. 
\begin{theorem}\label{thm:X2}
The bijection $\wX$ in Theorem~\ref{thm:X} satisfies
\begin{equation*} 
\oInt^d(\widetilde{\sigma} ,\widetilde{\gamma} )=\dim\Hom_{\cD} ( \wX(\widetilde{\sigma}),\wX(\widetilde{\gamma})[d] ), \quad \text{for all } d \in\ZZ.
\end{equation*}
\end{theorem}
 \begin{proof} 
 Assume that
 \[\widetilde{\sigma}=\widetilde{\sigma}^0[n],\ \widetilde{\gamma}=\widetilde{\gamma}^0[m],\] with
 \[[\wX(\widetilde{\sigma}^0)]  =\bv_p^x[0],\  [ \wX(\widetilde{\gamma}^0)]  =\bv_q^y[0],\] 
for some $n,m\in\ZZ$, $p, q\in\QQi$ and $x,y\in\bfH$.
By  Definition~\ref{def:Int}  and Proposition~\ref{prop:db}, it suffices to prove that  
\begin{equation}
\label{int=dimH}|\overrightarrow{\cap}^d(\widetilde{\sigma}^0 ,\widetilde{\gamma}^0 )|=\dim\Hom_{\cD} ( \wX(\widetilde{\sigma}^0),\wX(\widetilde{\gamma}^0) [d] ),  \quad \text{for } d = 0, 1.
\end{equation} 

If  $p = q$, then 
\[
\dim\Hom_{\cD} ( \wX(\widetilde{\sigma}^0),\wX(\widetilde{\gamma}^0) )=\begin{cases}
   1, &\text{if $x=y$,}\\
    0&\text{otherwise,}
\end{cases}
\]
and
\[
\dim\Hom_{\cD} ( \wX(\widetilde{\sigma}^0),\wX(\widetilde{\gamma}^0)[1]  )=\begin{cases}
    1, &\text{if $x=-y$,}\\
    0,&\text{otherwise.}
\end{cases}
\]
Thus the result for the case of $p = q$ follows directly from Definition~\ref{def:tarc}, Definition~\ref{def:boint}, and Theorem~\ref{thm:X}.

Now suppose that $p \ne q$. By the definition of the Euler form, we have
\[
\dim\Hom_{\cD} ( \wX(\widetilde{\sigma}^0),\wX(\widetilde{\gamma}^0) )=\begin{cases}
    \<\bv_p^x,\bv_q^y\>, &\text{if $p<q$,}\\
    0&\text{otherwise,}
\end{cases}
\]
and
\[
\dim\Hom_{\cD} ( \wX(\widetilde{\sigma}^0),\wX(\widetilde{\gamma}^0)[1]  )=\begin{cases}
    0, &\text{if $p<q$,}\\
    -\<\bv_p^x,\bv_q^y\>,&\text{otherwise.}
\end{cases}
\]
 We divide the proof into two cases  according to the types of  $p$  and  $q$:

First, suppose that $\typ{p}=\typ{q}$. Then  $\Delta(p,q)$ is even.
  Thus,  by Proposition~\ref{hom dim}, we have
  \[
\<\bv_p^x,\bv_q^y\>=\begin{cases}
\frac{\Delta(p,q)}{2} +1 & \text{if $y=x$},\\
  \frac{\Delta(p,q)}{2} -1 & \text{if $y=-x$},\\
  \frac{\Delta(p,q)}{2}  & \text{otherwise.}
\end{cases}
\]
We now  compute the oriented intersection numbers with the help of Proposition~\ref{inner int}.
 If $p<q$, then
\[|\overrightarrow{\cap}^0_{\surfi}(\widetilde{\sigma}^0 ,\widetilde{\gamma}^0 )|= i({\sigma}, {\gamma} )=\begin{cases}
\frac{\Delta(p,q)}{2}-1  & \text{if  $\sigma$ and $\gamma$ share the same endpoints, 
}\\
  \frac{\Delta(p,q)}{2} &   \text{otherwise,}
\end{cases}
\]
and 
\[
|\overrightarrow{\cap}^0_{\partial\surf}(\widetilde{\sigma}^0 ,\widetilde{\gamma}^0 )|=\begin{cases}
2 & \text{if $y=x$},\\
  1& \text{if  $\sigma$ and $\gamma$ share the same endpoints, $y\ne x$ and $y\ne -x$,}\\
0 & \text{otherwise.}
\end{cases}
\] 
If $p>q$, then 
\[|\overrightarrow{\cap}^1_{\surfi}(\widetilde{\sigma}^0 ,\widetilde{\gamma}^0 )|= i({\sigma},{\gamma} )=\begin{cases}
-\frac{\Delta(p,q)}{2}-1  & \text{if  $\sigma$ and $\gamma$ share the same endpoints, 
}\\
 - \frac{\Delta(p,q)}{2} &   \text{otherwise,}
\end{cases}
\]
and
\[
|\overrightarrow{\cap}^1_{\partial\surf}(\widetilde{\sigma}^0 ,\widetilde{\gamma}^0 )|=\begin{cases}
 2& \text{if $y=-x$},\\
 1& \text{if  $\sigma$ and $\gamma$ share the same endpoints, $y\ne x$ and $y\ne -x$},\\
0  & \text{otherwise.}
\end{cases}
\]
 Consequently,  it follows from Definition~\ref{def:boint}  that Equation~\eqref{int=dimH} holds.

  Second, suppose that $\typ{p}\ne\typ{q}$. Then   $\sigma$ and $\gamma$ share exactly one endpoint, and  
\[
\Delta(\typ{p} ,\typ{q})=\begin{cases}
    1 & \text{ if $\typ{p}<\typ{q}$
    ,}\\
    -1 & \text{otherwise.}
\end{cases}
\]
Thus,  by Proposition~\ref{hom dim}, we have
  \[
 \<\bv_p^x,\bv_q^y\>=\begin{cases} 
\<\bv_{\typ{p}}^x, \bv_{\typ{q}}^y\> +\frac{\Delta(p,q)-1}{2}   & \text{if  $\typ{p}<\typ{q}$},\\ 
 \<\bv_{\typ{p}}^x, \bv_{\typ{q}}^y\> + \frac{\Delta(p,q)+1}{2}   & \text{otherwise.}
\end{cases}
\]
On the other hand, also by Proposition~\ref{inner int},
for $p<q$:
\[|\overrightarrow{\cap}^0_{\surfi}(\widetilde{\sigma}^0 ,\widetilde{\gamma}^0 )|=i ({\sigma},{\gamma} )=\frac{\Delta(p,q)-1}{2},
\]
and
\[
|\overrightarrow{\cap}^0_{\partial\surf}(\widetilde{\sigma}^0 ,\widetilde{\gamma}^0 )|=\begin{cases}
    \<\bv_{\typ{p}}^x, \bv_{\typ{q}}^y\>, & \text{if  $\typ{p}<\typ{q}$,} \\
    1+ \<\bv_{\typ{p}}^x, \bv_{\typ{q}}^y\>, & \text{otherwise;}
 \end{cases}\]
For  $p>q$:
\[|\overrightarrow{\cap}^1_{\surfi}(\widetilde{\sigma}^0 ,\widetilde{\gamma}^0 )|= i({\sigma} ,{\gamma} )=\frac{-\Delta(p,q)-1}{2},
\]
and
\[
|\overrightarrow{\cap}^1_{\partial\surf}(\widetilde{\sigma}^0 ,\widetilde{\gamma}^0 )|=\begin{cases}
    1-\<\bv_{\typ{p}}^x, \bv_{\typ{q}}^y\>, & \text{if  $\typ{p}<\typ{q}$,} \\
    -\<\bv_{\typ{p}}^x, \bv_{\typ{q}}^y\>, & \text{otherwise.}
 \end{cases}\]
In both subcases, a straightforward calculation confirms that Equation \eqref{int=dimH} holds.
 \end{proof}
  
\subsection{Automorphism group via mapping class group}\label{sec:aut=mcg}
 In this subsection,  we construct a geometric interpretation for the automorphism group
  $\Aut\DC{\CPone} $ of the derived category $\DC{\CPone}$.
 
Before stating  the main theorem of this subsection, we recall some results about the  automorphism group  $ \Aut\DC{\CPone}$ from  \cite{LM}.
 Denote by $\Aut(\coh(\CPone))$ the automorphism group of $\coh(\CPone)$, and by
 $\Aut (\CPone)$ the subgroup consisting of all members of $\Aut(\coh(\CPone))$ fixing the structure sheaf $\cO$. 
\begin{proposition}[{\cite[Remark~3.3(iii)]{LM}}]\label{prop:aut-CPone}
 Let  $V$ be  the Klein four group generated by
 the  permutations
 \[
(0,1)(\infty,\lambda),\ 
(0,\infty)(1,\lambda),
\]
on the set   $\{0,1,\infty,\lambda\}$. The group $\Aut (\CPone)$ 
 depends on the  $j$-invariant  
\[
j(\lambda)=2^8\frac{(\lambda^2-\lambda+1)^3}{\lambda^2(\lambda-1)^2},
\]
and is given by
\[
\Aut(\CPone)\cong
\begin{cases}
\mathbb{A}_4,&j(\lambda)=0,\\
\mathbb{D}_4,&j(\lambda)=1728,\\
V,&j(\lambda)\ne0,1728,
\end{cases}
\]
where $\mathbb{A}_4$ is the alternating group on the set $\{0,1,\infty,\lambda\}$ and $\mathbb{D}_4$ is the dihedral group generated by the  group  $V$ together with the transposition $(0,1)$. All permutations are given in cycle notation.
\end{proposition}

\begin{remark}
    Observe that for $j(\lambda)=1728$, the group $\Aut(\CPone)$ is exactly the dihedral group $\mathbb{D}_4$, which will be our primary focus in what follows. 
The equation  $j(\lambda)=1728$ admits three distinct solutions $\lambda\in \{-1,\frac12,2\}$. As these lie in the same  Möbius orbit and therefore define isomorphic weighted projective lines,  we choose $\lambda=\frac12$ throughout this subsection.
\end{remark}
 \begin{proposition}[{\cite[Theorem~6.3]{LM}}]\label{prop:aut-cohCPone}
Let $\Pic_0(\CPone)$ be the subgroup of $\Aut(\coh(\CPone))$ consisting of all degree-preserving shifts.
Then  
\begin{equation}\label{aut of D}
     \Aut\DC{\CPone}\cong \Br_3\ltimes_{\rm conj}(\Aut(\CPone) \ltimes_2 \Pic_0\CPone), 
\end{equation}
where 
$\Aut(\CPone)$ acts on $ \Pic_0\CPone$ via
 \[
 \nu.\vec{x}= \overrightarrow{ \nu(x) }, \quad \text{for $\nu\in\Aut(\CPone)$, $\vec{x}\in  \Pic_0\CPone$},  
 \] 
and   the subgroup  
$\Br_3$ of $\Aut\DC{\CPone}$ is the 3-strand braid group generated by
the tubular left mutations $\tubL$ and $\tubR$  with 
respect to the $\tau$-orbit of the structure sheaf  $\mathcal{O}$ and   the  exceptional simple sheaf $S_{\lambda,0}$.
 Furthermore,  each element of $\Aut\DC{\CPone}$ can be  written as ${h}z$ with ${h}\in  \Br_3 $
     and $z\in \Aut(\CPone) \ltimes_2 \Pic_0\CPone$.

 \end{proposition}
\begin{proposition}\label{tubular left mutations} 
The action of  tubular left mutations $\tubR$ and $\tubL$   on indecomposable rigid objects $E_p^x[n]$ are given by
\begin{equation}\label{RL}
\tubR\bigl(E_{p}^x[n]\bigr)=E_{p+1}^y[n],\qquad
\tubL\bigl(E_{p}^x[n]\bigr)=
\begin{cases}
E_{\frac{p}{1-p}}^z[n],   & p\le 1,\\[4pt]
E_{\frac{p}{1-p}}^{-z}[n+1], & 1<p\le\infty,
\end{cases}
\end{equation}
where the labels $y$ and $z$ are given in Table~\ref{tab:bt-tag-action}. 
When the superscript $x$ is replaced by $-x$, the labels $y$ and $z$ are replaced accordingly by $-y$ and $-z$, respectively.
\end{proposition}
\begin{proof} 
Recall from \cite{M1997} that the left mutation $\tubL_A$ with respect to the $\tau$-orbit of an object $A\in\Ind^\circ\coh(\CPone)$ acts on an object $X\in\DC{\CPone}$ as the right-hand term of the distinguished triangle
\begin{equation}\label{left mutation}
\Hom^\bullet(A,X)\otimes A\oplus\Hom^\bullet(\tau A,X)\otimes\tau A\xrightarrow{\mathrm{can}_X}X\to \tubL_A(X),
\end{equation}
where $\Hom^\bullet(A,X)=\bigoplus_{j\in\mathbb{Z}}\Hom_{\mathcal{D}}(A,X[j])[-j]$ is viewed as a complex with zero differential, and the map  $\mathrm{can}_X$ is the canonical map induced by the identity endomorphisms.

By \Cref{hom dim}, we have
\begin{equation}
    \label{infty p}
    \langle \bv_{\infty}^{\pm k}, \bv_{p}^x\rangle
=\langle \bv_{\infty}^{\pm k}, \bv_{\typ{p}}^x\rangle-\frac{b(p)}{2}-\frac{1}{2}\Delta(\infty,\typ{p}),
\end{equation} 
\begin{equation}
    \label{0 p}
    \langle \bv_{0}^{\pm 1}, \bv_{p}^x\rangle
=\langle \bv_{0}^{\pm 1}, \bv_{\typ{p}}^x\rangle+\frac{a(p)}{2}-\frac{1}{2}\Delta(0,\typ{p}).
\end{equation}

\noindent\textbf{Computation of $\tubR(E_p^x[n])$.}
From \Cref{prop:db} and the triangle~\eqref{left mutation}, we obtain a short exact sequence in $\coh(\CPone)[n]$:
\[
0\to E_p^x[n]\to \tubR(E_p^x[n])\to
S_{\lambda,0}^{\oplus|\langle\bv_\infty^k,\bv_p^x\rangle|}[n]
\oplus
S_{\lambda,1}^{\oplus|\langle\bv_\infty^{-k},\bv_p^x\rangle|}[n]\to0,
\]
where $|a|$ denotes the absolute value of $a$.  Hence $\tubR(E_p^x[n])=E_q^y[n]$ for some $q\in\mathbb{Q}_{\infty}$. 
Taking the class in the Grothendieck group and applying~\eqref{infty p} yields
\[
[E_{q}^y]= \bv_{\typ{p}}^x
-\langle\bv_\infty^k,\bv_{\typ{p}}^x\rangle\bv_\infty^k
-\langle\bv_\infty^{-k},\bv_{\typ{p}}^x\rangle\bv_\infty^{-k} +  \lfloor\frac{b(p)}{2} \rfloor\mathbf{h}_0  
+  \lfloor\frac{a(p)}{2} \rfloor\mathbf{h}_{\infty}
+\frac{b(p)+\Delta(\infty,\typ{p})}{2}\mathbf{h}_{\infty}.
\]
Using the short exact sequence  in  $\coh(\CPone)$,
\[
  0\longrightarrow \mathcal{O}( -\vec{x}_4)\stackrel{x_4}{\longrightarrow} \mathcal{O}\longrightarrow  S_{\lambda,0} \longrightarrow 0,
\]
we  obtain
\[
\bv_{\typ{p}}^x
-\langle\bv_\infty^k,\bv_{\typ{p}}^x\rangle\bv_\infty^k
-\langle\bv_\infty^{-k},\bv_{\typ{p}}^x\rangle\bv_\infty^{-k}
=
\begin{cases}
 \bv_{\typ{p+1}}^y , & \typ{p}=0,\infty,\\[4pt]
 \bv_{\typ{p+1}}^y +\mathbf{h}_{\infty}, & \typ{p}=1,
\end{cases}
\] where the label  $y$  is given in Table~\ref{tab:bt-tag-action}; for $-x$ it is  replaced by $-y$.
Substituting this into the previous expression gives  
\[  
[E_{q}^y]= \bv_{\typ{p+1}}^{y} 
          + \lfloor\frac{b(p)}{2} \rfloor\mathbf{h}_0
          + \lfloor\frac{a(p)+b(p)}{2} \rfloor\mathbf{h}_{\infty},
\]
which is precisely the formula for $\tubR(E_p^x[n])$ stated in \Cref{RL}. 

\medskip\noindent
\textbf{Computation of $\tubL(E_p^x[n])$.}
We first recall from \cite[Corollary~2.6]{M1997} that
\[
\tubL(E_0^x[n])=
\begin{cases}
E_0^{-x}[n], & x\in\{-1,1\},\\[4pt]
E_0^x[n],    & \text{otherwise}.
\end{cases}
\]

For $p\neq0$, put
\[
\varepsilon_p:=\langle\bv_0^1,\bv_p^x\rangle,\qquad
\bar{\varepsilon}_p:=\langle\bv_0^{-1},\bv_p^x\rangle.
\] 
By the same corollary,  the following exact sequences hold   in $\coh(\CPone)[n]$:
\begin{align*}
0 &\to \tubL(E_p^x[n])[-1]
   \to \mathcal{O}^{\oplus\varepsilon_p}[n] \oplus \tau\mathcal{O}^{\oplus\bar{\varepsilon}_p}[n]
   \to E_p^x[n] \to 0, && p>1,\\[4pt]
0 &\to \mathcal{O}^{\oplus\varepsilon_p}[n] \oplus \tau\mathcal{O}^{\oplus\bar{\varepsilon}_p}[n]
   \to E_p^x[n] \to \tubL(E_p^x[n]) \to 0, && 0<p\le 1,\\[4pt]
0 &\to E_p^x[n] \to \tubL(E_p^x[n])
   \to \mathcal{O}^{\oplus|\varepsilon_p|}[n] \oplus \tau\mathcal{O}^{\oplus|\bar{\varepsilon}_p|}[n]
   \to 0, && p<0.
\end{align*}

A calculation analogous to the one for $\tubR$ (using~\eqref{0 p} and \eqref{es}) yields 
\[
\tubL\bigl(E_{p}^x[n]\bigr)=
\begin{cases}
X[n],   & p\le 1,\\[4pt]
X'[n+1], & 1<p\le\infty,
\end{cases}
\]
where
\[
[X]= \bv_{\typ{\frac{ {p}}{1- {p}}}}^{z}  
+ \lfloor\frac{b(p)-a(p)}{2} \rfloor\mathbf{h}_0
+ \lfloor\frac{a(p)}{2} \rfloor\mathbf{h}_{\infty}
\]
and 
\[
[X']=  \bv_{\typ{\frac{ {p}}{1- {p}}}}^{-z}
+ \lfloor\frac{a(p)-b(p)}{2} \rfloor\mathbf{h}_0
+ \lfloor\frac{-a(p)}{2} \rfloor\mathbf{h}_{\infty}.
\]
Here the label  $z$  is given in Table~\ref{tab:bt-tag-action}. Thus we have done.
\end{proof}

 Now we calculate 
  the graded mapping class group of  $\widetilde{S\x^2} $.
  For the graded marked surface with binaries $\widetilde{S\x^2} $, we denote by  
  $\mathrm{D}_{p_i}$   the Dehn twist along binary    $\sx_{p_i}$, $p_i\in\P$ and
 $\overline{\mathrm{D}}_{p_i}$   the image of  $\mathrm{D}_{p_i}$ in the quotient group $\overline{\mathrm{D}}_{\vot}=\mathrm{D}_{\vot}/ \Dfang $. It follows from Fact 3.9  in \cite{FM}  that   \[
 \mathrm{D}_{p_i}^{t}\mathrm{D}_{p_j}^{t'}=\mathrm{D}_{p_j}^{t'}\mathrm{D}_{p_i}^{t},\quad   p_i,p_j\in\P,\  t,t'\in \ZZ.
 \] 
 Thus, we obtain that
 $\mathrm{D}_{\vot}\cong \ZZ^4$ and  $\overline{\mathrm{D}}_{\vot}\cong \ZZ_2^4$.
 
 \begin{lemma}\label{2mcg} The graded   mapping class group $\MCG(\widetilde{S\x^2})$
 of $S\x^2$  satisfies
\begin{equation}
    \label{mcg2,2,2,2}
\MCG(\widetilde{S\x^2})\cong\MCG(\widetilde{S}_{0,4})\ltimes_{\rm conj}  \mathrm{D}_{\vot}\cong(\Br_3\ltimes_{\rm conj}\langle\tilde\iota_1,\tilde\iota_2\rangle)\ltimes_{\rm conj} \mathrm{D}_{\vot}.
\end{equation}
Furthermore, the binary   mapping class group $\MCG^\sx(\widetilde{S\x^2} )$ of $\widetilde{S\x^2} $  is
\begin{equation}
    \label{mcgg2222} \MCG^\sx(\widetilde{S\x^2} )\cong 
    (\Br_3\ltimes_{\rm conj}\langle\tilde\iota_1,\tilde\iota_2\rangle)\ltimes_{\rm conj}\overline{\mathrm{D}}_{\vot}\cong \Br_3\ltimes_{\rm conj}(\langle\tilde\iota_1,\tilde\iota_2\rangle \ltimes_{\rm conj}\overline{\mathrm{D}}_{\vot}).  
\end{equation}
 \end{lemma}
 \begin{proof}
 Shrinking the four binaries to punctures induces a homomorphism
\[
\psi:\MCG(\widetilde{S\x^2})\longrightarrow
\MCG(\widetilde{S}_{0,4}).
\]
Its kernel is generated by the Dehn twists along the four binaries, and hence equals $\mathrm D_{\vot}$. Besides, we can choose the lifts of
$B_0,B_\infty,\widetilde\iota_1,\widetilde\iota_2$ satisfy the same relations as  those holding in $\MCG(\widetilde S_{0,4})$ and therefore define a section $s$ of $\psi$.  Hence, we obtain the semidirect product decomposition
\[
\MCG(\widetilde{S\x^2})
\cong\MCG(\widetilde S_{0,4})\ltimes_{\rm conj}\mathrm D_{\vot},
\]
where $\MCG(\widetilde{S}_{0,4})$ acts on $\mathrm{D}_{\vot}$ by conjugation.
It follows directly from the definition of the  binary mapping class group that
\[
\MCG^\sx(\widetilde{S\x^2} )\cong 
    (\Br_3\ltimes_{\rm conj}\langle\tilde\iota_1,\tilde\iota_2\rangle)\ltimes_{\rm conj}\overline{\mathrm{D}}_{\vot},
\] where each element of 
 $\MCG^\sx(\widetilde{S\x^2})$ has the form $(g,\prod_{i=1}^4 \overline{\mathrm{D}}_{p_i}^{t_i})$ with $g \in \MCG (\widetilde{S\x^2})$ and $t_i\in \ZZ_2$. Moreover,  the multiplication in $\MCG^\sx(\widetilde{S\x^2})$ is given by  
 \[
(g, \prod_{i=1}^4 \overline{\mathrm{D}}_{p_i}^{t_i}) * (g', \prod_{i=1}^4 \overline{\mathrm{D}}_{p_i}^{t_i'}) := (g g', \prod_{i=1}^4 \overline{\mathrm{D}}_{g'(p_i)}^{t_i}   \overline{\mathrm{D}}_{ p_i }^{t_i'} )  , \quad (g, \prod_{i=1}^4 \overline{\mathrm{D}}_{p_i}^{t_i}),(g', \prod_{i=1}^4 \overline{\mathrm{D}}_{p_i}^{t_i'})\in\MCG^\sx(\widetilde{S\x^2}). 
\]
The subgroup
$\langle\tilde\iota_1,\tilde\iota_2\rangle\ltimes_{\rm conj}\overline{\mathrm{D}}_{\vot}$
is normal in $\MCG^\sx(\widetilde{S\x^2})$.
Indeed, for any $h\in \Br_3$, $\mathbf{i}\in\langle\tilde\iota_1,\tilde\iota_2\rangle$ and $ t_i  \in \ZZ_2$,    we have 
\[
(h^{-1},  {\rm id}) * (\mathbf{i},\prod_{i=1}^4 \overline{\mathrm{D}}_{p_i}^{t_i})*(h ,  {\rm id}) =(h^{-1}\mathbf{i}h,\prod_{i=1}^4 \overline{\mathrm{D}}_{h(p_i)}^{t_i}).
\]
Finally, we define the group $ \Br_3\ltimes_{\rm conj}(\langle\tilde\iota_1,\tilde\iota_2\rangle \ltimes_{\rm conj}\overline{\mathrm{D}}_{\vot})$ with multiplication
\[
h(\mathbf{i},\prod_{i=1}^4 \overline{\mathrm{D}}_{p_i}^{t_i})\cdot h'(\mathbf{i}',\prod_{i=1}^4 \overline{\mathrm{D}}_{p_i}^{t_i'})=hh'[({h'}^{-1}\mathbf{i}{h'}, \prod_{i=1}^4 \overline{\mathrm{D}}_{h'(p_i)}^{t_i} )*(\mathbf{i}',\prod_{i=1}^4 \overline{\mathrm{D}}_{p_i}^{t_i'})].
\]
 Then    the map
\[  f:  \Br_3\ltimes_{\rm conj}(\langle\tilde\iota_1,\tilde\iota_2\rangle \ltimes_{\rm conj}\overline{\mathrm{D}}_{\vot})\to\MCG^\sx(\widetilde{S\x^2}) ,\ h(\mathbf{i}, \prod_{i=1}^4 \overline{\mathrm{D}}_{ p_i}^{t_i})  \mapsto
(h\mathbf{i}, \prod_{i=1}^4 \overline{\mathrm{D}}_{ p_i}^{t_i})
\] is a group isomorphism. Thus we have done.
 \end{proof}

  \begin{proposition}\label{keep slope}
There  exists a group isomorphism
    \[
\Aut(\CPone) \ltimes_2 \Pic_0\CPone \cong \langle\tilde\iota_1,\tilde\iota_2\rangle \ltimes_{\rm conj}\overline{\mathrm{D}}_{\vot},
\] where  the multiplication in $\langle\tilde\iota_1,\tilde\iota_2\rangle\ltimes_{\rm conj} \overline{\mathrm{D}}_{\vot}$ is given by  
 \[
(\mathbf{i}, \prod_{i=1}^4 \overline{\mathrm{D}}_{p_i}^{t_i}) * (\mathbf{i}',\prod_{i=1}^4\overline{\mathrm{D}}_{p_i}^{t_i'}) := ( \mathbf{i}\mathbf{i}'  ,\prod_{i=1}^4 \overline{\mathrm{D}}_{\mathbf{i}'(p_i)}^{t_i}\overline{\mathrm{D}}_{p_i}^{t_i' }),  
\]for all $\mathbf{i},\mathbf{i}'\in \langle\tilde\iota_1,\tilde\iota_2\rangle$ and  $t_i,  t_i' \in \ZZ_2$. 
\end{proposition}
\begin{proof}
  Define 
\[ G:=
\{ ( { \rm id}, \overline{\mathrm{D}} ), (\tilde\iota_1, 
\overline{\mathrm{D}}_{p_1}\overline{\mathrm{D}}_{p_2}\overline{\mathrm{D}})|\overline{\mathrm{D}}\in\{ { \rm id}, \overline{\mathrm{D}}_{p_3} ,\overline{\mathrm{D}}_{p_4},\overline{\mathrm{D}}_{p_3}\overline{\mathrm{D}}_{p_4} \}
\}\subseteq \langle\tilde\iota_1,\tilde\iota_2\rangle\ltimes_{\mathrm{conj}}\overline{\mathrm{D}}_{\vot},
\]
with multiplication given by
 \[
(\tilde\iota_1^s , \prod_{i=1}^4 \overline{\mathrm{D}}_{p_i}^{t_i}) * (\tilde\iota_1^{s'} ,\prod_{i=1}^4\overline{\mathrm{D}}_{p_i}^{t_i'}) := (  \tilde\iota_1^{s+s'} ,\prod_{i=1}^4 \overline{\mathrm{D}}_{\tilde\iota_1^{s'}(p_i)}^{t_i}\overline{\mathrm{D}}_{p_i}^{t_i'}),  
\]for all  $  t_i ,{s}, t_i', {s'} \in \ZZ_2$.

Let $G'=\<\tilde\iota_2\>\times \<\overline{\mathrm{D}}_{p_1}\overline{\mathrm{D}}_{p_4},\overline{\mathrm{D}}_{p_2}\overline{\mathrm{D}}_{p_3}\>$.
Then  the  group homomorphisms
\[
f_1:\Aut(\CPone)\to G,\nu  \mapsto(\tilde\iota_1^s,\prod_{i=1}^4 \overline{\mathrm{D}}_{p_i}^{t_i}) \text{ and }  f_2: \Pic_0\CPone\to G',\vx\mapsto(\tilde\iota_2^u,\prod_{i=1}^4 \overline{\mathrm{D}}_{p_i}^{v_i}  )
\] defined by Table~\ref{2-iso} are group isomorphisms, where the generators of $\Aut(\CPone)$ are given in \Cref{prop:aut-CPone}.

\begin{table}[htbp]
\centering
\caption{Images of generators under $f_1$ and $f_2$}
\label{2-iso}
\renewcommand{\arraystretch}{1.3}
\setlength{\tabcolsep}{10pt}

\begin{tabular}{cc}
\toprule
$\nu$ &
$f_1(\nu)=\bigl(\tilde\iota_1^{\,s},\,\prod_{i=1}^{4}\overline{\mathrm{D}}_{p_i}^{\,t_i}\bigr)$ \\
\midrule
$(0,1)$ &
$(\mathrm{id},\,\overline{\mathrm{D}}_{p_4})$ \\
$(0,1)(\infty,\frac12)$ &
$(\mathrm{id},\,\overline{\mathrm{D}}_{p_3}\overline{\mathrm{D}}_{p_4})$ \\
$(0,\infty)(1,\frac12)$ &
$(\tilde\iota_1,\,\overline{\mathrm{D}}_{p_1}\overline{\mathrm{D}}_{p_2})$ \\
\bottomrule
\end{tabular}
\hfill
\begin{tabular}{cc}
\toprule
$\vx$ &
$f_2(\vx)=\bigl(\tilde\iota_2^{\,u},\,\prod_{i=1}^{4}\overline{\mathrm{D}}_{p_i}^{\,v_i}\bigr)$ \\
\midrule
$\vec{x}_1-\vec{x}_2$ &
$(\mathrm{id},\,\overline{\mathrm{D}}_{p_1}\overline{\mathrm{D}}_{p_4})$ \\
$\vec{x}_1-\vec{x}_3$ &
$(\tilde\iota_2,\,\overline{\mathrm{D}}_{p_2}\overline{\mathrm{D}}_{p_3})$ \\
$\vec{x}_1-\vec{x}_4$ &
$(\tilde\iota_2,\,\mathrm{id})$ \\
\bottomrule
\end{tabular}
\end{table}
Note that 
 \[f_2(\nu.\vx)=(\tilde\iota_2^u, \prod_{i=1}^4\overline{\mathrm{D}}_{p_i}^{t_i} \overline{\mathrm{D}}_{\tilde\iota_2^u(p_i)}^{t_i} \overline{\mathrm{D}}_{\tilde\iota_1^s(p_i)}^{v_i}),\]
 holds for all $\nu\in\Aut(\CPone) $  and $\vx\in\Pic_0\CPone$, which is verified on generators in Table~\ref{compatible}. 
 \begin{table}[htbp]
\centering
\caption{Values of  $\nu.\vx$ and $(\tilde\iota_2^u, \prod_{i=1}^4\overline{\mathrm{D}}_{p_i}^{t_i} \overline{\mathrm{D}}_{\tilde\iota_2^u(p_i)}^{t_i} \overline{\mathrm{D}}_{\tilde\iota_1^s(p_i)}^{v_i})$ on generators} 
\label{compatible}
\renewcommand{\arraystretch}{1.3}
\setlength{\tabcolsep}{12pt}
 \begin{tabular}{cccc}
\specialrule{\heavyrulewidth}{0pt}{0pt}
\multirow{2}{*}{$\nu$}& \multicolumn{3}{c}{$\vx$} \\[-2pt]
\cmidrule(lr){2-4}
  &$ \vec{x}_1-\vec{x}_2$ & $\vec{x}_1-\vec{x}_3$ & $\vec{x}_1-\vec{x}_4$ \\
\midrule
  $(0,1)$ & $ \vec{x}_1-\vec{x}_2$ & $\vec{x}_2-\vec{x}_3$ & $\vec{x}_2-\vec{x}_4$ \\
$(0,1)(\infty,\frac{1}{2})$ & $ \vec{x}_1-\vec{x}_2$ & $\vec{x}_2-\vec{x}_4$ & $\vec{x}_2-\vec{x}_3$\\ 
$(0,\infty)(1,\frac{1}{2})$ & $  \vec{x}_3-\vec{x}_4$ & $\vec{x}_1-\vec{x}_3$ & $\vec{x}_3-\vec{x}_2$\\
\bottomrule
\end{tabular}
 
\begin{tabular}{cccc}
\specialrule{\heavyrulewidth}{0pt}{0pt}
\multirow{2}{*}{$f_1(\nu)$}& \multicolumn{3}{c}{$f_2(\vx)$} \\[-2pt]
\cmidrule(lr){2-4}
  & $( \mathrm{id},\overline{\mathrm{D}}_{p_1}\overline{\mathrm{D}}_{p_4})$ & $(\tilde\iota_2,\overline{\mathrm{D}}_{p_2}\overline{\mathrm{D}}_{p_3})$ & $(\tilde\iota_2,\mathrm{id})$ \\
\midrule
$( \mathrm{id},\overline{\mathrm{D}}_{p_4})$ & $ ( \mathrm{id},\overline{\mathrm{D}}_{p_1}\overline{\mathrm{D}}_{p_4})$ & $(\tilde\iota_2,\overline{\mathrm{D}}_{p_1}\overline{\mathrm{D}}_{p_2}\overline{\mathrm{D}}_{p_3}\overline{\mathrm{D}}_{p_4})$ & $(\tilde\iota_2,\overline{\mathrm{D}}_{p_1}\overline{\mathrm{D}}_{p_4})$ \\
$( \mathrm{id},\overline{\mathrm{D}}_{p_3}\overline{\mathrm{D}}_{p_4})$ & $ ( \mathrm{id},\overline{\mathrm{D}}_{p_1}\overline{\mathrm{D}}_{p_4})$ & $(\tilde\iota_2,\overline{\mathrm{D}}_{p_1}\overline{\mathrm{D}}_{p_4})$ & $(\tilde\iota_2,\overline{\mathrm{D}}_{p_1}\overline{\mathrm{D}}_{p_2}\overline{\mathrm{D}}_{p_3}\overline{\mathrm{D}}_{p_4})$\\ 
$(\tilde\iota_1,\overline{\mathrm{D}}_{p_1}\overline{\mathrm{D}}_{p_2})$ & $  ( \mathrm{id},\overline{\mathrm{D}}_{p_2}\overline{\mathrm{D}}_{p_3})$ & $(\tilde\iota_2,\overline{\mathrm{D}}_{p_2}\overline{\mathrm{D}}_{p_3})$ & $(\tilde\iota_2,\overline{\mathrm{D}}_{p_1}\overline{\mathrm{D}}_{p_2}\overline{\mathrm{D}}_{p_3}\overline{\mathrm{D}}_{p_4})$\\
\bottomrule
\end{tabular}
\end{table} 
Now we define the  group $G\ltimes_3G'$
with multiplication
\begin{align*}
    &(\tilde\iota_1^{s},\prod_{i=1}^4 \overline{\mathrm{D}}_{p_i}^{t_i},\tilde\iota_2^{u},\prod_{i=1}^4 \overline{\mathrm{D}}_{p_i}^{v_i})*(\tilde\iota_1^{s'}, \prod_{i=1}^4 \overline{\mathrm{D}}_{p_i}^{t_i'},\tilde\iota_2^{u'},\prod_{i=1}^4 \overline{\mathrm{D}}_{p_i}^{v_i'})\\
&\quad = 
(\tilde\iota_1^{s+s'}, \prod_{i=1}^4 \overline{\mathrm{D}}_{\tilde\iota_1^{s}(p_i)}^{t_i}\overline{\mathrm{D}}_{p_i}^{t_i'},\tilde\iota_2^{u+u'},  \prod_{i=1}^4 \overline{\mathrm{D}}_{\tilde\iota_2^{u+u'}(p_i)}^{t_i}\overline{\mathrm{D}}_{\tilde\iota_2^{u'}(p_i)}^{t_i} \overline{\mathrm{D}}_{\tilde\iota_1^{s'}(p_i)}^{v_i }\overline{\mathrm{D}}_{p_i}^{v_i'} ),
\end{align*}
where $t_i,t_i',v_i,v_i', {s},{u}, {s'}, {u'}\in \ZZ_2$. 
Then one may check that the maps
\[
\varrho:\Aut(\CPone) \ltimes_2 \Pic_0\CPone\to G\ltimes_3G',(\nu,\vx)\mapsto(f_1(\nu),f_2(\vx))
\]
and
\[
\varrho':G\ltimes_3G'\to \langle\tilde\iota_1,\tilde\iota_2\rangle \ltimes_{\rm conj} \overline{\mathrm{D}}_{\vot}, (\tilde\iota_1^{s},\prod_{i=1}^4 \overline{\mathrm{D}}_{p_i}^{t_i},\tilde\iota_2^{u},\prod_{i=1}^4 \overline{\mathrm{D}}_{p_i}^{v_i})\mapsto (\tilde\iota_1^{s}\tilde\iota_2^{u}, \prod_{i=1}^4 \overline{\mathrm{D}}_{\tilde\iota_2^{u}(p_i)}^{t_i}\overline{\mathrm{D}}_{p_i}^{v_i}  )
\] 
are group isomorphisms.  In particular, their composition
\begin{equation}\label{eq:rho-normal-subgroup}
\rho\coloneqq  \varrho \circ \varrho':
\Aut(\CPone)\ltimes_2\Pic_0\CPone
\longrightarrow
\langle\tilde\iota_1,\tilde\iota_2\rangle\ltimes_{\rm conj}
\overline{\mathrm{D}}_{\vot}
\end{equation}
is the required group isomorphism.
\end{proof}

 \begin{theorem}\label{autd=mag}
For $\CPone=(\mathbb{P}^1,(0,1,\infty,\frac{1}{2}),(2,2,2,2))$ over $\CC$, there is a
 group isomorphism
 \[ \phi: \Aut\DC{\CPone}  \to     \MCG^\sx(\widetilde{S\x^2} ), \]     such that for any $\xi\in\Aut\DC{\CPone}$ and any indecomposable rigid sheaf $E_{p}^x[n]$ in  $\DC{\CPone}$,  \[\phi(\xi )(\wX^{-1}(E_{p}^x[n]))=\wX^{-1}(\xi (E_{p}^x[n])).\]
 \end{theorem} 
 \begin{proof}
 For an indecomposable rigid object $E_p^x[n]$, we may assume, without loss of generality, that
\[
\widetilde{X}\bigl(\varsigma(\alpha_{p,x},n)\bigr)=E_p^x[n].
\]
Under this assumption, we claim that for any $(\nu,\vx)\in\Aut(\CPone)\ltimes_2\Pic_0\CPone$,
\begin{equation}\label{communication}
(\nu,\vx)(E_p^x[n])=\widetilde{X}(\rho(\nu,\vx)(\varsigma(\alpha_{p,x},n))),
\end{equation}
where $\rho$ is as defined in the proof of Proposition~\ref{keep slope}.

   \begin{table}[htbp]
\centering
\caption{Action of  $\nu\in\Aut(\CPone)$ and $\vx\in\Pic_0\CPone$ on the objects $E_{p}^x[n]$}
\label{aut}
\renewcommand{\arraystretch}{1.3}
\setlength{\tabcolsep}{6pt}

\begin{minipage}[t]{0.48\textwidth}
\centering
\begin{tabular}{ccccc}
\specialrule{\heavyrulewidth}{0pt}{0pt}
\multirow{2}{*}{$\nu$}& \multicolumn{4}{c}{$E_0^x$} \\[-2pt]
\cmidrule(lr){2-5}
  & $E_{0}^1$ & $E_{0}^i$ & $E_{0}^j$ & $E_{0}^k$  \\ 
\midrule
$(0,1)$ & $E_{0}^{1}$ & $E_{0}^{-j}$ & $E_{0}^{-i}$ & $E_{0}^{k}$ \\ 
 
$(0,1)(\infty,\frac{1}{2})$  &$E_{0}^{1}$ & $E_{0}^{-i}$ & $E_{0}^{-j}$ & $E_{0}^{k}$   \\ 
 
  $(0,\infty)(1,\frac{1}{2})$ &$E_{0}^{1}$ & $E_{0}^{i}$ & $E_{0}^{-j}$ & $E_{0}^{-k}$ \\ 
\bottomrule
\end{tabular}
\end{minipage}
\begin{minipage}[t]{0.48\textwidth}
\centering
\begin{tabular}{ccccc}
\specialrule{\heavyrulewidth}{0pt}{0pt}
\multirow{2}{*}{$\vx$}& \multicolumn{4}{c}{$E_0^x$} \\[-2pt]
\cmidrule(lr){2-5}
  & $E_{0}^{1}$ & $E_{0}^{i}$ & $E_{0}^{j}$ & $E_{0}^{k}$  \\ 
\midrule
$\vx_1-\vx_2$ & $E_{0}^{-k}$ & $E_{0}^{-j}$ & $E_{0}^{-i}$ & $E_{0}^{-1}$  \\ 
 
$\vx_1-\vx_3$& $E_{0}^{-i}$ & $E_{0}^{-1}$ & $E_{0}^{-k}$ & $E_{0}^{-j}$   \\ 
 
$\vx_1-\vx_4$ &  $E_{0}^{-j}$ & $E_{0}^{-k}$ & $E_{0}^{-1}$ & $E_{0}^{-i}$  \\ 
\bottomrule
\end{tabular}
\end{minipage}
\vspace{0.3cm} 

\begin{minipage}[t]{0.48\textwidth}
\centering
 \begin{tabular}{ccccc}
\specialrule{\heavyrulewidth}{0pt}{0pt}
\multirow{2}{*}{$\nu$} & \multicolumn{4}{c}{$E_1^x$} \\[-2pt]
\cmidrule(lr){2-5}
  & $E_{1}^1$ & $ E_{1}^i$ & $ E_{1}^j$ & $E_{1}^{k}$   \\ 
\midrule
$(0,1)$ &   $E_{1}^{j}$ & $E_{1}^{i}$ & $E_{1}^{1}$ & $E_{1}^{k}$  \\ 
 
$(0,1)(\infty,\frac{1}{2})$    & $E_{1}^{-j}$ & $E_{1}^{k}$ & $E_{1}^{1}$ & $E_{1}^{i}$   \\ 
 
$(0,\infty)(1,\frac{1}{2})$  & $E_{1}^{k}$ & $E_{1}^{j}$ & $E_{1}^{i}$  & $E_{1}^{1}$   \\ 
\bottomrule
\end{tabular}
\end{minipage}
\begin{minipage}[t]{0.48\textwidth}
\centering
 \begin{tabular}{ccccc}
\specialrule{\heavyrulewidth}{0pt}{0pt}
\multirow{2}{*}{$\vx$}& \multicolumn{4}{c}{$E_1^x$} \\[-2pt]
\cmidrule(lr){2-5}
  & $E_{1}^{1}$ & $E_{1}^{i}$ & $E_{1}^{j}$ & $E_{1}^{k}$ \\
\midrule
$\vx_1-\vx_2$ &  $E_{1}^{j}$ & $E_{1}^{-k}$ & $E_{1}^{1}$ & $E_{1}^{-i}$  \\ 
 
$\vx_1-\vx_3$&  $E_{1}^{k}$ & $E_{1}^{-j}$ & $E_{1}^{-i}$ & $E_{1}^{1}$ \\ 
 
$\vx_1-\vx_4$ &   $E_{1}^{i}$ & $E_{1}^{1}$ & $E_{1}^{-k}$ & $E_{1}^{-j}$     \\ 
\bottomrule
\end{tabular}
\end{minipage}
\vspace{0.3cm} 
 
\begin{minipage}[t]{0.48\textwidth}
\centering
\begin{tabular}{ccccc}
\specialrule{\heavyrulewidth}{0pt}{0pt}
\multirow{2}{*}{$\nu$}& \multicolumn{4}{c}{$E_\infty^x$} \\[-2pt]
\cmidrule(lr){2-5}
  & $ E_{\infty}^1$ & $E_{\infty}^i$ & $E_{\infty}^j$ & $E_{\infty}^k$  \\ 
\midrule
$(0,1)$ &   $E_{\infty}^{-i}$ & $E_{\infty}^{-1}$ & $E_{\infty}^{j}$ & $E_{\infty}^{k}$\\ 
 
$(0,1)(\infty,\frac{1}{2})$    & $E_{\infty}^{-i}$ & $E_{\infty}^{-1}$ & $E_{\infty}^{k}$ & $E_{\infty}^{j}$ \\ 
 
$(0,\infty)(1,\frac{1}{2})$    & $E_{\infty}^{-j}$ & $E_{\infty}^{k}$ & $E_{\infty}^{-1}$ & $E_{\infty}^{i}$ \\ 
\bottomrule
\end{tabular}
\end{minipage}
\begin{minipage}[t]{0.48\textwidth}
\centering
\begin{tabular}{ccccc}
\specialrule{\heavyrulewidth}{0pt}{0pt}
\multirow{2}{*}{$\vx$}& \multicolumn{4}{c}{$E_\infty^x$} \\[-2pt]
\cmidrule(lr){2-5}
  & $E_{\infty}^{1}$ & $E_{\infty}^{i}$ & $E_{\infty}^{j}$ & $E_{\infty}^{k}$ \\
\midrule
$\vx_1-\vx_2$   & $E_{\infty}^{-1}$ & $E_{\infty}^{-i}$ & $E_{\infty}^{j}$ & $E_{\infty}^{k}$\\ 
$\vx_1-\vx_3$&   $E_{\infty}^{-1}$ & $E_{\infty}^{i}$ & $E_{\infty}^{-j}$ & $E_{\infty}^{k}$ \\ 
$\vx_1-\vx_4$ &   $E_{\infty}^{-1}$ & $E_{\infty}^{i}$ & $E_{\infty}^{j}$ & $E_{\infty}^{-k}$ \\ 
\bottomrule
\end{tabular}
 \end{minipage}
\end{table}

 \begin{table}[htbp]
\centering
\small
\caption{Action of  $f_{1}(\nu)$ and $f_2(\vx)$ on the arcs $\varsigma(\alpha_{p,x},n)  $}
\label{a}
\renewcommand{\arraystretch}{1.3}
\setlength{\tabcolsep}{4pt}
\begin{minipage}[t]{0.48\textwidth}
\centering
 \begin{tabular}{ccccc} 
 \specialrule{\heavyrulewidth}{0pt}{0pt}
\multirow{2}{*}{$f_1(\nu)$ }& 
  \multicolumn{4}{c}{$\alpha_{0,x}$} \\[-2pt]
\cmidrule(lr){2-5}
 & $\alpha_{0,1}$ & $\alpha_{0,i}$ & $\alpha_{0,j}$ & $\alpha_{0,k}$   \\ 
\midrule
$( { \rm id}, \overline{\mathrm{D}}_{p_4})$ & $\alpha_{0,1}$ & $\alpha_{0,-j}$ & $\alpha_{0,-i}$ & $\alpha_{0,k}$  \\ 
 
$( { \rm id},\overline{\mathrm{D}}_{p_3}\overline{\mathrm{D}}_{p_4}) $ &$\alpha_{0,1}$ & $\alpha_{0,-i}$ & $\alpha_{0,-j}$ & $\alpha_{0,k}$  \\ 
 
$(\tilde\iota_1,\overline{\mathrm{D}}_{p_1}\overline{\mathrm{D}}_{p_2}) $ &$\alpha_{0,1}$ & $\alpha_{0,i}$ & $\alpha_{0,-j}$ & $\alpha_{0,-k}$  \\ 
\bottomrule
\end{tabular} 
\end{minipage}
\begin{minipage}[t]{0.48\textwidth}
\centering
\begin{tabular}{ccccc}
\specialrule{\heavyrulewidth}{0pt}{0pt}
\multirow{2}{*}{$f_2(\vx)$ }&  \multicolumn{4}{c}{$\alpha_{0,x}$} \\[-2pt]
\cmidrule(lr){2-5}
  & $\alpha_{0,1}$ & $\alpha_{0,i}$ & $\alpha_{0,j}$ & $\alpha_{0,k}$   \\ 
\midrule
$( { \rm id}, \overline{\mathrm{D}}_{p_1}\overline{\mathrm{D}}_{p_4})$ & $\alpha_{0,-k}$ & $\alpha_{0,-j}$ & $\alpha_{0,-i}$ & $\alpha_{0,-1}$ \\ 
 
$(\tilde\iota_2,\overline{\mathrm{D}}_{p_2}\overline{\mathrm{D}}_{p_3}) $ &$\alpha_{0,-i}$ & $\alpha_{0,-1}$ & $\alpha_{0,-k}$ & $\alpha_{0,-j}$  \\ 
 
$(\tilde\iota_2,\mathrm{id}) $ &  $\alpha_{0,-j}$ & $\alpha_{0,-k}$ & $\alpha_{0,-1}$ & $\alpha_{0,-i}$ \\ 
\bottomrule
\end{tabular}
\end{minipage}
\vspace{0.3cm} 

\begin{minipage}[t]{0.48\textwidth}
\centering
 \begin{tabular}{ccccc}
\specialrule{\heavyrulewidth}{0pt}{0pt}
\multirow{2}{*}{$f_1(\nu)$ }& \multicolumn{4}{c}{$\alpha_{1,x}$} \\[-2pt]
\cmidrule(lr){2-5}
  & $\alpha_{1,1}$ & $\alpha_{1,i}$ & $\alpha_{1,j}$ & $\alpha_{1,k}$   \\ 
\midrule
$( { \rm id}, \overline{\mathrm{D}}_{p_4})$   & $\alpha_{1,j}$ & $\alpha_{1,i}$ & $\alpha_{1,1}$ & $\alpha_{1,k}$  \\ 
 
$( { \rm id},\overline{\mathrm{D}}_{p_3}\overline{\mathrm{D}}_{p_4}) $   & $\alpha_{1,-j}$ & $\alpha_{1,k}$ & $\alpha_{1,1}$ & $\alpha_{1,i}$  \\ 
 
$(\tilde\iota_1,\overline{\mathrm{D}}_{p_1}\overline{\mathrm{D}}_{p_2}) $   & $\alpha_{1,k}$ & $\alpha_{1,j}$ & $\alpha_{1,i}$ & $\alpha_{1,1}$  \\ 
\bottomrule
\end{tabular} 
\end{minipage}
\begin{minipage}[t]{0.48\textwidth}
\centering
\begin{tabular}{ccccc}
\specialrule{\heavyrulewidth}{0pt}{0pt}
\multirow{2}{*}{$f_2(\vx)$ }& \multicolumn{4}{c}{$\alpha_{1,x}$} \\[-2pt]
\cmidrule(lr){2-5}
  & $\alpha_{1,1}$ & $\alpha_{1,i}$ & $\alpha_{1,j}$ & $\alpha_{1,k}$   \\ 
\midrule
$( { \rm id}, \overline{\mathrm{D}}_{p_1}\overline{\mathrm{D}}_{p_4})$  & $\alpha_{1,j}$ & $\alpha_{1,-k}$ & $\alpha_{1,1}$ & $\alpha_{1,-i}$ \\ 
 
$(\tilde\iota_2,\overline{\mathrm{D}}_{p_2}\overline{\mathrm{D}}_{p_3}) $  & $\alpha_{1,k}$ & $\alpha_{1,-j}$ & $\alpha_{1,-i}$ & $\alpha_{1,1}$   \\ 
 
$(\tilde\iota_2,\mathrm{id}) $   & $\alpha_{1,i}$ & $\alpha_{1,1}$ & $\alpha_{1,-k}$ & $\alpha_{1,-j}$    \\ 
\bottomrule
\end{tabular}
\end{minipage}
\vspace{0.3cm} 

\begin{minipage}[t]{0.48\textwidth}
\centering
 \begin{tabular}{ccccc}
\specialrule{\heavyrulewidth}{0pt}{0pt}
\multirow{2}{*}{$f_1(\nu)$ }& \multicolumn{4}{c}{$\alpha_{\infty,x}$} \\[-2pt]
\cmidrule(lr){2-5}
  &  $\alpha_{\infty,1}$ & $\alpha_{\infty,i}$ & $\alpha_{\infty,j}$ & $\alpha_{\infty,k}$ \\
\midrule
$( { \rm id}, \overline{\mathrm{D}}_{p_4})$ &    $\alpha_{\infty,-i}$ & $\alpha_{\infty,-1}$ & $\alpha_{\infty,j}$ & $\alpha_{\infty,k}$\\ 
 
$( { \rm id},\overline{\mathrm{D}}_{p_3}\overline{\mathrm{D}}_{p_4}) $  & $\alpha_{\infty,-i}$ & $\alpha_{\infty,-1}$ & $\alpha_{\infty,k}$ & $\alpha_{\infty,j}$ \\ 
 
$(\tilde\iota_1,\overline{\mathrm{D}}_{p_1}\overline{\mathrm{D}}_{p_2}) $ & $\alpha_{\infty,-j}$ & $\alpha_{\infty,k}$ & $\alpha_{\infty,-1}$ & $\alpha_{\infty,i}$ \\ 
\bottomrule
\end{tabular} 
\end{minipage}
\begin{minipage}[t]{0.48\textwidth}
\centering
\begin{tabular}{ccccc}
\specialrule{\heavyrulewidth}{0pt}{0pt}
\multirow{2}{*}{$f_2(\vx)$ }& \multicolumn{4}{c}{$\alpha_{\infty,x}$} \\[-2pt]
\cmidrule(lr){2-5}
  &  $\alpha_{\infty,1}$ & $\alpha_{\infty,i}$ & $\alpha_{\infty,j}$ & $\alpha_{\infty,k}$ \\
\midrule
$( { \rm id}, \overline{\mathrm{D}}_{p_1}\overline{\mathrm{D}}_{p_4})$ &  $\alpha_{\infty,-1}$ & $\alpha_{\infty,-i}$ & $\alpha_{\infty,j}$ & $\alpha_{\infty,k}$\\ 
 
$(\tilde\iota_2,\overline{\mathrm{D}}_{p_2}\overline{\mathrm{D}}_{p_3}) $   & $\alpha_{\infty,-1}$ & $\alpha_{\infty,i}$ & $\alpha_{\infty,-j}$ & $\alpha_{\infty,k}$ \\ 
 
$(\tilde\iota_2,\mathrm{id}) $ &      $\alpha_{\infty,-1}$ & $\alpha_{\infty,i}$ & $\alpha_{\infty,j}$ & $\alpha_{\infty,-k}$ \\ 
\bottomrule
\end{tabular}
\end{minipage}
 
\end{table}
In fact, Table~\ref{aut} and  Table~\ref{a}
      confirm  that  Equation~\eqref{communication} holds 
 for   $p\in\{0,1,\infty\}$.   To establish the general case, we proceed as follows. First,  for any $g\in \langle\tilde\iota_1,\tilde\iota_2\rangle \ltimes_{\rm conj}\overline{\mathrm{D}}_{\vot} $, suppose   $g( \varsigma(\alpha_{\typ{p},x},n)   )= \varsigma(\alpha_{\typ{p},y},n)  $.    Then by Definition~\ref{def:tarc} and  \Cref{bi:ba-ta}, we have   
$$g(\varsigma(\alpha_{{p},x},n)    )= \varsigma(\alpha_{ {p},y},n) . $$ 
Second,  for any $p=\frac{a(p)}{b(p)}\in\QQi$,  we have the decomposition
\[
\bv_{p}^{x}=\bv_{\typ{p}}^{x} +\lfloor\frac{b(p)}{2}\rfloor \mathbf{h}_0+\lfloor\frac{a(p)}{2}\rfloor \mathbf{h}_{\infty}.
\]   Consequently,  if  $(\nu,\vx)(E_{\typ{p}}^{x}[n] )=E_{\typ{p}}^{z}[n]$,  then it follows that
$(\nu,\vx)(E_{p}^{x}[n])=E_{p}^{z}[n]$. 
This completes the proof of the claim.
 
Now we define group  homomorphism
\[
\phi :\Br_3\ltimes_{\rm conj}(\Aut(\CPone) \ltimes_2 \Pic_0\CPone)\to \Br_3\ltimes_{\rm conj}(\langle\tilde\iota_1,\tilde\iota_2\rangle \ltimes_{\rm conj}\overline{\mathrm{D}}_{\vot}).
\]
Define a  group  homomorphism $\beta:\Br_3\longrightarrow\<B_\infty,B_0\>$ satisfying
\[
\beta(\tubR)=B_\infty,\quad \beta(\tubL)=B_0.
\]
The braid relation for $\tubR,\tubL$  together with
\Cref{prop:mcgS42}  imply that  $\beta$ is an isomorphism. 
 
For
$s\in\{\tubR,\tubL\}$ and $\vartheta\in \Aut(\CPone) \ltimes_2 \Pic_0\CPone$, the equation
\eqref{communication}, together with
\Cref{cor:graded-tag-action} and \Cref{tubular left mutations}, yield
\begin{equation}\label{eq:rho-equivariant}
\rho(s^{-1}\vartheta s)
=\beta(s)^{-1}\rho(\vartheta)\beta(s).
\end{equation}
Every element of $\Aut\DC{\CPone}$ has a unique form $h\vartheta$, with
$h\in\Br_3$ and $\vartheta\in \Aut(\CPone) \ltimes_2 \Pic_0\CPone$. We define
\begin{equation}\label{eq:phi-normal-form}
\phi(h\vartheta)\coloneqq\beta(h)\rho(\vartheta).
\end{equation}
For $h,h'\in\Br_3$ and $\vartheta,\vartheta'\in \Aut(\CPone) \ltimes_2 \Pic_0\CPone$, the multiplication
convention for the semidirect product and \eqref{eq:rho-equivariant} give
\[
\begin{aligned}
\phi\bigl((h\vartheta)(h'\vartheta')\bigr)
&=\phi\bigl(hh'(h'^{-1}\vartheta h')\vartheta'\bigr)\\
&=\beta(hh')\rho(h'^{-1}\vartheta h')\rho(\vartheta')\\
&=\beta(h)\rho(\vartheta)\beta(h')\rho(\vartheta')\\
&=\phi(h\vartheta)\phi(h'\vartheta').
\end{aligned}
\]
Hence $\phi$ is a homomorphism. Since $\beta$ and $\rho$ are isomorphisms, $\phi$ is a group isomorphism.

Finally, \eqref{communication},
\Cref{cor:graded-tag-action}, and
\Cref{tubular left mutations} imply  that
\[
\phi(\xi)\bigl(\wX^{-1}(E_p^x[n])\bigr)
=\wX^{-1}\bigl(\xi(E_p^x[n])\bigr),
\] for
 any $\xi\in\Aut\DC{\CPone}$    and  indecomposable rigid object $E_p^x[n]$.
\end{proof}

\section{A surface model  for the rigid objects in cluster category $\CCP$}\label{sec:5}
 In this  section, we provide a geometric realization of the 
 indecomposable  rigid  objects in $\CCP$ via tagged arcs on a sphere with four punctures $S_{0,4}$, and prove that 
 the dimensions of Hom-spaces between any two such objects equal the tagged intersection numbers of the corresponding tagged arcs.
Moreover, we show that   if  $\CPone$ is the  weighted projective line $$(\mathbb{P}^1,(0,1,\infty,\frac{1}{2}),(2,2,2,2))$$
  over the  field  $\CC$,
 the automorphism group of the cluster category $\CCP$ is isomorphic to the tagged mapping class group $\MCG^\times(S_{0,4})$ of the sphere with four punctures $S_{0,4}$. 
  
Let 
  $ \Ind^\circ\CCP$  denote     the set of isomorphism classes of   indecomposable rigid objects in $\CCP$. 

\begin{theorem}\label{thm:C}
There is a bijection
$$X_c\colon \TA^\times(S_{0,4}) \to \Ind^\circ\CCP$$
sending a  simple arc $\alpha_{p,x}$ to  the rigid object ${E}_p^x$.
\end{theorem}
\begin{proof}
Since $\TA^\times(S_{0,4})=\{\alpha_{p,x}|p\in\QQi,x \in\bfH\}$, 
this theorem follows from  the fact  that  the rigid objects in  $\CCP$  are precisely  
$$\{E_p^x |p\in\QQi, x\in \bfH\},$$
where $E_{p}^{x}$ is the unique (up to isomorphism) rigid sheaf over $\CPone$ satisfying $ [E^x_p]=\bv_p^x$.
\end{proof}

Now we turn to study the dimensions of Hom-spaces between
two indecomposable rigid objects. We first introduce the following definitions.

 \begin{definition}\label{def:oint}
Let $(\gamma_1,\kappa_1)$ and $(\gamma_2,\kappa_2)$ be two tagged arcs on $\surf$ in minimal position.
The \emph{tagged intersection of index $0$}, denoted by $\cap^0((\gamma_1,\kappa_1),(\gamma_2,\kappa_2))$,    
consists of intersection points $a$ between the arcs of the following three types:
\begin{enumerate}
    \item[(I)] The point $a$ lies in the interior of $\surf$;
    \item[(II)] The point $a$ is a puncture, $\kappa_1(a)=\kappa_2(a)$ and $\gamma_1\ne \gamma_2$.
    \item[(III)] The point $a$ is a puncture, $\kappa_1(a)=\kappa_2(a)$ and $(\gamma_1,\kappa_1)=(\gamma_2,\kappa_2)$.
\end{enumerate}
Similarly, the \emph{tagged intersection of index $1$}, denoted by
$\cap^1((\gamma_1,\kappa_1),(\gamma_2,\kappa_2))$,  
consists of intersection points $a$ of the following three types:
\begin{enumerate}
    \item[(I)] The point $a$ lies in the interior of $\surf$;
    \item[(II)] The point $a$ is a puncture,  $\kappa_1(a)\neq\kappa_2(a)$  and $\gamma_1\ne \gamma_2$.
    \item[(III)] The point $a$ is a puncture, $\kappa_1(a)\ne\kappa_2(a)$ and $(\gamma_1,\kappa_1)=(\gamma_2,-\kappa_2)$.
\end{enumerate}
\end{definition}
   
\begin{definition}
   For  any two tagged  arcs   $(\gamma_1,\kappa_1),(\gamma_2,\kappa_2)$ on $\surf$, and for any $d\in\ZZ$,  the \emph{tagged intersection number}  of index $d$ between  $(\gamma_1,\kappa_1)$ and $(\gamma_2,\kappa_2)$    is   defined by
    \[
     \Int^d((\gamma_1,\kappa_1),(\gamma_2,\kappa_2))=\begin{cases}
         |\cap^0((\gamma_1,\kappa_1),(\gamma_2,\kappa_2))|, & \text{if $d$ is even,}\\
         |\cap^1((\gamma_1,\kappa_1),(\gamma_2,\kappa_2))|, & \text{otherwise.} 
     \end{cases}
    \]
\end{definition}

Using the bijection $X_c$ from Theorem~\ref{thm:C}, we obtain the theorem.

\begin{theorem}\label{thm:C2}
For any $d\in\mathbb{Z}$,  $p,q\in\QQi$  and $x,y\in\bfH$,   the following equality holds:
\[
    \Int^d( \alpha_{p,x},\alpha_{q,y})=\dim\Hom_\cC  ( {E}_p^x, {E}_q^y[d]  ).
\]
\end{theorem}
\begin{proof} 
Since $\coh(\CPone)$ is hereditary and $\tau {E}_q^y={E}_q^{-y}$,  
\[
\Hom_\cC  (  {E}_p^x ,  {E}_q^y [d]  )=\begin{cases}
     \Hom_{\cD} (  E_p^x[0] , E_q^y[0]   )\oplus\Hom_{\cD} (  E_p^x[0]  ,  E_q^{-y}[1]  ), & \text{$d$ is even,}\\
     \Hom_{\cD} (  E_p^x[0]  , E_q^y[1]   )\oplus\Hom_{\cD} (  E_p^x[0]  ,  E_q^{-y}[0]  ), & \text{otherwise.} 
\end{cases}
\]
In particular, for $p=q$, we have   
\[
\dim\Hom_\cC  (   {E}_p^x ,  {E}_p^y  )=\begin{cases}
2 & \text{if $y=x$,}\\
 0 & \text{otherwise;}   
\end{cases}
\text{ and }
\dim\Hom_\cC  (  {E}_p^x , {E}_p^y[1]   )=\begin{cases}
2 & \text{if $y=-x$,}\\
 0 & \text{otherwise.}   
\end{cases}
\]

Let $s$ (resp. $t$) denote the number of common endpoints $a$ with $\kappa_1(a)=\kappa_2(a)$ (resp. $\kappa_1(a)\ne\kappa_2(a)$). By   Theorem~\ref{thm:X2}, we obtain that
\[
\dim\Hom_\cC (   {E}_p^x ,  {E}_q^y  )= i(\alpha_{p,x}^\circ,\alpha_{q,y}^\circ)+s\text{ and }
\dim\Hom_\cC  (   {E}_p^x ,  {E}_q^y [1]  )= i(\alpha_{p,x}^\circ,\alpha_{q,y}^\circ)+t.
\]
Thus the result follows from Definition~\ref{def:oint}.
\end{proof}

This theorem implies that a set   of tagged arcs is a tagged triangulation of $S_{0,4}$  if and only if its image under $X_c $ is a cluster-tilting object in $\CCP$.  Consequently, the number of (pairwise non-isomorphic) indecomposable direct summands of any cluster-tilting object in   $\CCP$  is  6. 

Moreover,  the bijection $X_c$  induces an isomorphism 
 between the exchange graph 
 of $S_{0,4}$ and the exchange graph 
 for cluster-tilting objects in  $\CCP$,  which has been previously stated in \cite[Theorem 5.11]{BG}.
Next, we consider its local shape.
 \begin{proposition}\label{prop:commutation-relations} 
 Let $T=\bigoplus_{i=1}^{6}T_{i}$ be  a cluster-tilting object  in  $\CCP$, with tagged arc $(\gamma_i,\kappa_i)$ associated to $T_i$.    Fix a lift  $\overline{\gamma}_i \subset \mathbb{R}^2$   of  $\gamma_i$, and denote by $\mu_i(T)$   the mutation of $T$ at $T_i$. 
Then  for any distinct  $1\le i,j\le 6$, one of the following holds:

\begin{itemize}
    \item[(1)]  
 If $\gamma_i$ and $\gamma_j$ are disjoint (in particular, if $\gamma_i=\gamma_j$), or  if the lifts $\overline{\gamma}_i,\overline{\gamma}_j$   are as shown in \Cref{fig:commutation-cases}(I), then  \[\mu_{i}\mu_{j}(T)=\mu_{j}\mu_{i}(T).\]

    \item[(2)] If the lifts $\overline{\gamma}_i,\overline{\gamma}_j$ are arranged as in \Cref{fig:commutation-cases}(II) or (III),   then  
    \[\mu_i \mu_j \mu_i(T) = \mu_j \mu_i  (T).\]
\end{itemize}
 \end{proposition} 
 
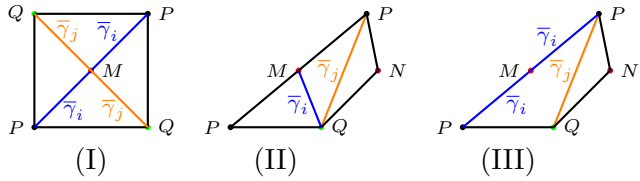
\begin{figure}[h]
     \centering
         \begin{tikzpicture}[scale=1.5 , >=Stealth]
    \filldraw[black] (0,0) circle (0.6pt);
    \filldraw[green] ( 1,0) circle (0.6pt);
    \filldraw[black] ( 1,1) circle (0.6pt);
    \filldraw[green] (0,1) circle (0.6pt);
     \filldraw[red] (0.5,0.5) circle (0.6pt);
     \node[left] at (0,0) {\tiny$P$};
     \node[right] at ( 1,1) {\tiny$P$};
      \node[left] at (0,1) {\tiny$Q$};
     \node[right] at ( 1,0) {\tiny$Q$};
     \node[right] at (0.5,0.5) {\tiny$M$};
      \node[blue] at (0.35,0.15) {\scriptsize$ \overline{\gamma}_i$};
     \node[orange] at (0.7,0.15) {\scriptsize$\overline{\gamma}_j$};
         \node[blue] at (0.65,0.85) {\scriptsize$ \overline{\gamma}_i$};
     \node[orange] at (0.3,0.85) {\scriptsize$\overline{\gamma}_j$};
    \draw[  thick ] (0,0) --  (0,1);
    \draw[ thick] ( 1,0) --  ( 1,1);
   \draw[  thick] (0,1) --  (  1,1);
    \draw[  thick ] (0,0) --  (  1,0); 
    \draw[  thick, blue] (0,0) --  (  1,1);
 \draw[ thick,orange ] ( 1,0) --  (0,1);
  \node  at (0.5,-0.3) {(I)};
\end{tikzpicture}  
\begin{tikzpicture}[scale=1.5  , >=Stealth]
    \filldraw[black] (0,0) circle (0.6pt);
    \filldraw[green] (0.8,0) circle (0.6pt);
    \filldraw[black] (1.2,1) circle (0.6pt);
     \filldraw[red] (0.6,0.5) circle (0.6pt);
      \filldraw[purple] (1.3,0.5) circle (0.6pt);
     \node[left] at (0,0) {\tiny$P$};
     \node[right] at (1.2,1) {\tiny$P$};
     \node[right] at (0.8,0) {\tiny$Q$};
     \node[left] at (0.6,0.5) {\tiny$M$};
      \node[right] at (1.3,0.5) {\tiny$N$};
      \node[orange] at (.87,0.5) {\scriptsize$\overline{\gamma}_j$};
     \node[blue] at (0.6,0.2) {\scriptsize$\overline{\gamma}_i$};
    \draw[ thick,orange] (0.8,0) --  (1.2,1);
    \draw[  thick] (0,0) --  ( 0.8,0); 
    \draw[  thick] (0,0) --  ( 1.2,1);
 \draw[ thick,blue ] (0.8,0) --  (0.6,0.5);
 \draw  [thick](0.8,0) -- (1.3,0.5);
  \draw [thick](1.2,1) -- (1.3,0.5);
  \node  at (0.4,-0.3) {(II)};
\end{tikzpicture}  
 \begin{tikzpicture}[scale=1.5 , >=Stealth]
    \filldraw[black] (0,0) circle (0.6pt);
    \filldraw[green] (0.8,0) circle (0.6pt);
    \filldraw[black] (1.2,1) circle (0.6pt);
     \filldraw[red] (0.6,0.5) circle (0.6pt);
      \filldraw[purple] (1.3,0.5) circle (0.6pt);
     \node[left] at (0,0) {\tiny$P$};
     \node[right] at (1.2,1) {\tiny$P$};
     \node[right] at (0.8,0) {\tiny$Q$};
     \node[left] at (0.6,0.5) {\tiny$M$};
      \node[right] at (1.3,0.5) {\tiny$N$};
      \node[orange] at (.87,0.5) {\scriptsize$\overline{\gamma}_j$};
     \node[blue] at (0.48,0.2) {\scriptsize$\overline{\gamma}_i$};
     \node[blue] at (0.75,0.8) {\scriptsize$\overline{\gamma}_i$};
    \draw[ thick,orange] (0.8,0) --  (1.2,1);
    \draw[  thick] (0,0) --  ( 0.8,0); 
    \draw[ blue, thick] (0,0) --  ( 1.2,1);
 \draw  [thick](0.8,0) -- (1.3,0.5);
  \draw [thick](1.2,1) -- (1.3,0.5);
   \node  at (0.4,-0.3) {(III)};
\end{tikzpicture} 
 \caption{All possible geometric configurations of the lifts $\overline{\gamma}_i $ and $\overline{\gamma}_j $ in the universal cover $\mathbb{R}^2$,  they share an endpoint.}
    \label{fig:commutation-cases}
\end{figure}
 \begin{proof}
Note that if  $\gamma_i$ and $\gamma_j$ belong to different triangles of a tagged triangulation, then the mutations $\mu_i(T)$ and $\mu_j(T)$ are independent; consequently, $\mu_i\mu_j(T)=\mu_j\mu_i(T)$.
  
In the following diagrams, we adopt the  convention: 
for a tagged arc $(\gamma,\kappa)$, an endpoint $p$ with $\kappa(p)=1$ is marked by the symbol $\times$, while the case $\kappa(p)=-1$ is indicated by the absence of such a mark.

 (1)  $\gamma_i = \gamma_j$. 
By definition,    we may assume that the lifts $\overline{\gamma}_i$ and $\overline{\gamma}_j$ in the universal cover are positioned as shown in the left‑most diagram of \Cref{fig:i}.   Applying the flips $\mu_i$   and $\mu_j$   in either order yields the same tagged triangulation.
Therefore, if $\gamma_i=\gamma_j$  or if the configuration of their lifts  $\overline{\gamma}_i$ and $\overline{\gamma}_j$  in $\mathbb{R}^2$ is as shown in \Cref{fig:commutation-cases}(I), then
 $\mu_{i}\mu_{j}(T)=\mu_{j}\mu_{i}(T)$.
 \begin{figure}[h]
     \centering
      \begin{tikzpicture}  [scale=1.5, >=Stealth]
    \filldraw[black] (0,0) circle (0.5pt);
    \filldraw[green] (1,0) circle (0.5pt);
    \filldraw[black] (1,1) circle (0.5pt);
    \filldraw[green] (0,1) circle (0.5pt);
     \filldraw[red] (0.5,0.5) circle (0.5pt); 
       \node[left] at (0,0) {\tiny$P$};
     \node[right] at (1,1) {\tiny$P$};
      \node[left] at (0,1) {\tiny$Q$};
     \node[right] at (1,0) {\tiny$Q$};
     \node[right] at (0.5,0.5) {\tiny$M$};
      \node[blue] at (0.2,0.4) {\tiny$ \overline{\gamma}_i$};
     \node[orange] at (0.55,0.15) {\tiny$\overline{\gamma}_j$};
    \draw[  thick ] (0,0) --  (0,1);
    \draw[ thick] (1,0) --  (1,1);
   \draw[  thick] (0,1) --  ( 1,1);
    \draw[  thick ] (0,0) --  ( 1,0); 
    \draw[  thick, blue] (0,0) --  ( 1,1);
  \draw [thick ,orange](0,0) .. controls
(0.45 ,0.25)  ..(0.5,0.5);
 \draw [thick ,orange](1,1) .. controls
(0.55 ,0.75)  ..(0.5,0.5);
\node[orange] at (0.49,0.4) {\tiny$\times$};
\node[orange] at (0.51,0.6) {\tiny$\times$};
\draw[-> ] (1.3,.5) --  (  1.8,.5);
      \node[above] at (1.55,0.5) {\small$\mu_i$};
     \draw[-> ] (0.5,-0.15) --  (  0.5,-0.55);
      \node[above] at (0.7,-0.5) {\small$\mu_j$}; 
        \filldraw[green] ( 2.1+1,0) circle (0.5pt);
    \filldraw[black] ( 2.1+1,1) circle (0.5pt);
    \filldraw[green] ( 2.1+0,1) circle (0.5pt);
     \filldraw[red] ( 2.1+0.5,0.5) circle (0.5pt); 
      \node[blue] at ( 2.1+0.65,0.2) {\tiny$ \overline{\gamma}_i$};
     \node[orange] at ( 2.1+0.2,0.3) {\tiny$\overline{\gamma}_j$};
    \draw[  thick ] ( 2.1+0,0) --  ( 2.1+0,1);
    \draw[ thick] ( 2.1+1,0) --  ( 2.1+1,1);
   \draw[  thick] ( 2.1+0,1) --  ( 2.1+ 1,1);
    \draw[  thick ] ( 2.1+0,0) --  ( 2.1+ 1,0); 
  \draw[ thick,blue] ( 2.1+1,0) --  ( 2.1+0,1);
  \node[blue] at ( 2.1+0.6,0.4) {\tiny$\times$};
\node[blue] at ( 2.1+0.4,0.6) {\tiny$\times$};
  \draw [thick ,orange](2.1+0,0) .. controls
(2.1+0.45 ,0.25)  ..(2.1+0.5,0.5);
 \draw [thick ,orange](2.1+1,1) .. controls
(2.1+0.55 ,0.75)  ..(2.1+0.5,0.5);
\node[orange] at (2.1+0.49,0.4) {\tiny$\times$};
\node[orange] at (2.1+0.51,0.6) {\tiny$\times$};
\draw[-> ] ( 2.1+0.5,-0.15) --  (  2.1+ 0.5,-0.55);
      \node[above] at ( 2.1+0.7,-0.5) {\small $\mu_j$ };
     \filldraw[black] (0,0 -2+0.3)circle (0.5pt);
    \filldraw[green] (1,0 -2+0.3)circle (0.5pt);
    \filldraw[black] (1,1 -2+0.3)circle (0.5pt);
    \filldraw[green] (0,1 -2+0.3)circle (0.5pt);
     \filldraw[red] (0.5,0.5 -2+0.3)circle (0.5pt); 
      \node[blue] at (0.4,0.2 -2+0.3){\tiny$ \overline{\gamma}_i$};
     \node[orange] at (0.65,0.2 -2+0.3){\tiny$\overline{\gamma}_j$};
    \draw[  thick ] (0,0 -2+0.3)--  (0,1 -2+0.3);
    \draw[ thick] (1,0 -2+0.3)--  (1,1 -2+0.3);
   \draw[  thick] (0,1 -2+0.3)--  ( 1,1 -2+0.3);
    \draw[  thick ] (0,0 -2+0.3)--  ( 1,0 -2+0.3); 
    \draw[  thick, blue] (0,0 -2+0.3)--  ( 1,1 -2+0.3);
   \draw[ thick,orange ] (1,0 -2+0.3)--  (0,1-2+0.3);
\draw[-> ] (1.3,.5 -2+0.3)--  (  1.8,.5 -2+0.3);
      \node[above] at (1.55,0.5 -2+0.3){\small $\mu_i$ }; 
       \filldraw[black] ( 2.1+0,0 -2+0.3)circle (0.5pt);
    \filldraw[green] ( 2.1+1,0 -2+0.3)circle ( 0.5pt);
    \filldraw[black] ( 2.1+1,1 -2+0.3)circle ( 0.5pt);
    \filldraw[green] ( 2.1+0,1 -2+0.3)circle (0.5pt);
     \filldraw[red] ( 2.1+0.5,0.5 -2+0.3)circle ( 0.5pt); 
    \draw[  thick ] ( 2.1+0,0 -2+0.3)--  ( 2.1+0,1 -2+0.3);
    \draw[ thick] ( 2.1+1,0 -2+0.3)--  ( 2.1+1,1 -2+0.3);
   \draw[  thick] ( 2.1+0,1 -2+0.3)--  ( 2.1+ 1,1 -2+0.3);
    \draw[  thick ] ( 2.1+0,0 -2+0.3)--  ( 2.1+ 1,0 -2+0.3); 
\draw[ thick,] ( 2.1+1,0-2+0.3) --  ( 2.1+0,1-2+0.3);
  \draw [thick ]( 2.1+1,0-2+0.3).. controls
( 2.1+0.55 ,0.25 -2+0.3) ..( 2.1+0.5,0.5  -2+0.3);

  \draw [thick ]( 2.1+0,1-2+0.3).. controls
( 2.1+0.55 ,0.75 -2+0.3) ..( 2.1+0.5,0.5  -2+0.3); 
\node  at ( 2.1+0.52,0.4 -2+0.3){\tiny$\times$};
\node  at ( 2.1+0.52,0.55 -2+0.3){\tiny$\times$};
\end{tikzpicture}  
 \caption{$\gamma_i = \gamma_j$}
    \label{fig:i}
\end{figure}
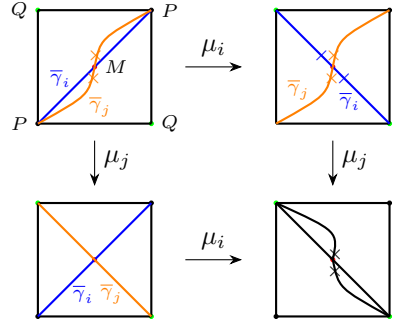

(2) The lifts $\overline{\gamma}_i$ and $\overline{\gamma}_j$ are configured as in \Cref{fig:commutation-cases}(II).  
Assume they are positioned as in the left‑most diagram of \Cref{fig:ii}. 
In this situation the flips satisfy the relation
$\mu_i \mu_j \mu_i(T) = \mu_j \mu_i  (T),$
as is verified by following the diagram.

  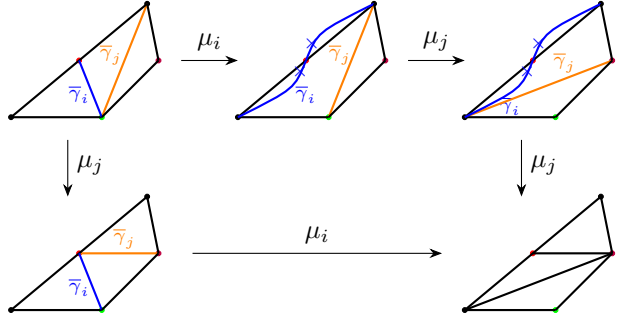
\begin{figure}[h]
    \centering
    \begin{tikzpicture}[scale=1.5 , >=Stealth]
    \filldraw[black] (0,0) circle (0.6pt);
    \filldraw[green] (0.8,0) circle (0.6pt);
    \filldraw[black] (1.2,1) circle (0.6pt);
     \filldraw[red] (0.6,0.5) circle (0.6pt);
      \filldraw[purple] (1.3,0.5) circle (0.6pt);
      \node[orange] at (1-0.12,0.55) {\tiny$\overline{\gamma}_j$};
     \node[blue] at (0.6,0.2) {\tiny$\overline{\gamma}_i$};
    \draw[ thick,orange] (0.8,0) --  (1.2,1);
    \draw[  thick] (0,0) --  ( 0.8,0); 
    \draw[  thick] (0,0) --  ( 1.2,1);
  \draw[ thick,blue ] (0.8,0) --  (0.6,0.5);
 \draw  [thick] (0.8,0) -- (1.3,0.5);
  \draw [thick] (1.2,1) -- (1.3,0.5);
 \draw[-> ] (1.5,.5) --  (  2,.5);
      \node[above] at (1.75,0.5) {\small$\mu_i$};
     \draw[-> ] (0.5,-0.2) --  (  0.5,-0.7);
      \node[above] at (0.7,-0.6) {\small$\mu_j$}; 

 \filldraw[black]  (2+0,0) circle  ( 0.6pt);
    \filldraw[green]  (2+0.8,0) circle  ( 0.6pt);
    \filldraw[black]  (2+1.2,1) circle  ( 0.6pt);
     \filldraw[red]  (2+0.6,0.5) circle  ( 0.6pt);
      \filldraw[purple]  (2+1.3,0.5) circle  ( 0.6pt);
      \node[orange] at  (2+.9,0.55) {\tiny$\overline{\gamma}_j$};
     \node[blue] at  (2+0.6,0.2) {\tiny$\overline{\gamma}_i$};
    \draw[ thick,orange]  (2+0.8,0) --   (2+1.2,1);
    \draw[  thick]  (2+0,0) --   (2+ 0.8,0); 
    \draw[  thick]  (2+0,0) --   (2+ 1.2,1);
 \draw [thick ,blue](2+0,0) .. controls
(2+0.5 ,0.25)  ..(2+0.6,0.5);
 \draw [thick ,blue](2+1.2,1) .. controls
(2+0.7 ,0.75)  ..(2+0.6,0.5);
\node[blue] at (2+0.55,0.4) {\tiny$\times$};
\node[blue] at (2+0.65,0.65) {\tiny$\times$};
 \draw  [thick] (2 +0.8,0) --  (2+1.3,0.5);
  \draw [thick] (2 +1.2,1) --  (2+1.3,0.5);
 \draw[-> ]  (2+1.5,.5) --   (2+  2,.5);
      \node[above] at  (2+1.75,0.5) {\small$\mu_j$}; 
  \filldraw[black]  (4+0,0) circle  ( 0.6pt);
    \filldraw[green]  (4+0.8,0) circle  ( 0.6pt);
    \filldraw[black]  (4+1.2,1) circle  ( 0.6pt);
     \filldraw[red]  (4+0.6,0.5) circle  ( 0.6pt);
      \filldraw[purple]  (4+1.3,0.5) circle  ( 0.6pt);
      \node[orange] at  (4+0.88,0.48) {\tiny$\overline{\gamma}_j$};
     \node[blue] at  (4+0.4,0.1) {\tiny$\overline{\gamma}_i$};
    \draw[ thick,orange]  (4+0,0) --   (4+1.3,0.5);
    
    \draw[  thick]  (4+0,0) --   (4+ 0.8,0); 
    \draw[  thick]  (4+0,0) --   (4+ 1.2,1);
  \draw [thick ,blue](4+0,0) .. controls
(4+0.5 ,0.25)  ..(4+0.6,0.5);
 \draw [thick ,blue](4+1.2,1) .. controls
(4+0.7 ,0.75)  ..(4+0.6,0.5);
\node[blue] at (4+0.55,0.4) {\tiny$\times$};
\node[blue] at (4+0.65,0.65) {\tiny$\times$};

 \draw  [thick] (4+0.8,0) --  (4+1.3,0.5);
  \draw [thick] (4+1.2,1) --  (4+1.3,0.5);
           \draw[-> ] (4+0.5,-0.2) --  (  4+0.5,-0.7);
      \node[above] at (4+0.7,-0.6) {\small$\mu_j$}; 
    \filldraw[black]  (0,0-2+0.3) circle  ( 0.6pt );
    \filldraw[green]  (0.8,0-2+0.3) circle  ( 0.6pt);
    \filldraw[black]  (1.2,1-2+0.3) circle  ( 0.6pt);
     \filldraw[red]  (0.6,0.5-2+0.3) circle  ( 0.6pt);
      \filldraw[purple]  (1.3,0.5-2+0.3) circle  ( 0.6pt);
      \node[orange] at  (1,0.6-2+0.33) {\tiny$\overline{\gamma}_j$};
     \node[blue] at  (0.6,0.2-2+0.3) {\tiny$\overline{\gamma}_i$};
    \draw[ thick,orange]  (0.6,0.5-2+0.3) --   (1.3,0.5-2+0.3);
    
    \draw[  thick]  (0,0-2+0.3) --   ( 0.8,0-2+0.3); 
    \draw[  thick]  (0,0-2+0.3) --   ( 1.2,1-2+0.3);
 \draw[ thick,blue ]  (0.8,0-2+0.3) --   (0.6,0.5-2+0.3);
 \draw  [thick] (0.8,0-2+0.3) --  (1.3,0.5-2+0.3);
  \draw [thick] (1.2,1-2+0.3) --  (1.3,0.5-2+0.3);
   \draw[-> ]  (1.6,.5-2+0.3) --   (  3.8,.5-2+0.3);
      \node[above] at  (2+.7,0.5-2+0.3) {\small$\mu_i$};

        \filldraw[black]  (4+0,0-2+0.3) circle  ( 0.6pt );
    \filldraw[green]  (4+0.8,0-2+0.3) circle  ( 0.6pt);
    \filldraw[black]  (4+1.2,1-2+0.3) circle  ( 0.6pt);
     \filldraw[red]  (4+0.6,0.5-2+0.3) circle  ( 0.6pt);
      \filldraw[purple]  (4+1.3,0.5-2+0.3) circle  ( 0.6pt);
       \draw[ thick]  (4+0.6,0.5-2+0.3) --   (4+1.3,0.5-2+0.3);
    \draw[  thick]  (4+0,0-2+0.3) --   (4+ 0.8,0-2+0.3); 
    \draw[  thick]  (4+0,0-2+0.3) --   (4+ 1.2,1-2+0.3);
 \draw[ thick ]  (4+0,0-2+0.3) --   (4+1.3,0.5-2+0.3);
 
 \draw  [thick] (4+0.8,0-2+0.3) --  (4+1.3,0.5-2+0.3);
  \draw [thick] (4+1.2,1-2+0.3) --  (4+1.3,0.5-2+0.3);  
\end{tikzpicture}  
    \caption{Configuration as in \Cref{fig:commutation-cases}(II)}
    \label{fig:ii}
\end{figure}
 
(3) The lifts $\overline{\gamma}_i$ and $\overline{\gamma}_j$ are configured as in \Cref{fig:commutation-cases}(III). 
Assume they are positioned as in the left‑most diagram of \Cref{fig:iii}. 
Again one obtains
$\mu_i \mu_j \mu_i(T) = \mu_j \mu_i  (T),$
which is exhibited by the corresponding diagram.

  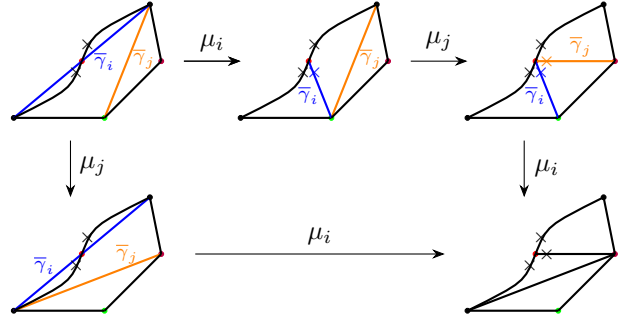
\begin{figure}[h]
    \centering
    \begin{tikzpicture}[scale=1.5 , >=Stealth]
    \filldraw[black] (0,0) circle (0.6pt);
    \filldraw[green] (0.8,0) circle (0.6pt);
    \filldraw[black] (1.2,1) circle (0.6pt);
     \filldraw[red] (0.6,0.5) circle (0.6pt);
      \filldraw[purple] (1.3,0.5) circle (0.6pt);
      \node[orange] at (1.15,0.52) {\tiny$\overline{\gamma}_j$};
     \node[blue] at (0.8,0.5) {\tiny$\overline{\gamma}_i$};
    \draw[ thick,orange] (0.8,0) --  (1.2,1);
    \draw[  thick] (0,0) --  ( 0.8,0); 
    \draw[  thick, blue ] (0,0) --  ( 1.2,1);

      \draw[-> ] (1.5,.5) --  (  2,.5);
      \node[above] at (1.75,0.5) {\small$\mu_i$};
     \draw[-> ] (0.5,-0.2) --  (  0.5,-0.7);
      \node[above] at (0.7,-0.6) {\small$\mu_j$}; 
 \filldraw[black]  (2+0,0) circle  ( 0.6pt);
    \filldraw[green]  (2+0.8,0) circle  ( 0.6pt);
    \filldraw[black]  (2+1.2,1) circle  ( 0.6pt);
     \filldraw[red]  (2+0.6,0.5) circle  ( 0.6pt);
      \filldraw[purple]  (2+1.3,0.5) circle  ( 0.6pt);
      \node[orange] at  (2+1.15,0.52) {\tiny$\overline{\gamma}_j$};
     \node[blue] at  (2+0.6,0.2) {\tiny$\overline{\gamma}_i$};
    \draw[ thick,orange]  (2+0.8,0) --   (2+1.2,1);
    \draw[  thick]  (2+0,0) --   (2+ 0.8,0); 
 \draw[ thick,blue ]  (2+0.8,0) --   (2+0.6,0.5);
 \node[blue]  at (2+0.65,0.4) {\tiny$\times$};
 
 \draw  [thick] (2 +0.8,0) --  (2+1.3,0.5);
  \draw [thick] (2 +1.2,1) --  (2+1.3,0.5);
 \draw[-> ]  (2+1.5,.5) --   (2+  2,.5);
      \node[above] at  (2+1.75,0.5) {\small$\mu_j$}; 
  \filldraw[black]  (4+0,0) circle  ( 0.6pt);
    \filldraw[green]  (4+0.8,0) circle  ( 0.6pt);
    \filldraw[black]  (4+1.2,1) circle  ( 0.6pt);
     \filldraw[red]  (4+0.6,0.5) circle  ( 0.6pt);
      \filldraw[purple]  (4+1.3,0.5) circle  ( 0.6pt);
      \node[orange] at  (4+1,0.63) {\tiny$\overline{\gamma}_j$};
     \node[blue] at  (4+0.6,0.2) {\tiny$\overline{\gamma}_i$};
    \draw[ thick,orange]  (4+0.6,0.5) --   (4+1.3,0.5);
    \node[orange]  at (4+0.7,0.5) {\tiny$\times$};
    
    \draw[  thick]  (4+0,0) --   (4+ 0.8,0); 
 \draw[ thick,blue ]  (4+0.8,0) --   (4+0.6,0.5);
 \node[blue]  at (4+0.65,0.4) {\tiny$\times$};
 \draw  [thick] (4+0.8,0) --  (4+1.3,0.5);
  \draw [thick] (4+1.2,1) --  (4+1.3,0.5);
 \draw[-> ] (4+0.5,-0.2) --  (4+  0.5,-0.7);
      \node[above] at (4+0.7,-0.6) {\small$\mu_i$}; 
 
  \filldraw[black]  (0,0-2+0.3) circle  ( 0.6pt );
    \filldraw[green]  (0.8,0-2+0.3) circle  ( 0.6pt);
    \filldraw[black]  (1.2,1-2+0.3) circle  ( 0.6pt);
     \filldraw[red]  (0.6,0.5-2+0.3) circle  ( 0.6pt);
      \filldraw[purple]  (1.3,0.5-2+0.3) circle  ( 0.6pt);
      \node[orange] at  (1,0.5-2+0.3) {\tiny$\overline{\gamma}_j$};
     \node[blue] at  (0.27,0.7-2) {\tiny$\overline{\gamma}_i$};
    \draw[ thick,orange]  (0.0,0-2+0.3) --   (1.3,0.5-2+0.3);
    \draw[  thick]  (0,0-2+0.3) --   ( 0.8,0-2+0.3); 
    \draw[blue,  thick]  (0,0-2+0.3) --   ( 1.2,1-2+0.3);
 \draw  [thick] (0.8,0-2+0.3) --  (1.3,0.5-2+0.3);
  \draw [thick] (1.2,1-2+0.3) --  (1.3,0.5-2+0.3);
   \draw[-> ]  (1.6,.5-2+0.3) --   (  3.8,.5-2+0.3);
      \node[above] at  (2+.7,0.5-2+0.3) {\small$\mu_i$};

        \filldraw[black]  (4+0,0-2+0.3) circle  ( 0.6pt );
    \filldraw[green]  (4+0.8,0-2+0.3) circle  ( 0.6pt);
    \filldraw[black]  (4+1.2,1-2+0.3) circle  ( 0.6pt);
     \filldraw[red]  (4+0.6,0.5-2+0.3) circle  ( 0.6pt);
      \filldraw[purple]  (4+1.3,0.5-2+0.3) circle  ( 0.6pt);
    \draw[ thick ]  (4+0.0,0-2+0.3) --   (4+1.3,0.5-2+0.3);
    \draw[ thick ]  (4+0.6,0.5-2+0.3) --   (4+1.3,0.5-2+0.3);
    \node   at (4+0.7,0.5-2+0.3) {\tiny$\times$};
    \draw[  thick]  (4+0,0-2+0.3) --   (4+ 0.8,0-2+0.3); 
 \draw  [thick] (4+0.8,0-2+0.3) --  (4+1.3,0.5-2+0.3);
  \draw [thick] (4+1.2,1-2+0.3) --  (4+1.3,0.5-2+0.3);
 \draw  [thick] (0.8,0) -- (1.3,0.5);
  \draw [thick] (1.2,1) -- (1.3,0.5);
 \draw[-> ] (1.5,.5) --  (  2,.5);
 
\draw [thick  ]( 0,0) .. controls
(0.5 ,0.25)  ..( 0.6,0.5);
 \draw [thick  ]( 1.2,1) .. controls
(0.7 ,0.75)  ..(0.6,0.5);
\node  at ( 0.55,0.4) {\tiny$\times$};
\node  at ( 0.65,0.65) {\tiny$\times$};

\draw [thick  ](2+0,0) .. controls
(2+0.5 ,0.25)  ..(2+0.6,0.5);
 \draw [thick  ](2+1.2,1) .. controls
(2+0.7 ,0.75)  ..(2+0.6,0.5);
\node  at (2+0.55,0.4) {\tiny$\times$};
\node  at (2+0.65,0.65) {\tiny$\times$};

\draw [thick  ](4+0,0) .. controls
(4+0.5 ,0.25)  ..(4+0.6,0.5);
 \draw [thick  ](4+1.2,1) .. controls
(4+0.7 ,0.75)  ..(4+0.6,0.5);
\node  at (4+0.55,0.4) {\tiny$\times$};
\node  at (4+0.65,0.65) {\tiny$\times$};
 \draw [thick  ]( 0,0-2+0.3) .. controls
(0.5 ,0.25-2+0.3)  ..( 0.6,0.5-2+0.3);
 \draw [thick  ]( 1.2,1-2+0.3) .. controls
(0.7 ,0.75-2+0.3)  ..(0.6,0.5-2+0.3);
\node  at ( 0.55,0.4-2+0.3) {\tiny$\times$};
\node  at ( 0.65,0.65-2+0.3) {\tiny$\times$};

\draw [thick  ](4+0,0-2+0.3) .. controls
(4+0.5 ,0.25-2+0.3)  ..(4+0.6,0.5-2+0.3);
 \draw [thick  ](4+1.2,1-2+0.3) .. controls
(4+0.7 ,0.75-2+0.3)  ..(4+0.6,0.5-2+0.3);
\node  at (4+0.55,0.4-2+0.3) {\tiny$\times$};
\node  at (4+0.65,0.65-2+0.3) {\tiny$\times$};
\end{tikzpicture}  
    \caption{Configuration as in \Cref{fig:commutation-cases}(III)}
    \label{fig:iii}
\end{figure}
\end{proof}
According to  \cite[Proposition~7.10]{FST},  the exchange graph of cluster-tilting objects in $\CCP$ is  connected.   Thus, we  deduce the following corollary    from  Proposition~\ref{prop:commutation-relations}.
\begin{corollary}\label{cor:5.6}
Any loop in the exchange graph of cluster-tilting objects in $\CCP$ decomposes into squares and pentagons. 
\end{corollary}

At the end of this section, we provide a geometric interpretation of the automorphism group $\Aut\CCP$ of the cluster category $\CCP$.
Recall from \cite{LM,BKL} that the automorphism group of $\CCP$ admits a semidirect product decomposition
  \[
  \Aut\CCP\cong \PSL(2,\ZZ)\ltimes_{\rm conj}(\Aut(\CPone) \ltimes_2 \Pic_0\CPone). 
  \]  
Moreover, by \Cref{tubular left mutations}, the subgroup $\PSL(2,\ZZ)$ of $\Aut\CCP$ is generated by automorphisms $r$ and $\ell$, whose action on indecomposable rigid objects $E_p^x$ is given by
\[
r(E_p^x)=E_{p+1}^y,\qquad
\ell(E_p^x)=E_{\frac{p}{1-p}}^z,
\]
where the labels $y$ and $z$ are specified in Table~\ref{tab:bt-tag-action} (for $-x$ they are replaced by $-y$ and $-z$, respectively).

 Combining \Cref{tagmcg} and \Cref{prop:MCGofS42}, we have
\[
\MCG^\times(S_{0,4})
\cong
\bigl(\PSL(2,\ZZ)\ltimes_{\rm conj}
\langle\iota_1,\iota_2\rangle\bigr)
\ltimes_1\{\pm1\}^{\P},
\]
where $\P$ is the set of punctures   on $S_{0,4}$. The strategy used in the proof of \Cref{autd=mag}, with $\Br_3$ replaced by $\PSL(2,\ZZ)$, gives the following result.
\begin{theorem}\label{autc=mac}
For $\CPone=(\mathbb{P}^1,(0,1,\infty,\frac{1}{2}),(2,2,2,2))$ over $\CC$, there is a group isomorphism
 \[ \phi': \Aut\CCP  \to   \MCG^\times  (S_{0,4}) , \]   such that for any $\xi\in\Aut(\CCP)$ and any indecomposable rigid object $E_p^x$ in $\CCP$,
 \[\phi'(\xi )(X_c^{-1}(E_{p}^x))=X_c^{-1}(\xi (E_{p}^x)).\] 
\end{theorem}
 
\appendix
\section{Even Continued-Fraction Expansions}\label{app:even-cf}
For $p\in\mathbb Q$, we write
\[
[a_1,\ldots,a_m]=a_1+\cfrac{1}{a_2+\cfrac{1}{\ddots+\cfrac{1}{a_m}}}.
\]
We also use the conventions $0=[-1,1]$ and $\infty=[]$.

\begin{lemma}\label{lem:even-cf}
Every nonzero rational number $p$ has a unique expansion $p=[a_1,\ldots,a_{2n}]$ of even length satisfying
\[
\begin{cases}
a_1\in\mathbb Z_{\geq0},\quad a_i\in\mathbb Z_{>0}  \text{ for }i>1, & \text{if }p>0,\\
a_1\in\mathbb Z_{\leq0},\quad a_i\in\mathbb Z_{<0} \text{ for }i>1, & \text{if }p<0.
\end{cases}
\]
\end{lemma}

\begin{proof}
Suppose first that $p>0$. The Euclidean algorithm gives a regular continued fraction $p=[b_1,\ldots,b_l]$, where $b_1\geq0$, $b_i>0$ for $i>1$, and $b_l>1$ if $l>1$. If $l$ is even, there is nothing to prove. If $l$ is odd, use $[b_1,\ldots,b_l]=[b_1,\ldots,b_l-1,1]$. For $l=1$, this reads $[b_1]=[b_1-1,1]$; otherwise $b_l-1>0$. Thus an expansion of the required form always exists.

For uniqueness, consider a positive tail $q=[c_1,\ldots,c_r]$. If $r$ is even, then $c_1<q\leq c_1+1$, so $c_1=\lceil q\rceil-1$. If $r$ is odd, then $c_1\leq q<c_1+1$, so $c_1=\lfloor q\rfloor$; equality $q=c_1$ occurs precisely when $r=1$. Starting with an even-length expansion, these two rules determine the entries successively: after finding $c_1$, the next tail is $1/(q-c_1)$, and its parity is opposite to that of $r$. The process stops exactly when an odd-length tail is an integer. Hence both the entries and the point at which the expansion ends are uniquely determined.

For $p<0$, apply the positive case to $-p$ and negate each entry; both existence and uniqueness follow.
\end{proof}
\paragraph{\textbf{Farey sums.}}
Let $\frac{r}{s}$ and $\frac{u}{v}$ be reduced fractions with $s,v\geq0$,
allowing the signed infinite fractions $\pm\infty=\pm\frac{1}{0}$. If
$|rv-us|=1$, their \emph{Farey sum} is
\[
\frac{r}{s}\oplus\frac{u}{v}\coloneqq\frac{r+u}{s+v}.
\]

\begin{lemma}[{\cite[Section 2]{MGO}, \cite[Lemma/Definition 2.3]{FQ}}]\label{lem:farey-decomposition}
Every $p\in\mathbb Q\setminus\{0\}$ has a unique Farey decomposition
\[
p=p_-\oplus p_+,
\qquad p_-<p<p_+,
\]
such that
\[
\begin{cases}
 p_-,p_+\in\mathbb Q_{\geq0}\cup\{\infty\},&\text{if }p>0,\\
 p_-,p_+\in\mathbb Q_{\leq0}\cup\{-\infty\},&\text{if }p<0.
\end{cases}
\]
More explicitly, if $p=[a_1,\ldots,a_{2m}]>0$, then
\[
p_-=
\begin{cases}
[a_1,\ldots,a_{2m-1}-1,1],&\text{if }a_{2m-1}\ne1,\\
[a_1,\ldots,a_{2m-2}+1],&\text{if }a_{2m-1}=1\text{ and }m>1,\\
[0,1],&\text{if }a_{2m-1}=1\text{ and }m=1,
\end{cases}
\]
and
\[
p_+=
\begin{cases}
[a_1,\ldots,a_{2m-1},a_{2m}-1],&\text{if }a_{2m}\ne1,\\
[a_1,\ldots,a_{2m-2}],&\text{if }a_{2m}=1\text{ and }m>1,\\
[],&\text{if }a_{2m}=1\text{ and }m=1.
\end{cases}
\]
For $p<0$, apply these formulas to $-p=q_-\oplus q_+$ and set
$p_-=-q_+$ and $p_+=-q_-$.
\end{lemma}
\section{Braid-Twist Word Identities}\label{app:bt-words}
In this appendix we record the braid-word calculations used in
\Cref{prop:braid-twist-slope-action}.
For $p=[a_1,\ldots,a_{2m}]\in\QQ$, set
\[
W_p=B_\infty^{a_1}B_0^{-a_2}\cdots B_\infty^{a_{2m-1}}B_0^{-a_{2m}},
\]
and set $W_\infty=1$. We also use the conventions $0=[-1,1]$ and
$\infty=[]$.

\begin{lemma}\label{lem:bt-word-identities}
The braid twists $B_0$ and $B_\infty$ satisfy
\begin{equation}\label{eq:braid-relation-B0-Binfty}
B_0B_\infty B_0=B_\infty B_0B_\infty
\end{equation}
and
\begin{equation}\label{eq:central-shift}
(B_0B_\infty)^3=(B_\infty B_0)^3=[1].
\end{equation}
Moreover, for every $k\in\ZZ$,
\begin{align}
B_0B_\infty^kB_0^{-1}&=B_\infty^{-1}B_0^kB_\infty,\label{eq:conj-1}\\
B_0^{-1}B_\infty^kB_0&=B_\infty B_0^kB_\infty^{-1},\label{eq:conj-2}\\
B_\infty B_0B_\infty^k&=B_0^kB_\infty B_0,\label{eq:conj-3}\\
B_\infty^kB_0B_\infty&=B_0B_\infty B_0^k.\label{eq:conj-4}
\end{align}
\end{lemma}

\begin{proof}
The relation \eqref{eq:braid-relation-B0-Binfty} is the usual braid relation for the two adjacent braid twists associated with the arcs of slopes $0$ and $\infty$. In the local smoothing model, the two sides also give the same lifted tangent path, so no extra grading shift occurs.

For \eqref{eq:central-shift}, first forget the grading. The images of $B_0$ and
$B_\infty$ in the $\PSL(2,\ZZ)$-factor are $t_2$ and $t_1$, respectively, and
$(t_2t_1)^3=1$ in $\PSL(2,\ZZ)$. Hence $(B_0B_\infty)^3$ acts trivially on the
underlying surface. By \Cref{lem:1.13}, the kernel of the forgetful map from the
graded mapping class group to the ungraded one is generated by the grading shift
$[1]$. Thus
\[
(B_0B_\infty)^3=[m]
\]
for some $m\in\ZZ$.

It remains to determine $m$. Locally near an endpoint, the composition $(B_0B_\infty)^3$ fixes the underlying arc and moves each end once clockwise around the endpoint. Thus the lifted tangent line makes one clockwise loop in the projectivized tangent fiber. In an angle coordinate $\theta\in\mathbb{R}/\pi\mathbb{Z}$, this means $\theta\mapsto\theta-\pi$. By the convention in \Cref{def:gcurve}, this is exactly the shift $[1]$. Hence $m=1$, and $(B_0B_\infty)^3=[1]$. The equality $(B_\infty B_0)^3=[1]$ follows from the braid relation.

The identities \eqref{eq:conj-1}--\eqref{eq:conj-4} are consequences of
\eqref{eq:braid-relation-B0-Binfty}. For example, from
$B_0B_\infty B_0=B_\infty B_0B_\infty$ we get
$B_0B_\infty B_0^{-1}=B_\infty^{-1}B_0B_\infty$. Taking powers gives
\[
B_0B_\infty^kB_0^{-1}
=(B_0B_\infty B_0^{-1})^k
=(B_\infty^{-1}B_0B_\infty)^k
=B_\infty^{-1}B_0^kB_\infty
\]
for $k\ge 0$, and the case $k<0$ follows by taking inverses. The other three
formulas follow similarly, or by interchanging $B_0$ and $B_\infty$ and taking
inverses.
\end{proof}

For the next lemma, put
\[
\nu(p,n)=
\begin{cases}
n,& p\ge 0\text{ or }p=\infty,\\
n-1,& p<0.
\end{cases}
\]
Thus the definition of the graded arcs can be written as
\[
\widetilde{\alpha}_p^\epsilon[n]
=W_p\bigl(\widetilde{\alpha}_\infty^\epsilon[\nu(p,n)]\bigr).
\]

\begin{lemma}\label{lem:bt-normal-forms}
For every $p\in\QQi$ and $n\in\ZZ$, the following identities hold for
$\epsilon=-$:
\begin{equation}\label{eq:B-infty-normal-form}
B_\infty W_p\bigl(\widetilde{\alpha}_\infty^-[\nu(p,n)]\bigr)
=W_{p+1}\bigl(\widetilde{\alpha}_\infty^-[\nu(p+1,n)]\bigr),
\end{equation}
and, writing $q=\frac{p}{1-p}$ with $q=-1$ when $p=\infty$,
\begin{equation}\label{eq:B-zero-normal-form}
B_0W_p\bigl(\widetilde{\alpha}_\infty^-[\nu(p,n)]\bigr)
=
\begin{cases}
W_q\bigl(\widetilde{\alpha}_\infty^-[\nu(q,n)]\bigr),& p\le 1,\\[4pt]
W_q\bigl(\widetilde{\alpha}_\infty^-[\nu(q,n+1)]\bigr),& 1<p\le\infty.
\end{cases}
\end{equation}
The same identities hold with $\widetilde{\alpha}_\infty^-$ replaced by
$\widetilde{\alpha}_\infty^+$.
\end{lemma}

\begin{proof}
We prove the identities for $\widetilde{\alpha}_\infty^-$. The proof for
$\widetilde{\alpha}_\infty^+$ is the same, since the braid-word calculations do
not depend on the sign of the initial arc.

We first prove \eqref{eq:B-infty-normal-form}. If $p=\infty$, then $B_\infty$
is the braid twist along $\widetilde{\alpha}_\infty^-$ itself, so it preserves
$\widetilde{\alpha}_\infty^-[n]$. If $p=[a_1,\ldots,a_{2m}]$ and $p,p+1$ obey
the same sign convention, then
\[
p+1=[a_1+1,a_2,\ldots,a_{2m}]
\]
is the even continued fraction of $p+1$, and hence
\[
B_\infty W_p
=B_\infty^{a_1+1}B_0^{-a_2}\cdots B_\infty^{a_{2m-1}}B_0^{-a_{2m}}
=W_{p+1}.
\]
This proves the claim for $p>0$ and for $p<-1$.

It remains to check the crossing from negative to nonnegative slopes. For
$p=-1$ we have $p=[0,-1]$ and $p+1=0=[-1,1]$. Hence
$W_p=B_0$ and $W_{p+1}=B_\infty^{-1}B_0^{-1}$. Using
\eqref{eq:central-shift} and the braid relation,
\[
\begin{aligned}
B_\infty B_0\bigl(\widetilde{\alpha}_\infty^-[n-1]\bigr)
&=B_\infty B_0B_\infty\bigl(\widetilde{\alpha}_\infty^-[n-1]\bigr)\\
&=B_0B_\infty B_0\bigl(\widetilde{\alpha}_\infty^-[n-1]\bigr)\\
&=B_0B_\infty B_0B_\infty\bigl(\widetilde{\alpha}_\infty^-[n-1]\bigr)\\
&=B_\infty^{-1}B_0^{-1}(B_0B_\infty)^3
\bigl(\widetilde{\alpha}_\infty^-[n-1]\bigr)\\
&=B_\infty^{-1}B_0^{-1}\bigl(\widetilde{\alpha}_\infty^-[n]\bigr),
\end{aligned}
\]
which is exactly
$W_0(\widetilde{\alpha}_\infty^-[\nu(0,n)])$.

Now assume $-1<p<0$. Then the even continued fraction of $p$ starts with
$a_1=0$ and has negative remaining entries. The base case
$p=[0,a_2]$ is obtained as follows. If $a_2=-1$, this is the preceding
calculation. If $a_2=-2$, then repeated use of
\eqref{eq:braid-relation-B0-Binfty} and \eqref{eq:central-shift} gives
\[
B_\infty B_0^2\bigl(\widetilde{\alpha}_\infty^-[n-1]\bigr)
=B_0^{-2}\bigl(\widetilde{\alpha}_\infty^-[n]\bigr)
=W_{[0,2]}\bigl(\widetilde{\alpha}_\infty^-[n]\bigr).
\]
Since $[0,2]=[0,-2]+1$, this is the desired arc.
If $a_2\le -3$, then
we first use the braid relations to obtain the right hand word. We claim that
\[
B_\infty B_0^{-a_2}\bigl(\widetilde{\alpha}_\infty^-[n-1]\bigr)
=B_0^{-1}B_\infty^{-a_2-2}B_0^{-1}
\bigl(\widetilde{\alpha}_\infty^-[n]\bigr).
\]
Equivalently, it is enough to prove
\[
B_0B_\infty^{a_2+2}B_0B_\infty B_0^{-a_2}
\bigl(\widetilde{\alpha}_\infty^-[n-1]\bigr)
=\widetilde{\alpha}_\infty^-[n].
\]
We use the following elementary claim. For every $a\in\ZZ$,
\[
B_\infty^aB_\infty B_0B_\infty B_0^{-a}
=B_\infty B_0B_\infty B_0^aB_0^{-a}
=B_\infty B_0B_\infty,
\]
and
\[
B_0^{-a}B_\infty B_0B_\infty B_\infty^a
=B_\infty B_0B_\infty B_\infty^{-a}B_\infty^a
=B_\infty B_0B_\infty.
\]
These are the special cases $k=a+1$ of \eqref{eq:conj-4} and \eqref{eq:conj-3}, respectively.
Hence, since $B_\infty$ fixes
$\widetilde{\alpha}_\infty^-[n-1]$, the left hand side of the preceding display
is
\[
\begin{aligned}
&B_0B_\infty^{a_2+2}B_0B_\infty B_0^{-a_2}
\bigl(\widetilde{\alpha}_\infty^-[n-1]\bigr)\\
&\quad =
B_0B_\infty^{a_2+1}(B_\infty B_0B_\infty)B_0^{-a_2}B_\infty
\bigl(\widetilde{\alpha}_\infty^-[n-1]\bigr)\\
&\quad =
B_0B_\infty (B_\infty B_0B_\infty)B_\infty
\bigl(\widetilde{\alpha}_\infty^-[n-1]\bigr)\\
&\quad =
B_0B_\infty B_0B_\infty B_0B_\infty
\bigl(\widetilde{\alpha}_\infty^-[n-1]\bigr)\\
&\quad =
(B_0B_\infty)^3
\bigl(\widetilde{\alpha}_\infty^-[n-1]\bigr)
=\widetilde{\alpha}_\infty^-[n].
\end{aligned}
\]
Thus the claim is proved, and the resulting word is
$W_{\beta}(\widetilde{\alpha}_\infty^-[n])$ with
\[
\beta=[0,1,-a_2-2,1].
\]
Now we identify this slope:
\[
[0,1,-a_2-2,1]
=\frac{1}{1+\frac{1}{-a_2-1}}
=\frac{a_2+1}{a_2}
=[0,a_2]+1
=p+1.
\]
For longer continued fractions, 
 let $p=[0,a_2,\ldots,a_{2m}]$ with
$m>1$. In the case $a_2\le -2$ and $a_{2m}\le -2$, the braid calculation gives
\[
\begin{aligned}
B_\infty W_p\bigl(\widetilde{\alpha}_\infty^-[n-1]\bigr)
&=B_0^{-1}B_\infty^{-a_2-1}B_0^{a_3}B_\infty^{-a_4}\cdots
  B_0^{a_{2m-1}}B_\infty^{-a_{2m}-1}B_0^{-1}
  \bigl(\widetilde{\alpha}_\infty^-[n]\bigr).
\end{aligned}
\]
Here the initial block is obtained by the same claim as above with
$a=-a_2+1$, and the terminal block is checked in the same way after reading the
word from the right. Thus this word is $W_\beta$ for
\[
\beta=[0,1,-a_2-1,-a_3,\ldots,-a_{2m-1},-a_{2m}-1,1].
\]
It remains only to identify the slope represented by this word. Put
$u=[a_2,\ldots,a_{2m}]$. Then $p=1/u$, and the tail
\[
[-a_2-1,-a_3,\ldots,-a_{2m-1},-a_{2m}-1,1]
\]
is equal to $-u-1$ by expanding the continued fraction from the right. Hence
\[
\beta=[0,1,-u-1]
=\frac{-u-1}{-u}
=1+\frac{1}{u}
=p+1.
\]

If $a_2=-1$ or $a_{2m-1}=-1$, 
the same braid computation yields the corresponding word with the zero block removed:
\[
\begin{array}{c|c}
\text{condition} & \beta \text{ obtained from the braid word} \\ \hline
a_2=-1,\ a_{2m}\le -2,\ m>2
& [0,1-a_3,-a_4,\ldots,-a_{2m-1},-a_{2m}-1,1]\\[2pt]
a_2\le -2,\ a_{2m}=-1,\ m>2
& [0,1,-a_2-1,-a_3,\ldots,-a_{2m-2},1-a_{2m-1}]\\[2pt]
a_2=-1,\ a_{2m}=-1,\ m>2
& [0,1-a_3,-a_4,\ldots,-a_{2m-2},1-a_{2m-1}]\\[2pt]

a_2=-1,\ a_4\le -2,\ m=2
& [0,1-a_3,-a_4-1,1]\\[2pt]
a_2\le -2,\ a_4=-1,\ m=2
& [0,1,-a_2-1,1-a_3]\\[2pt]
a_2=-1,\ a_4=-1,\ m=2
& [0,2-a_3].
\end{array}
\]
In each row the same continued-fraction computation gives $\beta=p+1$. This proves
\eqref{eq:B-infty-normal-form}.

We now prove \eqref{eq:B-zero-normal-form}. The boundary values are immediate.
For $p=\infty$, one has $q=-1=[0,-1]$ and
\[
B_0\bigl(\widetilde{\alpha}_\infty^-[n]\bigr)
=W_{-1}\bigl(\widetilde{\alpha}_\infty^-[n]\bigr)
=\widetilde{\alpha}_{-1}^-[n+1].
\]
For $p=1=[0,1]$, $W_p=B_0^{-1}$ and $q=\infty$, so
\[
B_0W_p\bigl(\widetilde{\alpha}_\infty^-[n]\bigr)
=\widetilde{\alpha}_\infty^-[n].
\]
For $p=0=[-1,1]$, the arc
$W_0(\widetilde{\alpha}_\infty^-[n])$ is $\widetilde{\alpha}_0^-[n]$, and the
braid twist $B_0=B_{\widetilde{\alpha}_0^-}$ preserves this graded arc.

It remains to treat the non-boundary cases. Let
$p=[a_1,\ldots,a_{2m}]$ and put $q=\frac{p}{1-p}$. We first determine the
right hand word by braid relations, and only afterwards identify the slope.

Assume first that $p<0$. Then the base arc is
$\widetilde{\alpha}_\infty^-[n-1]$. Moving the initial $B_0$ across $W_p$ gives
a word $W_\beta$ with
\[
\beta=
\begin{cases}
[0,-1,a_1,a_2,\ldots,a_{2m}],& a_1<0,\\[2pt]
[0,a_2-1,a_3,\ldots,a_{2m}],& a_1=0.
\end{cases}
\]
These continued fractions satisfy $\beta=p/(1-p)=q$. Thus the braid calculation
gives
\[
B_0W_p\bigl(\widetilde{\alpha}_\infty^-[n-1]\bigr)
=W_q\bigl(\widetilde{\alpha}_\infty^-[n-1]\bigr).
\]
Both $\widetilde{\alpha}_p^-[n]$ and $\widetilde{\alpha}_q^-[n]$ are defined
using the base shift $[n-1]$, so no extra grading shift appears.

Now assume $0<p< 1$. Then $a_1=0$ and the base arc is
$\widetilde{\alpha}_\infty^-[n]$. The braid calculation gives $W_\beta$, where
\[
\beta=
\begin{cases}
[0,a_2-1,a_3,\ldots,a_{2m}],& a_2>1,\\[2pt]
[a_3,a_4,\ldots,a_{2m}],& a_2=1.
\end{cases}
\]
Again one checks directly from the continued fractions that
$\beta=p/(1-p)=q$. Hence
\[
B_0W_p\bigl(\widetilde{\alpha}_\infty^-[n]\bigr)
=W_q\bigl(\widetilde{\alpha}_\infty^-[n]\bigr),
\]
and no extra grading shift appears.

Finally assume 
$p=[a_1,\ldots,a_{2m}]>1$. Moving the initial $B_0$ through $W_p$ by
\eqref{eq:conj-1}--\eqref{eq:conj-4}, and using that $B_\infty$ fixes
$\widetilde{\alpha}_\infty^-[n]$, gives
\[
\begin{aligned}
B_0W_p\bigl(\widetilde{\alpha}_\infty^-[n]\bigr)
&=B_\infty^{-1}B_0^{a_1-1}B_\infty^{-a_2}B_0^{a_3}\cdots
  B_0^{a_{2m-1}}B_\infty^{1-a_{2m}}B_0^{1}
  \bigl(\widetilde{\alpha}_\infty^-[n]\bigr),
\end{aligned}
\]
where the first and last factors are the boundary terms produced in this
calculation, with the evident deletion of such a term when the corresponding
exponent is zero. Thus, for $m>1$, the resulting word is $W_\beta$, where
$\beta$ is read from the word as follows:
\[
\begin{array}{c|c}
\text{condition} & \beta \text{ obtained from the braid word}\\ \hline
a_1\ge 2,\ a_{2m}\ge 2
& [-1,1-a_1,-a_2,\ldots,-a_{2m-1},1-a_{2m},-1]\\[2pt]
a_1\ge 2,\ a_{2m}=1
& [-1,1-a_1,-a_2,\ldots,-a_{2m-1}-1]\\[2pt]
a_1=1,\ a_{2m}\ge 2
& [-a_2-1,-a_3,\ldots,-a_{2m-1},1-a_{2m},-1]\\[2pt]
a_1=1,\ a_{2m}=1
& [-a_2-1,-a_3,\ldots,-a_{2m-2},-a_{2m-1}-1].
\end{array}
\]
For $m=1$, the same deletion gives
\[
\beta=
\begin{cases}
[-a_2,-1],& a_1=1,\\
[-1,-a_1],& a_1\ge2,\ a_2=1,\\
[-1,1-a_1,1-a_2,-1],& a_1\ge2,\ a_2\ge2.
\end{cases}
\]
In all these cases the displayed continued fraction is the normalized expansion of
$q=p/(1-p)$.  Indeed, the tail after the initial $-1$ is
$1-p$. Hence the word above is $W_q$. Since $q<0$, the normalized lift
$\widetilde{\alpha}_q^-[n+1]$ is defined using the base shift $[n]$, which gives
exactly the extra grading shift. Together with the cases above, this proves \eqref{eq:B-zero-normal-form}.
\end{proof}
\begin{proposition}\label{prop:farey-braid-recursion}
Let $p\in\mathbb Q\setminus\{0\}$, and write $p=p_-\oplus p_+$ as in
\Cref{lem:farey-decomposition}. For the signed endpoint $-\infty$, use the
convention
\[
\widetilde{\alpha}_{-\infty}^\epsilon[n]
\coloneqq\widetilde{\alpha}_\infty^\epsilon[n-1].
\]
Then, for $\epsilon\in\{+,-\}$ and $n\in\ZZ$, we have
\[
\widetilde{\alpha}_p^\epsilon[n]
=B_{\widetilde{\alpha}_{p_+}^{-}}
\bigl(\widetilde{\alpha}_{p_-}^\epsilon[n]\bigr).
\]
\end{proposition}

\begin{proof}
By the definition of $\widetilde{\alpha}_r^-[n]$, \eqref{eq:bt}, and the
invariance of a braid twist under grading shifts, we have
\[
B_{\widetilde{\alpha}_r^-[n]}
=W_rB_\infty W_r^{-1}
\]
for all $r\in\QQi$ and $n\in\ZZ$.

Let $p=[a_1,\ldots,a_{2m}]>0$. Assume first that $m>1$. By
\Cref{lem:farey-decomposition},
\[
W_{p_-}=
\begin{cases}
B_\infty^{a_1}B_0^{-a_2}\cdots B_0^{-a_{2m-2}}
B_\infty^{a_{2m-1}-1}B_0^{-1},&a_{2m-1}\ne1,\\
B_\infty^{a_1}B_0^{-a_2}\cdots
B_\infty^{a_{2m-3}}B_0^{-a_{2m-2}-1},&a_{2m-1}=1,
\end{cases}
\]
and
\[
W_{p_+}=
\begin{cases}
B_\infty^{a_1}B_0^{-a_2}\cdots B_0^{-a_{2m-2}}
B_\infty^{a_{2m-1}}B_0^{1-a_{2m}},&a_{2m}\ne1,\\
B_\infty^{a_1}B_0^{-a_2}\cdots B_0^{-a_{2m-2}},&a_{2m}=1.
\end{cases}
\]
In both cases, $W_{p_-}$ has the form in the first line, allowing
$a_{2m-1}-1=0$. Hence, if
$a_{2m}\ne1$, \eqref{eq:conj-1}--\eqref{eq:conj-4} give
\[
\begin{aligned}
B_{\widetilde{\alpha}_{p_+}^{-}}
\bigl(\widetilde{\alpha}_{p_-}^\epsilon[n]\bigr)
&=W_{p_+}B_\infty W_{p_+}^{-1}W_{p_-}
  \bigl(\widetilde{\alpha}_\infty^\epsilon[n]\bigr)\\
&=\Bigl(B_\infty^{a_1}B_0^{-a_2}\cdots B_0^{-a_{2m-2}}
  B_\infty^{a_{2m-1}}B_0^{1-a_{2m}}\Bigr)B_\infty\\
&\quad\cdot\Bigl(B_0^{a_{2m}-1}B_\infty^{-a_{2m-1}}
  B_0^{a_{2m-2}}\cdots B_0^{a_2}B_\infty^{-a_1}\Bigr)\\
&\quad\cdot\Bigl(B_\infty^{a_1}B_0^{-a_2}\cdots B_0^{-a_{2m-2}}
  B_\infty^{a_{2m-1}-1}B_0^{-1}\Bigr)
  \bigl(\widetilde{\alpha}_\infty^\epsilon[n]\bigr)\\
&=B_\infty^{a_1}B_0^{-a_2}\cdots B_0^{-a_{2m-2}}
  B_\infty^{a_{2m-1}}B_0^{1-a_{2m}}B_\infty
  B_0^{a_{2m}-1}B_\infty^{-1}B_0^{-1}
  \bigl(\widetilde{\alpha}_\infty^\epsilon[n]\bigr)\\
&=W_pB_0B_\infty B_0^{a_{2m}-1}B_\infty^{-1}B_0^{-1}
  \bigl(\widetilde{\alpha}_\infty^\epsilon[n]\bigr)\\
&=W_pB_\infty^{a_{2m}-1}
  \bigl(\widetilde{\alpha}_\infty^\epsilon[n]\bigr)
=\widetilde{\alpha}_p^\epsilon[n].
\end{aligned}
\]
If $a_{2m}=1$, the same calculation gives
\[
B_{\widetilde{\alpha}_{p_+}^{-}}
\bigl(\widetilde{\alpha}_{p_-}^\epsilon[n]\bigr)
=W_p\bigl(\widetilde{\alpha}_\infty^\epsilon[n]\bigr)
=\widetilde{\alpha}_p^\epsilon[n].
\]

Now let $m=1$, so $p=[a_1,a_2]$. The two parent words are
\[
W_{p_-}=
\begin{cases}
B_\infty^{a_1-1}B_0^{-1},&a_1\ne1,\\
B_0^{-1},&a_1=1,
\end{cases}
\qquad
W_{p_+}=
\begin{cases}
B_\infty^{a_1}B_0^{1-a_2},&a_2\ne1,\\
1,&a_2=1.
\end{cases}
\]
Here $B_0^{-1}=W_{[0,1]}$ and $1=W_{[]}$. For $a_2\ne1$, the calculation above
reduces to
\[
\begin{aligned}
B_{\widetilde{\alpha}_{p_+}^{-}}
\bigl(\widetilde{\alpha}_{p_-}^\epsilon[n]\bigr)
&=B_\infty^{a_1}B_0^{1-a_2}B_\infty
  B_0^{a_2-1}B_\infty^{-a_1}B_\infty^{a_1-1}B_0^{-1}
  \bigl(\widetilde{\alpha}_\infty^\epsilon[n]\bigr)\\
&=W_pB_0B_\infty B_0^{a_2-1}B_\infty^{-1}B_0^{-1}
  \bigl(\widetilde{\alpha}_\infty^\epsilon[n]\bigr)\\
&=W_pB_\infty^{a_2-1}
  \bigl(\widetilde{\alpha}_\infty^\epsilon[n]\bigr)
 =\widetilde{\alpha}_p^\epsilon[n].
\end{aligned}
\]
For $a_2=1$, it becomes
\[
B_\infty W_{p_-}\bigl(\widetilde{\alpha}_\infty^\epsilon[n]\bigr)
=B_\infty^{a_1}B_0^{-1}
\bigl(\widetilde{\alpha}_\infty^\epsilon[n]\bigr)
=\widetilde{\alpha}_p^\epsilon[n].
\]

For $p<0$, write $-p=q_-\oplus q_+$. By
\Cref{lem:farey-decomposition}, we have $p_-=-q_+$ and $p_+=-q_-$. The
corresponding braid words are obtained from those above by changing the signs
of all exponents and interchanging the two parents. Substitution followed by
the same braid-word simplification gives
\[
B_{\widetilde{\alpha}_{p_+}^{-}}
\bigl(\widetilde{\alpha}_{p_-}^\epsilon[n]\bigr)
=W_p\bigl(\widetilde{\alpha}_\infty^\epsilon[n-1]\bigr)
=\widetilde{\alpha}_p^\epsilon[n].
\]
The cases where one of the Farey parents is $0$ or $-\infty$ are covered by
the conventions $0=[-1,1]$ and
$\widetilde{\alpha}_{-\infty}^\epsilon[n]
=\widetilde{\alpha}_\infty^\epsilon[n-1]$.
\end{proof}
\begin{remark}\label{rem:farey-smoothing-index}
By the smoothing description of a braid twist, the grading of
$\widetilde{\alpha}_p^\epsilon[n]$ along the branch inherited from
$\widetilde{\alpha}_{p_-}^\epsilon[n]$ is unchanged. Hence, at their common
endpoint $z$,
\[
\gind_z\bigl(\widetilde{\alpha}_{p_-}^\epsilon[n],
\widetilde{\alpha}_p^\epsilon[n]\bigr)=0.
\]
\end{remark}

\end{document}